\documentclass[a4paper,11pt]{amsart}

\usepackage[top=30truemm,bottom=25truemm,left=35truemm,right=35truemm]{geometry}

\usepackage{amssymb}
\usepackage{url} 

\usepackage{amscd} 
\usepackage[all]{xy} 
\usepackage{graphics}
\usepackage{multienum}

\newtheorem{thm}{Theorem}[section]
\newtheorem{lem}[thm]{Lemma}
\newtheorem{lem-dfn}[thm]{Lemma-Definition}
\newtheorem{prop}[thm]{Proposition}
\newtheorem{cor}[thm]{Corollary}

\theoremstyle{definition}
\newtheorem{defn}[thm]{Definition}

\newtheorem{ex}[thm]{Example}

\newtheorem{quest}[thm]{Question}

\newtheorem*{acknowledgement}{Acknowledgement}
\newtheorem*{conventions}{Conventions}

\theoremstyle{remark}

\newtheorem{rem}[thm]{Remark}

\begin{document}


\title[$F$-rationality of 2-dimensional graded rings]{$F$-rationality of two-dimensional graded rings  with a rational singularity}

\author{Kohsuke Shibata}
\address{School of Engineering, 
Tokyo Denki University, Adachi-ku, Tokyo 120-8551, Japan.}
\email{shibata.kohsuke@mail.dendai.ac.jp}

\begin{abstract}
It is known that a two-dimensional $F$-rational ring has a rational singularity.
However a two-dimensional ring with a rational singularity is not $F$-rational in general. In this paper, we  investigate $F$-rationality of a two-dimensional graded ring with a rational singularity in terms of the multiplicity. Moreover, we determine when a two-dimensional graded ring with a rational singularity and a small multiplicity is $F$-rational.
\end{abstract}

\maketitle

\section{Introduction}
It is well known by now that there is an interesting connection between $F$-singularities and singularities in birational geometry.
In \cite{HW02}, Hara and Watanabe showed that  a strongly $F$-regular ring has log terminal singularities and an $F$-pure ring has log canonical singularities.
In \cite{S}, Smith showed that an $F$-rational ring has pseudo-rational singularities.
Therefore a two-dimensional excellent $F$-rational ring has a rational singularity. 
However two-dimensional excellent  ring with a rational singularity is not $F$-rational in general.
Thus a natural question is when rings with rational singularities are  $F$-rational.

In \cite{HW}, Hara and Watanabe investigated $F$-rationality of a two-dimensional graded ring with a rational singularity in terms of Pinkham-Demazure construction and gave the necessary and sufficient condition for $F$-rationality of a two-dimensional graded ring with a rational singularity.

In April 2020, Kei-ichi Watanabe asked the author the following question.
\begin{quest}\label{Watanabe's question}
Let $D=\sum_{i=1}^r \frac{c_i}{d_i} P_i$ be an ample $\mathbb Q$-divisor on $\mathbb P_k^1$, where $c_i\in\mathbb Z$,  $d_i\in \mathbb N$  and $P_i$ are distinct points of $\mathbb P_k^1$.
Let $R=\bigoplus_{n\ge 0} H^0(\mathbb P_k^1,\mathcal O_{\mathbb P_k^1}([nD]))t^n$.
Assume that $R$ has a rational singularity and $d_i>p$ for all $i$.
Then  is $R$  $F$-rational?
\end{quest}
In this paper, we give an affirmative answer to this question.
\begin{thm}
Let $D=\sum_{i=1}^r \frac{c_i}{d_i} P_i$ be an ample $\mathbb Q$-divisor on $\mathbb P_k^1$, where $c_i\in\mathbb Z$,  $d_i\in \mathbb N$  and $P_i$ are distinct points of $\mathbb P_k^1$.
Let $R=\bigoplus_{n\ge 0} H^0(\mathbb P_k^1,\mathcal O_{\mathbb P_k^1}([nD]))t^n$.
Assume that $R$ has a rational singularity and $p$ does not divide  any $d_i$.
Then   $R$ is $F$-rational.
In particular, Question \ref{Watanabe's question} is affirmative.
\end{thm}

In \cite{H}, Hara proved that a two-dimensional log terminal singularity is strongly $F$-regular if the characteristic is larger than $5$.
This implies that a two-dimensional rational double point is $F$-rational if  the characteristic is larger than $5$.
In this paper, we investigate $F$-rationality of a two-dimensional graded ring with a rational singularity  in terms of the multiplicity. 
We prove the following theorem.
\begin{thm}\label{Main theorem}
Let $m\in \mathbb N$.
 There exists a positive integer $p(m)$ such that 
   $R$ is $F$-rational for any two-dimensional  graded ring $R$ with a rational singularity, $e(R)=m$ and $R_0=k$, an algebraically closed field of  characteristic  $p\ge p(m)$.
\end{thm}

Moreover,  we can determine  $p(3)$ and $p(4)$ in the above theorem. 
\begin{thm}\label{Main theorem2}
Let $R$ be a two-dimensional graded ring with a rational singularity.
\begin{enumerate}
\item If $e(R)= 3$ and $p\ge 7$, then  $R$ is $F$-rational.

\item If $e(R)= 4$ and $p\ge 11$,  then  $R$ is $F$-rational.

\end{enumerate}
Furthermore, 
these inequalities are best possible.
\end{thm}

The paper is organized as follows.
In Section 2, we review  definitions and some facts on $F$-rational rings, rational singularities  and Pinkham-Demazure construction.
In Section 3, we investigate $F$-rationality of a two-dimensional graded ring with a rational singularity in terms of Pinkham-Demazure construction and give  an affirmative answer to Question \ref{Watanabe's question}.
In Section 4, we prove Theorem \ref{Main theorem}.
In Section 5, we classify two-dimensional graded rings with a rational singularity and multiplicity $3$ and $4$ in terms of Pinkham-Demazure construction. 
In Section 6, we determine $p(3)$ and $p(4)$ in Theorem \ref{Main theorem}.

\begin{acknowledgement}
The  author would like to thank   Kei-ichi Watanabe for the discussion and many suggestions. 
The author are grateful to Alessandro De Stefani and Ilya Smirnov for insightful conversations and comments
on a rough draft of this paper.
The  author is partially supported by JSPS KAKENHI Grant Number JP20J00132. 

\end{acknowledgement}

\begin{conventions}
Throughout this paper, $p$ is a  prime number and $k$ is an algebraically closed field of characteristic $p$.
We assume that a ring is  essentially of finite type  over $k$.
By a graded ring, we mean a ring $R=\oplus_{n\ge 0} R_n$, which is finitely generated over the subring $R_0=k$. 
\end{conventions}

\section{Preliminaries}   
In this section we introduce  definitions and some facts on $F$-rational rings, rational singularities  and Pinkham-Demazure construction.

\subsection{$F$-rational rings and rational singularities}
In this subsection  we introduce the definitions of $F$-rational rings and rational singularities.

\begin{defn}
Let $R$ be a ring and $I$  an ideal of $R$. The tight closure $I^*$ of $I$ is defined by  $x\in I^*$
if and only if there  exists $c\in R^\circ$ such that $cx^{p^e}\in I^{[p^e]}$ for $e\gg 0$, where $R^\circ$ is the set of elements of $R$ which are not in any minimal prime ideal and $I^{[p^e]}$ is the ideal generated by the $p^e$-th powers of the elements of $I$.
We say that $I$ is tightly closed if $I^*=I$.
\end{defn}

\begin{defn}
A local ring $(R,\mathfrak m)$  is $F$-rational if every parameter ideal is tightly closed.
An arbitrary ring $R$  is  $F$-rational if $R_\mathfrak m$ is $F$-rational for every maximal ideal $\mathfrak m$.

\end{defn}

\begin{defn}
A local ring $(R,\mathfrak m)$  is $F$-injective if $R$-module homomorphism
\[
H^i_\mathfrak m(F):H^i_\mathfrak m(R) \to H^i_\mathfrak m(R)
\]
is injective for all $i$.
An arbitrary ring $R$  is  $F$-injective if $R_\mathfrak m$ is $F$-injective for every maximal ideal $\mathfrak m$.
\end{defn}

\begin{defn}
Let $R$ be a  two-dimensional  normal ring,  and let $f:Y \to X:=\mathrm{Spec} (R)$ be a resolution of singularities. 
The ring $R$ is said to be (or have) a rational singularity if $R^1f_*\mathcal O_Y=0$.
\end{defn}

\begin{rem}
It is known that there exists a resolution of singularity even in positive characteristic for any two-dimensional excellent normal ring (see e.g. \cite[Theorem 2.1]{L}).
\end{rem}

\subsection{Hirzebruch-Jung Continued fraction}
In this subsection, we introduce the definition and basic properties of the  Hirzebruch-Jung continued fraction.

\begin{defn}
Let $a_1,a_2,\dots,a_n$ be real numbers.
We denote by $[[a_1,\dots,a_n]]$ the  Hirzebruch-Jung continued fraction:
\[ [[a_1,\dots,a_n]]:=a_1-\cfrac{1}{a_2-\cfrac{1}{a_3-\cfrac{1}{\cdots-\cfrac{1}{a_n}}}}. \]
\end{defn}

\begin{rem}\label{HJ fraction uniqueness}
For any rational number $r\in \mathbb Q$ with $r>1$,
there exist unique natural numbers $a_1,\dots,a_n\in \mathbb N$ such $r=[[a_1,\dots,a_n]]$ and $a_i\ge 2$ for all $i$ (see \cite[Lemma 7.4.14]{I}).
\end{rem}

\begin{lem}\label{[[a,b]]=[[a,[[b]]]]}
Let $m,n$ be positive integers with $m< n$, and let $a_1,\dots,a_n$ be real numbers.
Then we have 
\[ [[a_1,\dots,a_n]]= [[a_1,\dots,a_m,[[a_{m+1}\dots,a_n]]]]\]
\end{lem}

\begin{proof}
This follows directly from the definition.
\end{proof}

\begin{lem}\label{[[a]]>1}
Let $a_1,\dots,a_n$ be positive integers with   $\min\{a_1,\dots,a_n\}\ge 2$.
Then
$$[[a_1,\dots,a_n]]>1.$$
\end{lem}

\begin{proof}
We prove this by induction on $n$.
If $n=1$, then $[[a_1]]=a_1>1.$
If $n>1$, then  
\[[[a_1,\dots,a_n]]=[[a_1,[[a_2,\dots,a_n]]]]>a_1-1\ge1\]
by Lemma \ref{[[a,b]]=[[a,[[b]]]]}.
\end{proof}

\begin{lem}\label{[[a]]<[[b]]}
Let $a_1,\dots,a_l$, $b_1,\dots,b_m$, $c_1,\dots,c_n$ be positive integers with  $b_1<c_1$ and $\min\{a_1,\dots,a_l, b_1,\dots,b_m, c_1,\dots,c_n\}\ge 2$. Then
\begin{enumerate}
\item
$[[a_1,\dots,a_l,b_1,\dots,b_m]]< [[a_1,\dots,a_l]].$

\item
$[[a_1,\dots,a_l,b_1,\dots,b_m]]< [[a_1,\dots,a_l,c_1,\dots,c_n]].$

\end{enumerate}

\end{lem}

\begin{proof}
(1) By Lemma \ref{[[a,b]]=[[a,[[b]]]]} and Lemma \ref{[[a]]>1}, we have
\[
[[a_1,\dots,a_l,b_1,\dots,b_m]]< [[a_1,\dots,a_l,N]]=[[a_1,\dots,a_l-\frac{1}{N}]]<[[a_1,\dots,a_l]]
\]
 for a positive integer $N>[[b_1,\dots,b_m]]>1$.

(2) By Lemma \ref{[[a,b]]=[[a,[[b]]]]},
it is enough to prove that 
\[[[b_1,\dots,b_m]]< [[c_1,\dots,c_n]].\]
By Lemma \ref{[[a,b]]=[[a,[[b]]]]},  Lemma \ref{[[a]]>1} and Lemma \ref{[[a]]<[[b]]}.(1), we have
\[
[[b_1,\dots,b_m]]\le b_1\le c_1-1<[[c_1,[[c_2,\dots,c_n]]]]=[[c_1,\dots,c_n]].
\]
\end{proof}

We denote by $(2)^l$ the sequence obtained by repeating $l$ times the number $2$.
\begin{ex}\label{[[2]]}  Let $l$ be a positive integer. Then we have
\[  [[(2)^l]] = \dfrac{l+1}{l}.\]
Indeed, if $[[(2)^n]]=\frac{n+1}{n}$ holds for $n\in\mathbb N$,
we have $$[[(2)^{n+1}]]=[[2,(2)^n]]=[[2,\frac{n+1}{n}]]=2-\frac{n}{n+1}=\frac{n+2}{n+1}.$$ 
\end{ex}

\begin{ex}
$2=[[2]]< [[3,(2)^l]]=3-\frac{l}{l+1}$ for any $l\in \mathbb Z_{\ge 0}$.
\end{ex}

\subsection{Pinkham-Demazure construction}
In this subsection  we introduce the construction of a two-dimensional normal graded ring using a $\mathbb Q$-divisor on a smooth curve.
By a $\mathbb Q$-divisor on a variety $X$, we mean a $\mathbb Q$-linear combination of codimension-one irreducible subvarieties of $X$.
If $D=\sum a_iD_i$, where $a_i\in\mathbb Q$ and $D_i$ are distinct irreducible subvarieties,
we set $[D]=\sum [a_i]D_i$, where $[a]$ denotes the greatest integer less than or equal to $a$.  

In \cite{P}, Pinkham proved the following result.
In \cite{D}, Demazure generalized this result in higher dimensional case. 

\begin{thm}[{\cite[3.5]{D}},{\cite[Theorem 5.1]{P}}]\label{DP const}
Let $R$ be a two-dimensional normal graded ring over $R_0=k$.
Then there exists an ample $\mathbb Q$-divisor $D$ on $C=\mathrm{Proj}(R)$ such that 
\[ R\cong R(C,D):=\bigoplus_{n\ge 0} H^0(C,\mathcal O_C([nD]))t^n. \]
\end{thm}
We call this representation Pinkham-Demazure construction.

\begin{rem}\label{ample divisor on curve}
\begin{enumerate}
\item
A divisor $D$ on a smooth curve is ample if and only if $\mathrm{deg}D>0$ (see \cite[I\hspace{-.1em}V.Corollary 3.3]{Har}).
\item
Let  $D_1,D_2$ be ample $\mathbb Q$-divisors on a smooth curve $C$.
If $D_1-D_2$ is a principal divisor on $C$, 
then  $R(C,D_1)\cong R(C,D_2)$.
Indeed, let $f$ be the rational function on $C$ with $\mathrm{div}(f)=D_1-D_2$, and let $g$ be a rational function on $C$ with $\mathrm{div}(g)+nD_1\ge 0$.
Then  
$\mathrm{div}(f^ng)+nD_2=\mathrm{div}(g)+nD_1\ge 0.$
Therefore we have an isomorphism $R(C,D_1)\cong R(C,D_2)$ defined by $gt^n \mapsto f^ngt^n$.

\item If $C=\mathbb P_k^1$, we can put $D=sP_0 - \sum_{i=1}^r a_i P_i$ in Theorem \ref{DP const}, where $s\in \mathbb N$ and $a_i\in \mathbb Q_{>0}$ with  $0<a_i<1$, and $P_i$ are distinct points of $\mathbb P_k^1$.
Indeed, since $P$ is linearly equivalent to $Q$ for any points $P,Q$ of $\mathbb P_k^1$ by \cite[I\hspace{-.1em}I.Proposition 6.4]{Har}, this remark holds by the above remark.
\end{enumerate}
\end{rem}

A resolution of a singularity is said to be good if the exceptional divisor has normal crossing and each irreducible components of the exceptional divisor is smooth.
A resolution of a surface  singularity is called a minimal good resolution if the resolution is the smallest resolution of good resolutions, i.e. every good resolution factors through a minimal good resolution.
An exceptional divisor $E$ of the minimal good resolution of a two-dimensional singularity is said to be a central curve if $E$ has positive genus or $E$ meets at least three other exceptional divisors of the minimal good resolution.
The dual graph of the minimal good resolution is said to be star-shaped if the dual graph has at most one central curve.

In \cite{P}, Pinkham determined the exceptional set of the minimal good resolution of $\mathrm{Spec} (R(\mathbb P_k^1,D))$.
\begin{thm}[{\cite[Section 2 and Theorem 5.1]{P}}]\label{dual graph of graded ring}
Let $D=sP_0 - \sum_{i=1}^r \frac{c_i}{d_i} P_i$ be an ample $\mathbb Q$-divisor on $\mathbb P_k^1$, where $s,c_i,d_i\in \mathbb N$ with  $0<c_i<d_i$, and $P_i$ are distinct points of $\mathbb P_k^1$.
Let $b_{i1},\dots,b_{im_i}$ be positive integers with $\frac{d_i}{c_i}=[[b_{i1},\dots,b_{im_i}]]$.
Then 
the exceptional set of the minimal good resolution of  $\mathrm{Spec} (R(\mathbb P_k^1,D))$ consists of 
\begin{enumerate}
\item unique central curve $E_0\cong \mathbb P_k^1$ with $E_0^2 = -s$ and 
\item $r$ branches of $\mathbb P_k^1$'s $E_{i1}-E_{i2}-\cdots - E_{i m_i}$ corresponding to $P_i$ with $E_{ij}^2 = - b_{ij}$ and $E_0 E_{i1}=1$.
\end{enumerate}
Thus the dual graph is star-shaped as follows:
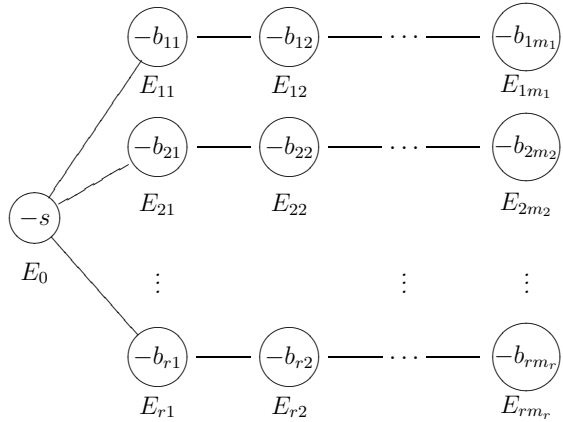
\begin{eqnarray*}
\scalebox{0.8}{
\xymatrix@C=25pt@R=10pt{
& *++[o][F]{-b_{11}} \ar@{-}[r] \ar@{}[d]|(0.9) *{E_{11}} 
 &*++[o][F]{-b_{12}} \ar@{-}[r] \ar@{}[d]|(0.9) *{E_{12}} 
 & \cdots \ar@{-}[r]  
 &*++[o][F]{-b_{1m_1}} \ar@{}[d]|(0.9) *{E_{1m_1}} \\
&&&&\\
& *++[o][F]{-b_{21}} \ar@{-}[r] \ar@{}[d]|(0.8) *{E_{21}} 
 &*++[o][F]{-b_{22}} \ar@{-}[r] \ar@{}[d]|(0.8) *{E_{22}} 
 & \cdots \ar@{-}[r]  
 &*++[o][F]{-b_{2m_2}} \ar@{}[d]|(0.8) *{E_{2m_2}} \\
*++[o][F]{-s} \ar@{}[d]|(0.8) *{E_0} \ar@{-}[uuur] \ar@{-}[ur] \ar@{-}[ddr]
&&&&\\
&\vdots&&\vdots&\vdots\\
&*++[o][F]{-b_{r1}} \ar@{-}[r] \ar@{}[d]|(0.9) *{E_{r1}} 
 &*++[o][F]{-b_{r2}} \ar@{-}[r] \ar@{}[d]|(0.9) *{E_{r2}} 
 & \cdots \ar@{-}[r]  
 &*++[o][F]{-b_{rm_r}} \ar@{}[d]|(0.9) *{E_{rm_r}} \\
&&&&\\
}
}
\end{eqnarray*}

\end{thm}

\begin{defn}
An irreducible curve $E$ on a smooth surface is called a $(-i)$-curve if $E\cong \mathbb P_k^1$ with $E^2=-i$.
\end{defn}

\begin{defn}
Let $(R,\mathfrak m)$ be a $d$-dimensional normal graded ring.
The $a$-invariant $a(R)$ of $R$ is defined by 
\[
a(R):=\max\left\{n\in \mathbb Z \mid [H_\mathfrak m^d(R)]_n\not =0\right\},
\]
where $[H_\mathfrak m^d(R)]_n$ denotes the $n$-th graded piece of the highest local cohomology module of $H_\mathfrak m^d(R)$.
\end{defn}

Theorem \ref{graded rational singularity} is a very useful characterization of  a rational singularity.

\begin{thm}[{\cite[Corollary 5.8]{P}},{\cite[Korollary 3.10]{F}},{\cite[Theorem 2.2]{W}} ]\label{graded rational singularity} 
Let $C$ be a smooth curve, $D$  an ample $\mathbb Q$-divisor  on $C$ and  $R= R(C,D)$.
Then the following conditions are equivalent.
\begin{enumerate}
\item 
 $R$ has a rational singularity. 

\item  $C=\mathbb P_k^1$ and $\mathrm{deg}[nD]\ge -1$ for any positive integer $n$.

\item   $a(R)<0$.

\end{enumerate}

\end{thm}

\begin{lem}\label{non-rational s<r-1}
Let $D=sP_0 - \sum_{i=1}^r \frac{c_i}{d_i} P_i$ be an ample $\mathbb Q$-divisor on $\mathbb P_k^1$, where $s,c_i,d_i\in\mathbb N$ with  $0<c_i<d_i$, and $P_i$ are distinct points of $\mathbb P_k^1$.
If $s+2\le r$, then $R(\mathbb P_k^1,D)$ does not have a rational singularity.
\end{lem}

\begin{proof}
Since $\mathrm{deg}[D]=s-r\le -2$, $R(\mathbb P_k^1,D)$ does not have a rational singularity by Theorem \ref{graded rational singularity}.
\end{proof}

For graded rings, $F$-rationality is characterized in terms of $F$-injectivity in \cite{FW} and \cite{HH}.

\begin{thm}[{\cite[Theorem 2.8]{FW}}, {\cite[Theorem 7.12]{HH}}]\label{HH F-rational}
Let $R$ be a two-dimensional normal graded ring.
Then $R$ is $F$-rational if and only if $R$ is $F$-injective and $a(R)<0$.
\end{thm}

\subsection{Fundamental cycle}
In this subsection, we introduce the definition and useful properties of the fundamental cycle.

Let $(X,x)$ be a two-dimensional normal  singularity, and let  $f: Y \to X$ be a resolution of singularity.
We denote by $\mathrm{Exc}(f)$ the exceptional set of $f$. 
We call the minimum element of the set 
\[
\left\{
Z\in \mathrm{Div}(Y)\setminus\{0\}  \ 
\middle|  \begin{array}{l} 
\mathrm{Supp}(Z)\subset \mathrm{Exc}(f)\ \
\mbox{and }\  ZE\le 0\\ \mbox{for any prime exceptional }
\mbox{divisor } E \mbox{ of } f 
\end{array}
\right\}.
\]
 the fundamental cycle of $f$.
For an exceptional divisor $D$ on $Y$,
we denote by 
$p_a(D):=\frac{D^2+K_YD}{2}+1$ and call  it the virtual genus of $D$.

\begin{prop}\label{Z^2} 
Let $R$ be a two-dimensional local ring with a rational singularity,  $f:X\to \mathrm{Spec} (R)$  the minimal good resolution and $E_1,\dots,E_r$   the prime exceptional divisors of $f$.
Let $Z = \sum_{i=1}^r  n_{i} E_{i}$ be the fundamental cycle of $f$.
Then 

\[ e(R) = - Z^2 = \sum_{i=1}^r  n_{i} (- E_{i}^2 - 2) +2.\]
\end{prop}

\begin{proof}
We can compute $Z$ by a computation sequence of cycles 
\[0 < Z_1< \ldots < Z_s =Z\]
defined by $Z_1 = F_1$ (we can take any prime exceptional divisor of $f$) and $Z_i = Z_{i-1} + F_i$, where $F_i$ is any prime exceptional  divisor $f$ with 
$Z_{i-1} F_i > 0$ (see for example \cite[Proposition 7.2.4]{I}).
 Then we have 
\[ p_a( Z_{i}) = p_a( Z_{i-1} ) + p_a(F_i) + Z_{i-1} F_i - 1\ge p_a( Z_{i-1})\]
\[\mbox{and} \ \  Z_{i}^2 = Z_{i-1}^2 + 2 Z_{i-1} F_i + F_i^2 \]
since    $p_a(F_i)\ge 0$ by \cite[Proposition 7.2.8]{I}.
Since  $p_a(Z)=0$ by \cite[ Proposition 7.3.1]{I}, we have  $p_a(F_i)= 0$ and  $Z_{i-1} F_i =1$ for all $i$. Hence we have 
\[
e(R) = - Z^2 = \sum_{i=1}^r  n_{i} (- E_{i}^2 - 2) +2.
\]
 by \cite[Proposition 7.3.5]{I}.
\end{proof}

\begin{rem}\label{rem dual graph HJ fraction}
\begin{enumerate}
\item
Note that the dual graph of the minimal good resolution of a two-dimensional rational singularity contains no $(-1)$-curves since the minimal resolution of a two-dimensional rational singularity is the minimal good resolution.
Indeed, we have $p_a(F_i)= 0$ and $Z_{i-1} F_i =1$ in the above proof, which implies that 
  all irreducible components of the exceptional set have to be smooth rational curves, pairwise
intersecting transversally in at most one point (see \cite[Proposition 7.2.8.(ii)]{I}).

\item
Let $D=sP_0 - \sum_{i=1}^r \frac{c_i}{d_i} P_i$ be an ample $\mathbb Q$-divisor on $\mathbb P_k^1$, where $s,c_i,d_i\in \mathbb N$ with  $0<c_i<d_i$, and $P_i$ are distinct points of $\mathbb P_k^1$.
Suppose that $R(\mathbb P_k^1,D)$ has a rational singularity.
If we obtain the dual graph of the minimal good resolution of $\mathrm{Spec} (R(\mathbb P_k^1,D))$, we can determine the  Hirzebruch-Jung continued fraction of $\frac{d_i}{c_i}$.
Indeed, since this dual graph contains no $(-1)$-curves, as stated in (1), and Hirzebruch-Jung continued fractions are uniquely determined by natural numbers greater than 1 by Remark \ref{HJ fraction uniqueness}, we can determine the  Hirzebruch-Jung continued fraction of $\frac{d_i}{c_i}$ by Theorem \ref{dual graph of graded ring}.
\end{enumerate}
\end{rem}

Once we have  the coefficient of the central curve of the fundamental cycle of the minimal good resolution of $\mathrm{Spec} (R(\mathbb P_k^1,D))$,
the fundamental cycle  can be computed by the following 
formula.

For a divisor $D=\sum_{i=1}^ra_iE_i$, where $E_i$ is a prime divisor, we denote by $\mathrm{Coeff}_{E_i}D$ the  coefficient $a_i$.
For a real number $a$, we denote by $\left\lceil   a \right\rceil$ the smallest integer greater than or equal to $a$.  

\begin{lem}\label{formula}
Let $D=sP_0 - \sum_{i=1}^r \frac{c_i}{d_i} P_i$ be an ample $\mathbb Q$-divisor on $\mathbb P_k^1$, where $s,c_i,d_i\in\mathbb N$ with  $0<c_i<d_i$, and $P_i$ are distinct points of $\mathbb P_k^1$.
Let $f:X\to \mathrm{Spec} (R(\mathbb P_k^1,D))$ be the minimal good resolution.
Let $F$ be a non-zero effective divisor on $X$ with $\mathrm{Supp}(F)\subset \mathrm{Exc}(f)$ and $n_0$   the coefficient of the central curve
$E_0$ on $F$.
Let $E_{i1}-E_{i2}-\cdots - E_{im_i}$ be the branch of $\mathbb P_k^1$'s corresponding to $P_i$ 
such that 
\[ \frac{d_i}{c_i} = [[ b_{i1}, b_{i2}, \ldots , b_{im_i}]],\]
 $E_{ij}^2 = - b_{ij}$ and $E_0E_{i1}=1$.  Define $e_{i1},\ldots , e_{im_i}\in \mathbb Q$ by 
\[ e_{ij}  = [[ b_{ij}, b_{i,j+1}, \ldots , b_{im_i}]].\]
We assume that  the coefficient  $n_{ij}$ of $E_{ij}$ on $F$ is given inductively,
\[ n_{i1} = \left\lceil   \frac{n_0}{e_{i1}} \right\rceil = \left\lceil   \frac{n_0 c_i}{d_i} \right\rceil, \ldots , n_{i,j+1}= \left\lceil   \frac{n_{ij} }{e_{i,j+1}} \right\rceil,\ldots , n_{im_i}= \left\lceil   \frac{n_{i,m_i-1} }{e_{im_i}} \right\rceil. \] 
Then $F$ is the smallest element of the set 
\[
\left\{
G\in \mathrm{Div}(X)\setminus\{0\}  \ 
\middle|  \begin{array}{l} 
\mathrm{Supp}(G)\subset \mathrm{Exc}(f),\ \mathrm{Coeff}_{E_0}G=n_0\\
\mbox{and } GE\le 0 \mbox{ for any prime exceptional}\\
\mbox{divisor } E \mbox{ of } f \mbox{ with } E\not= E_0
\end{array}
\right\}.
\]
Moreover if $n_0$ is equal to the coefficient of the central curve of  the fundamental cycle of $f$, then $F$ is the the fundamental cycle.
\end{lem}

\begin{proof}
The statement follows from \cite[Lemma 1.1]{KN}.
\end{proof}


\begin{cor}\label{n_0=min} 
Let $D=sP_0 - \sum_{i=1}^r \frac{c_i}{d_i} P_i$ be an ample $\mathbb Q$-divisor on $\mathbb P_k^1$, where $s,c_i,d_i\in\mathbb N$ with  $0<c_i<d_i$, and $P_i$ are distinct points of $\mathbb P_k^1$.
Let $Z$ be the fundamental cycle of the minimal good resolution of  $\mathrm{Spec} (R(\mathbb P_k^1,D))$, and let $E_0$ be the central curve of the minimal good resolution. Then
\[\mathrm{Coeff}_{E_0}Z=\min \{ n \in\mathbb N \;|\; \mathrm{deg} [ n D] \ge 0 \}. \]
In particular, if $s+1\le r$, then 
\[\mathrm{Coeff}_{E_0}(Z)\ge 2.\] 
\end{cor}

\begin{proof} 
Let $f:X\to \mathrm{Spec} (R(\mathbb P_k^1,D))$ be the minimal good resolution.
For $l\in\mathbb N$, let  $F_l$ be the smallest element of the set 
\[
\left\{
G\in \mathrm{Div}(X)\setminus\{0\}  \ 
\middle|  \begin{array}{l} 
\mathrm{Supp}(G)\subset \mathrm{Exc}(f),\ \mathrm{Coeff}_{E_0}G=l\\
\mbox{and } GE\le 0 \mbox{ for any prime exceptional}\\
\mbox{divisor } E \mbox{ of } f \mbox{ with } E\not= E_0
\end{array}
\right\}.
\]
Then we  have 
\[
F_l E_0 =-ls+\sum_{i=1}^r \left\lceil   \frac{l c_i}{d_i} \right\rceil =- \mathrm{deg} [lD]
\]
 for $l\in \mathbb N$ by Lemma \ref{formula}.
Let $n_0$ be the coefficient of $E_0$ in $Z$.
Then  $F_{n_0}=Z$ by Lemma \ref{formula}.
Therefore \[n_0 = \min \{ n \in\mathbb N  \;|\; \mathrm{deg} [ n D] \ge 0 \}. \]

If $s+1\le r$, then $\mathrm{deg}[D]\le -1$.
Therefore we have  $\mathrm{Coeff}_{E_0}(Z)\ge 2$.
\end{proof}

\section{$F$-rationality of $R(\mathbb P_k^1,D)$}

 In this section, we  investigate $F$-rationality of a two-dimensional graded ring with a rational singularity in terms of Pinkham-Demazure construction and give  an affirmative answer to Question \ref{Watanabe's question}. 

The following criterion for $F$-rationality is given in \cite{HW}.

\begin{thm}[{\cite[Theorem 2.9]{HW}}] \label{criterion}
Let $D$ be an ample $\mathbb Q$-divisor on $\mathbb P_k^1$ and $R = R(\mathbb P_k^1, D)$.
Let $B_n = -p [-nD] + [-pnD]$ for a positive integer $n$, and let $(B_n)_{\mathrm{red}}$ be the reduced divisor with the same support as $B_n$.
Assume that $R$ has a rational singularity. 
Then $R$ is $F$-rational if and only if for every positive integer $n$, we have 
\[ \mathrm{deg} [ -pnD ] + \mathrm{deg} (B_n)_{\mathrm{red}} \le  1.\]   
\end{thm}

\begin{rem}\label{rem not in Z}
Let $D = \sum_{i=1}^r a_i P_i$ with $a_i\in \mathbb Q$.
Then 
\[ \mathrm{deg} (B_n)_{\mathrm{red}}  \le \sharp \{ i\;| \; na_i \not\in \mathbb Z\}.\] 
In general, $\mathrm{deg} (B_n)_{\mathrm{red}}  \neq \sharp \{ i\;| \; na_i \not\in \mathbb Z\}.$
For example, if $p=2$ and $D=\frac{2}{3}P$, then $\mathrm{deg} (B_1)_{\mathrm{red}}=0$ and $\sharp \{ i\;| \; a_i \not\in \mathbb Z\}=1$.
\end{rem}

\begin{prop}\label{f-rational deg 1}
Let $D=sP_0 - \sum_{i=1}^r \frac{c_i}{d_i} P_i$ be an ample $\mathbb Q$-divisor on $\mathbb P_k^1$, where $s,c_i,d_i\in\mathbb N$ with  $0<c_i<d_i$, and $P_i$ are distinct points of $\mathbb P_k^1$.
\begin{enumerate}
\item  If $s\ge r$, then $R(\mathbb P_k^1,D)$ is $F$-rational.
\item  If $R(\mathbb P_k^1,D)$ has a rational singularity and is not $F$-rational, then $s+1=r$. 
\item  If $R(\mathbb P_k^1,D)$ has a rational singularity, $\mathrm{deg} D \ge 1$ and   $p\ge r-1$, then  $R(\mathbb P_k^1,D)$ is $F$-rational. 
\end{enumerate}
\end{prop}

\begin{proof}
Let $B_n = -p [-nD] + [-pnD]$ for a positive integer $n$.\\
$(1)$ Since $\mathrm{deg} [ nD ] \ge 0$ for any positive integer $n$, $R(\mathbb P_k^1,D)$ has a rational singularity by Theorem \ref{graded rational singularity}.
Note that $\mathrm{deg} [ -pnD ] \le -s$ and $\mathrm{deg} (B_n)_{\mathrm{red}}\le r$ for any positive integer $n$.
Therefore $R(\mathbb P_k^1,D)$ is $F$-rational by Theorem \ref{criterion}.\\
$(2)$ This statement follows from Lemma \ref{non-rational s<r-1} and Proposition \ref{f-rational deg 1}.(1).\\
$(3)$ Since $\mathrm{deg} [ -pnD ] \le \mathrm{deg} ( -pnD )\le -pn$ and $\mathrm{deg} (B_n)_{\mathrm{red}}\le r$ for any positive integer $n$, $R(\mathbb P_k^1,D)$ is $F$-rational by Theorem \ref{criterion}.
\end{proof}

\begin{ex}  Let $D= 2P_0 - \frac{1}{3} (P_1+P_2+P_3)$, where $P_i$ are distinct points of $\mathbb P_k^1$.
Then $R(\mathbb P_k^1,D)$ is $F$-rational for all  $p$.
Indeed, since $\mathrm{deg} [ nD ] \ge -1$ for any positive integer $n$, $R(\mathbb P_k^1,D)$ has a rational singularity by Theorem \ref{graded rational singularity}.
Therefore $R(\mathbb P_k^1,D)$ is $F$-rational by Proposition \ref{f-rational deg 1}.(3).
\end{ex}

Watanabe asked the following question.
\begin{quest}\label{Watanabe's question 1}
Let $D=\sum_{i=1}^r \frac{c_i}{d_i} P_i$ be an ample $\mathbb Q$-divisor on $\mathbb P_k^1$, where $c_i\in\mathbb Z$,  $d_i\in \mathbb N$  and $P_i$ are distinct points of $\mathbb P_k^1$.
Let $R=R(\mathbb P_k^1,D)$.
Assume that $R$ has a rational singularity and $d_i>p$ for all $i$.
Then  is $R$  $F$-rational?
\end{quest}

\begin{thm}\label{$F$-rationality when p not | d_i}
Let $D=\sum_{i=1}^r \frac{c_i}{d_i} P_i$ be an ample $\mathbb Q$-divisor on $\mathbb P_k^1$, where $c_i\in\mathbb Z$,  $d_i\in \mathbb N$  and $P_i$ are distinct points of $\mathbb P_k^1$.
Let $R=R(\mathbb P_k^1,D)$.
Assume that $R$ has a rational singularity and $p$ does not divide any $d_i$.
Then   $R$ is $F$-rational.
In particular, Question \ref{Watanabe's question 1} is affirmative.
\end{thm}

\begin{proof}
We assume that $R$ is not $F$-rational.
Then there exists a positive integer $m$ such that 
\[
\mathrm{deg}[-pmD]+\mathrm{deg}(B_m)_{\mathrm{red}}\ge 2,
\]
where $B_m = -p [-mD] + [-pmD]$ by Theorem \ref{criterion}.
Since $R$ has a rational singularity,  we have for  every positive integer $n$,
$$\mathrm{deg}[nD]\ge -1$$
by Theorem \ref{graded rational singularity}.
Let $l=\sharp\{i\in \mathbb N \mid  \frac{mc_i}{d_i}\not\in \mathbb Z \}$.
Then we have 
 $$\left\lceil  \frac{pmc_j}{d_j}\right\rceil-\left[ \frac{pmc_j}{d_j}\right]=1$$ for $j\in \{i\in \mathbb N\ |\ \frac{mc_i}{d_i}\not\in \mathbb N \}$
and $\mathrm{deg}(B_m)_{\mathrm{red}}\le l$ (see Remark \ref{rem not in Z}).
We have
\begin{align*}
\mathrm{deg}[-pmD]&= \sum_{i=1}^r\left[ -\frac{pmc_i}{d_i}\right]
= -\sum_{i=1}^r\left\lceil  \frac{pmc_i}{d_i}\right\rceil
=-l-\sum_{i=1}^r\left[ \frac{pmc_i}{d_i}\right]\\
&=-l-\mathrm{deg}[pmD]\le -l+1.
\end{align*}
Hence we have 
\[
2\le \mathrm{deg}[-pmD]+\mathrm{deg}(B_m)_{\mathrm{red}}\le -l+1+l=1,\]
 which is a contradiction.
Therefore $R$ is  $F$-rational.

\end{proof}

\section{$F$-rationality of two-dimensional graded rings  with a rational singularity}

In this section we show that for a positive integer $m$, any two-dimensional graded ring with multiplicity $m$ and a rational singularity is $F$-rational if the characteristic of the base field is sufficiently large depending on $m$.

\begin{defn}\label{T frac}
Let $a$ be a rational number with $a>1$. 
Let   $a = [[ b_{1}, \ldots , b_{m}]]$ be the Hirzebruch-Jung continued fraction of  $a$ and $n=\sharp\{ i\, |\, b_i= 2\}$. 
Let $\{j_1,j_2,\dots,j_{m-n}\}$ be the set of numbers such that $b_{j_l}\neq 2$ and $j_l<j_{l+1}$.
We define 
\[T(a)=
\begin{cases}(b_{j_1},b_{j_2},\dots,b_{j_{m-n}})\in \mathbb N^{m-n} & \text{if}\ m \not = n\\
T(a)=\emptyset & \text{if}\ m  = n.
\end{cases}
\]
\end{defn}

For an ample $\mathbb Q$-divisor $D=sP_0-\sum_{i=1}^r \frac{c_i}{d_i} P_i$  on $\mathbb P_k^1$, Theorem \ref{dual graph of graded ring} implies that the Hirzebruch-Jung continued fraction of  $\frac{d_i}{c_i}$ is determined by the exceptional curves in the branch corresponding to $P_i$ of the minimal good resolution of  $\mathrm{Spec} (R(\mathbb P_k^1,D))$.
Therefore, in the proof of Theorem \ref{$F$-rationality p(m)}, we will use $T(\frac{d_i}{c_i} )$  to control the exceptional curves in the branch corresponding to $P_i$ in the dual graph.

\begin{ex}
$T([[2,3,2,4,2,2,5]])=(3,4,5)$.
\end{ex}

\begin{lem}\label{lem T(a)}
For any positive integer $l$, let $c^{(l)},d^{(l)}\in\mathbb N$  with  $0< c^{(l)}< d^{(l)}$.
We assume that for any $l$, $T(\frac{d^{(l)}}{c^{(l)}})$ is constant 
and $T(\frac{d^{(l)}}{c^{(l)}})=(f_1,f_2,\cdots, f_m)\in \mathbb N^m$.
Suppose that $f_j\ge 3$ for any $1\le j\le m$.
\begin{enumerate}
\item
There is a subsequence of $\left\{\frac{d^{(l)}}{c^{(l)}}\right\}_{l\in \mathbb N}$ which is constant or strictly decreasing.

\item
If the sequence of $\left\{\frac{d^{(l)}}{c^{(l)}}\right\}_{l\in \mathbb N}$ is strictly decreasing, then $\displaystyle \lim_{l\to \infty}\frac{d^{(l)}}{c^{(l)}}$ is a rational number greater than or equal to 1.
\end{enumerate}
\end{lem}

\begin{proof}
Since $T(\frac{d^{(l)}}{c^{(l)}})$ is the sequence obtained by removing 2 from the sequence of numbers representing the Hirzebruch-Jung continued fraction of $\frac{d^{(l)}}{c^{(l)}}$,
we can express the Hirzebruch-Jung continued fraction of $\frac{d^{(l)}}{c^{(l)}}$   as
\[
\frac{d^{(l)}}{c^{(l)}}=[[(2)^{e_0^{(l)}},f_{1},(2)^{e_1^{(l)}},f_2,\dots,f_{m-1},(2)^{e_{m-1}^{(l)}},f_m,(2)^{e_m^{(l)}}]]
\]
for some $e_0^{(l)},e_1^{(l)},\dots,e_m^{(l)}\in \mathbb N$.

First, we prove $(1)$.
Suppose that there is not a subsequence of $\left\{\frac{d^{(l)}}{c^{(l)}}\right\}_{l\in \mathbb N}$ which is constant.
Then there exists $i$ such that $\limsup_{l \to \infty} e_i^{(l)}=\infty$.
Let $g=\min_{0\le i\le m}\{i \mid \limsup_{l \to \infty} e_i^{(l)}=\infty\}.$
Then by taking a subsequence of $\left\{\frac{d^{(l)}}{c^{(l)}}\right\}_{l\in \mathbb N}$,
we may assume that $\{e_i^{(l)}\}_l$ is constant for any fixed $i\in\{ 0,\dots, g-1\}$ and $\{e_g^{(l)}\}_l$ is strictly increasing.
Therefore by Lemma \ref{[[a]]<[[b]]}, there is a subsequence of $\left\{\frac{d^{(l)}}{c^{(l)}}\right\}_{l\in \mathbb N}$ which is strictly decreasing.

(2) 
Since $f_j\ge 3$ and $\left\{\frac{d^{(l)}}{c^{(l)}}\right\}_{l\in \mathbb N}$ is strictly decreasing, Lemma \ref{[[a]]<[[b]]} implies that $(e_0^{(l)},\cdots,e_m^{(l)})<_{\mathrm{lex}} (e_0^{(l+1)},\cdots,e_m^{(l+1)})$ for any $l\in \mathbb N$, where $<_{\mathrm{lex}}$ is the lexicographic order.
Let $g=\min_{0\le i\le m}\{i \mid \limsup_{l \to \infty} e_i^{(l)}=\infty\}.$
Then for any sufficiently large number $l$, $e_i^{(l)}$ is constant for any fixed $i\in\{ 0,\dots, g-1\}$ and $e_g^{(l)}<e_g^{(l+1)}$.
Let $e_i=\lim_{l \to \infty} e_i^{(l)}$ for  $0\le i \le g-1$.
By Lemma \ref{[[a]]<[[b]]}, we have
\[
[[(2)^{e_0^{(l)}},f_{1},\dots,(2)^{e_{g-1}^{(l)}},f_g,(2)^{e_g^{(l)}},2]]\le
\frac{d^{(l)}}{c^{(l)}}\le[[(2)^{e_0^{(l)}},f_{1},\dots,(2)^{e_{g-1}^{(l)}},f_g,(2)^{e_g^{(l)}}]].
\]
Note that $[[(2)^e]] = \frac{e+1}{e}$ for any $e\in \mathbb N$ by Example \ref{[[2]]}.
Hence we have
\[
\lim_{l\to \infty}\frac{d^{(l)}}{c^{(l)}}=[[(2)^{e_0},f_{1},\dots,(2)^{e_{g-1}},f_g,1]]=[[(2)^{e_0},f_{1},\dots,(2)^{e_{g-1}},f_g-1]],
\]
which implies that $\lim_{l\to \infty}\frac{d^{(l)}}{c^{(l)}}$ is a rational number greater than or equal to 1 by Lemma \ref{[[a]]>1}.
\end{proof}

\begin{lem}\label{rationality lim}
For any positive integer $l$, let $D_l=sP_0-\sum_{i=1}^r \frac{c^{(l)}_i}{d^{(l)}_i} P_i$ be an ample $\mathbb Q$-divisor on $\mathbb P_k^1$, where $s,c^{(l)}_i,d^{(l)}_i\in\mathbb N$  with  $0< c^{(l)}_i< d^{(l)}_i$,  and $P_i$ are distinct points of $\mathbb P_k^1$.
We assume that $\frac{c^{(l)}_i}{d^{(l)}_i}\le \frac{c^{(l+1)}_i}{d^{(l+1)}_i}$ for any $i,l$ and
$\lim_{l\to \infty}\frac{c^{(l)}_i}{d^{(l)}_i}\in\mathbb Q$ for any $i$.
Let $c_i,d_i$ be positive integers with  $\frac{c_i}{d_i}=\lim_{l\to \infty}\frac{c^{(l)}_i}{d^{(l)}_i}$.
Let $D=sP_0-\sum_{i=1}^{r}\frac{c_i}{d_i}P_i$.
Assume $R(\mathbb P_k^1,D_l)$ has a rational singularity for any $l$.
For any positive integer $n$, let $B^l_n = -p [-nD_l] + [-pnD_l]$  and $B_n = -p [-nD] + [-pnD]$ and let $(B^l_n)_{\mathrm{red}}$  and $(B_n)_{\mathrm{red}}$ be the reduced divisors with the same support as $B^l_n$ and $B_n$, respectively.
\begin{enumerate}
\item
We have
$$\mathrm{deg}[nD]=sn-\sum_{i=1}^{r}\left\lceil  \frac{nc_i}{d_i}  \right\rceil\ge -1$$
for any positive integer $n$.
In particular,   if $\mathrm{deg}D>0$, then $R(\mathbb P_k^1,D)$ has a rational singularity.

\item
If $p$ does not divide any $d_i$, then
\[ \mathrm{deg} [ -pnD ] + \mathrm{deg} (B_n)_{\mathrm{red}} \le  1\]
for any positive integer $n$.  

\item
If   $p>d_i$ for any $i$, then
\[ \mathrm{deg} [ -pnD_l ] + \mathrm{deg} (B^l_n)_{\mathrm{red}}\le \mathrm{deg} [ -pnD ] + \mathrm{deg} (B_n)_{\mathrm{red}}\]
for any positive integers $l$ and $n$.
\end{enumerate}
\end{lem}

\begin{proof}
(1)
If the lemma fails, then   there exists a positive integer $m$ such that
$$\mathrm{deg}[mD]=sm-\sum_{i=1}^{r}\left\lceil \frac{mc_i}{d_i}  \right\rceil\le -2.$$
Since $R(\mathbb P_k^1,D_l)$ has a rational singularity for any $l$, we have
$$\mathrm{deg}[nD_{l}]=sn-\sum_{i=1}^{r}\left\lceil  \frac{nc^{(l)}_i}{d^{(l)}_i}  \right\rceil\ge -1$$
for any positive integer $n$ by Theorem \ref{graded rational singularity}.
Since  $\frac{c^{(l)}_i}{d^{(l)}_i}\le \frac{c^{(l+1)}_i}{d^{(l+1)}_i}$ for any $i,l$, we have
$$\left\lceil  \frac{mc_i}{d_i}\right\rceil=\left\lceil  \frac{mc^{(l)}_i}{d^{(l)}_i}\right\rceil$$
for any $i$ and   any sufficiently large number $l$.
Therefore for  any sufficiently large $l$, we have
$$-1\le \mathrm{deg}[mD_{l}]=\mathrm{deg}[mD]\le -2,$$
which is contradiction.

If $\mathrm{deg}D>0$, then $R(\mathbb P_k^1,D)$ has a rational singularity by Theorem \ref{graded rational singularity}.

\noindent
(2) If the lemma fails, then   there exists a positive integer $m$ such that
\[ \mathrm{deg} [ -pmD ] + \mathrm{deg} (B_m)_{\mathrm{red}} \ge  2.\]
Let $l=\sharp\{i\in \mathbb N \mid  \frac{mc_i}{d_i}\not\in \mathbb Z \}$.
Then we have 
 $$\left\lceil  \frac{pmc_j}{d_j}\right\rceil-\left[ \frac{pmc_j}{d_j}\right]=1$$ for $j\in \{i\in \mathbb N\ |\ \frac{mc_i}{d_i}\not\in \mathbb N \}$
and $\mathrm{deg}(B_m)_{\mathrm{red}}\le l$.
Since we have for  every positive integer $n$,
$$\mathrm{deg}[nD]=sn-\sum_{i=1}^{r}\left\lceil   \frac{nc_i}{d_i}  \right\rceil\ge -1$$
by Lemma \ref{rationality lim} (1),
we have
\begin{align*}
\mathrm{deg}[-pmD]&= -pms+\sum_{i=1}^r\left[ \frac{pmc_i}{d_i}\right]
= -pms+\sum_{i=1}^r\left\lceil  \frac{pmc_i}{d_i}\right\rceil-l\\
&=-\mathrm{deg}[pmD]-l\le 1-l.
\end{align*}
Hence we have 
\[
2\le \mathrm{deg}[-pmD]+\mathrm{deg}(B_m)_{\mathrm{red}}\le 1-l+l=1,\]
 which is a contradiction.
 
\noindent
(3) 
Let $I=\left\{i\in\mathbb N \mid 1\le i\le r\right\}$,
$U_{l,n}=\{i\in I \mid  \frac{nc_i}{d_i}\in \mathbb Z, \frac{c^{(l)}_i}{d^{(l)}_i} \neq\frac{c_i}{d_i} \}$ and $V_n=\{i\in I\mid  \frac{nc_i}{d_i}\not\in \mathbb Z\}.$
If $\frac{nc_i}{d_i}\in \mathbb Z$ and $\frac{c^{(l)}_i}{d^{(l)}_i} \neq\frac{c_i}{d_i}$, then
$\left[ \frac{pnc_i}{d_i}\right]-\left[ \frac{pnc^{(l)}_i}{d^{(l)}_i}\right]\ge 1$.
Therefore we have
for any positive integers $l,n$,
\[\sum_{i=1}^r \left[ \frac{pnc_i}{d_i}\right]-\sum_{i=1}^r \left[ \frac{pnc^{(l)}_i}{d^{(l)}_i}\right]\ge \sharp U_{l,n}.
\]
Since $p>d_i$ for any $i$, we have $-p\left[\frac{nc_j}{d_j}\right]+\left[\frac{pnc_j}{d_j}\right]\ge 1$ for  any positive integer $n$ and $j\in V_n$.
Therefore we have for any positive integer $n$,
\begin{align*}
\mathrm{deg} (B_n)_{\mathrm{red}}&=\mathrm{deg}\left(-p\left[ \sum_{i=1}^r\frac{nc_i}{d_i}P_i\right]+\left[\sum_{i=1}^r \frac{pnc_i}{d_i}P_i\right]\right)_{\mathrm{red}}=\sharp V_n.
\end{align*}
We have $i\in U_{n,l}\cup V_n$
for any positive integers $l,n$ and $i\in I$ with $ \frac{nc^{(l)}_i}{d^{(l)}_i}\not\in \mathbb Z$.
Hence we have for any positive integers $l,n$,
\[
\sharp U_{l,n}+\sharp V_n\ge\sharp\left\{i\in I\ \middle|\  \frac{nc^{(l)}_i}{d^{(l)}_i}\not\in \mathbb Z\right\}\ge\mathrm{deg} (B^l_n)_{\mathrm{red}}.
\]
Therefore for any positive integers $l,n$,
\begin{align*}
&\mathrm{deg} [ -pnD ] + \mathrm{deg} (B_n)_{\mathrm{red}}
=-pns+\sum_{i=1}^r \left[ \frac{pnc_i}{d_i}\right]+\mathrm{deg} (B_n)_{\mathrm{red}}\\
\ge& -pns+\sum_{i=1}^r \left[ \frac{pnc^{(l)}_i}{d^{(l)}_i}\right]+\sharp U_{l,n}+\sharp V_n
\ge\mathrm{deg} [ -pnD_l ] + \mathrm{deg} (B^l_n)_{\mathrm{red}}.
\end{align*}

\end{proof}

\begin{thm}\label{$F$-rationality p(m)}
Let $m\in \mathbb N$.
 There exists a positive integer $p(m)$ such that 
   $R$ is $F$-rational for any two-dimensional  graded ring $R$ with a rational singularity,  $e(R)=m$ and $R_0=k$, an algebraically closed field of  characteristic  $p\ge p(m)$.
\end{thm}

\begin{proof}
If the theorem fails, then by Theorem \ref{DP const} and Theorem \ref{graded rational singularity}, there exist a positive integer $m$ and a sequence  $\{D_l\}_{l\in \mathbb N}$ of
ample $\mathbb Q$-divisors  on $\mathbb P_{k_l}^1$
  such that $R(\mathbb P_{k_l}^1,D_l)$ is a
two-dimensional non-$F$-rational  graded ring with a rational singularity, 
$$e(R(\mathbb P_{k_l}^1,D_l))=m$$
 and
$$\lim_{l\to \infty} p_l=\infty,$$
where  $p_l$ is the characteristic of the field $k_l$.
Let $E_0^{l}$ be the central curve of the exceptional set of the minimal good resolution of $\mathrm{Spec}(R(\mathbb P_{k_l}^1,D_l))$.
Then $-(E_0^{l})^2\le m$ by Proposition \ref{Z^2}.
Therefore by taking a subsequence of $\{D_l\}_{l\in \mathbb N}$, we may assume that $(E_0^{l})^2$ is constant for any $l$.
We put $s=-(E_0^{l})^2$  for any $l$.
Since $R(\mathbb P_{k_l}^1,D_l)$ has a rational singularity and is not $F$-rational, by Proposition \ref{f-rational deg 1}.(2), we may put $D_l=sP^{(l)}_0-\sum_{i=1}^{s+1}\frac{c^{(l)}_i}{d^{(l)}_i}P^{(l)}_i$, where $s,c^{(l)}_i,d^{(l)}_i\in\mathbb N$ with  $0<c^{(l)}_i<d^{(l)}_i$, and $P^{(l)}_i$ are distinct points of $\mathbb P_{k_l}^1$.

We denote by $e_j^{(l)}$ the number of $(-j)$-curves in the exceptional set of the minimal good resolution of $\mathrm{Spec}(R(\mathbb P_{k_l}^1,D_l))$.
By Proposition \ref{Z^2}, we have
\[
\sum_{j\ge 2}e_j^{(l)}(j-2)+2\le e(R(\mathbb P_{k_l}^1,D_l))= m.
\]
Hence we have
$e_j^{(l)} \le m$  for any $l,j$  with $3\le j \le m$ and 
$e_{j'}^{(l)}=0$ for any $l,j'$ with $j'\ge m+1$.
Therefore we may assume that 
$e_j^{(l)}$ is constant for any $l$ when we fix $j$ with $3\le j \le m$.
By Theorem \ref{dual graph of graded ring}, the number of the  branch in the dual graph of the minimal good resolution of $\mathrm{Spec}(R(\mathbb P_{k_l}^1,D_l))$ is $s+1$ and is constant for any $l$.
Then we may assume that the number of
$(-j)$-curves in the branch corresponding to $P^{(l)}_i$ in the dual graph is constant for any $l$ when we fix $i,j$ with $j\neq 2$.
Since the Hirzebruch-Jung continued fraction of  $\frac{d^{(l)}_i}{c^{(l)}_i}$ is the sequence of negatives of self intersection numbers of the exceptional curves  in the branch corresponding to $P^{(l)}_i$ in the minimal good resolution of  $\mathrm{Spec} (R(\mathbb P_{k_l}^1,D_l))$ by Theorem \ref{dual graph of graded ring} and Remark \ref{rem dual graph HJ fraction}(2),
$T(\frac{d^{(l)}_i}{c^{(l)}_i} )$ is the sequence of negatives of self intersection numbers of the exceptional curves, excluding those with self-intersection $-2$, in the branch corresponding to $P^{(l)}_i$ in the dual graph.
Therefore we may assume that
 $T(\frac{d^{(l)}_i}{c^{(l)}_i})$ is constant for any $l$ when we fix $i$.
Thus when we fix $i$, we may assume that a sequence $\left\{\frac{c^{(l)}_i}{d^{(l)}_i}\right\}_l$ is constant or strictly increasing by Lemma \ref{lem T(a)} (1).
Let  $I=\left\{i\in \mathbb N\ \middle|\ \left\{\frac{c^{(l)}_i}{d^{(l)}_i}\right\}_l\ \mbox{is strictly increasing} \right\}.$
Then we have for $i\in I$,
$\lim_{l\to \infty}\frac{d^{(l)}_i}{c^{(l)}_i}\in \mathbb Q_{\ge 1}$ by Lemma \ref{lem T(a)} (2).
Let $c_i,d_i$ be positive integers with   $0< c_i\le d_i$ and
 \[
\frac{c_i}{d_i} = \lim_{l\to \infty}\frac{c^{(l)}_i}{d^{(l)}_i}.
\]
For any positive integer $j,l$, let $F_j^{(l)}=sP_0^{(l)}-\sum_{i=1}^{s+1}\frac{c^{(j)}_i}{d^{(j)}_i}P_i^{(l)}$ and $F^{(l)}=sP_0^{(l)}-\sum_{i=1}^{s+1}\frac{c_i}{d_i}P_i^{(l)}$ be $\mathbb Q$-divisors on $\mathbb P_{k_l}^1$.
Since $R(\mathbb P_{k_j}^1,D_j)$ has a rational singularity, it follows from Theorem \ref{graded rational singularity} that $\mathrm{deg} [nF_j^{(l)}]=\mathrm{deg} [nD_j]\ge -1$ for any $n\in \mathbb Z$ and $j,l\in \mathbb N$.
Hence  $R(\mathbb P_{k_l}^1,F_j^{(l)})$ is a
two-dimensional graded ring with a rational singularity for any $j\in \mathbb N$ by Theorem \ref{graded rational singularity}.
Since $\lim_{l\to \infty} p_l=\infty,$ we may assume that $p_l>d_i$ for any $i,l$.
By Lemma \ref{rationality lim} (2) and (3), we have for any positive integers $j,l,n$,
\[ \mathrm{deg} [ -p_lnF_j^{(l)} ] + \mathrm{deg} (B^{(l)}_{j,n})_{\mathrm{red}}\le \mathrm{deg} [ -p_lnF^{(l)} ] + \mathrm{deg} (B^{(l)}_{n})_{\mathrm{red}}\le1,\]
where $B^{(l)}_{j,n}=-p_l [-nF_j^{(l)}] + [-p_lnF_j^{(l)}]$ and $B^{(l)}_{n}=-p_l [-nF^{(l)}] + [-p_lnF^{(l)}]$.  
Since $F_l^{(l)}=D_l$ for any $l$, we have for any positive integers $l,n$
\[ \mathrm{deg} [ -p_lnD_l ] + \mathrm{deg} (B^{(l)}_{l,n})_{\mathrm{red}}\le 1.\]
By Theorem \ref{criterion},  $R(\mathbb P_{k_l}^1,D_l)$ is $F$-rational for any positive integer $l$, which is contradiction. 
\end{proof}

\begin{ex}\label{ex1}
Let $D=2P_0-\frac{p+1}{2p}P_1-\frac{p-1}{p}P_2-\frac{1}{2}P_3$, where $P_i$ are distinct points of $\mathbb P_k^1$.
Then $R=R(\mathbb P_k^1,D)$ has a rational singularity with $e(R)=\left\lceil  \frac{p+1}{2}  \right\rceil$
 but is not $F$-rational.
Indeed, we have for $m\in \mathbb N$,
\[
\mathrm{deg}[2mD]=\left[\frac{2m}{p}\right]-\left\lceil  \frac{m}{p}  \right\rceil\ge -1
\]
and
\[
\mathrm{deg}[(2m-1)D]=\left[\frac{2m-1}{p}\right]-\left\lceil  \frac{p+2m-1}{2p}  \right\rceil\ge -1.
\]
Therefore $R$ has a rational singularity by Theorem \ref{graded rational singularity}.
Since \[ \mathrm{deg} [ -pD ] + \mathrm{deg} (B_1)_{\mathrm{red}} = 2  ,\]
where $B_1=-p [-D] + [-pD]$,   $R$ is not $F$-rational by Theorem \ref{criterion}.

If  $p=2$, then the dual graph of the minimal good resolution of $\mathrm{Spec} (R)$ is the following:
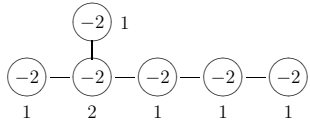
\begin{eqnarray*}
\scalebox{0.6}{
\xymatrix@C=12pt@R=12pt{
 & *++[o][F]{-2} \ar@{-}[d] \ar@{}[r]|(0.5) *{1}
 &
 &
\\
  *++[o][F]{-2}  \ar@{-}[r]  \ar@{}[d]|(0.8) *{1}
 & *++[o][F]{-2}  \ar@{-}[r] \ar@{}[d]|(0.8) *{2}
 & *++[o][F]{-2}  \ar@{-}[r] \ar@{}[d]|(0.8) *{1}
 & *++[o][F]{-2}  \ar@{-}[r] \ar@{}[d]|(0.8) *{1}
 & *++[o][F]{-2} \ar@{}[d]|(0.8) *{1} \\
&&&&
}
}
\end{eqnarray*}
Here the  number next to a vertex means  the coefficient of the  relevant exceptional divisor in the fundamental cycle.
Therefore we have $e(R)=2$ by Proposition \ref{Z^2}.

If  $p\ge 3$, then the dual graph of the minimal good resolution of $\mathrm{Spec} (R)$ is the following:
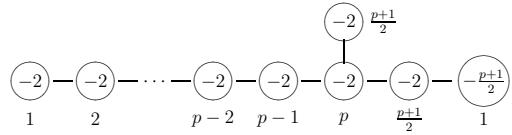
\begin{eqnarray*}
\scalebox{0.6}{
\xymatrix@C=12pt@R=12pt{
 &
 &
 & 
 &
 &*++[o][F]{-2}  \ar@{-}[d] \ar@{}[r]|(0.6) *{\frac{p+1}{2}}
 &
 &
\\
 *++[o][F]{-2}  \ar@{-}[r]  \ar@{}[d]|(0.8) *{1}
 &*++[o][F]{-2}  \ar@{-}[r]  \ar@{}[d]|(0.8) *{2}
 & \cdots \ar@{-}[r]  
 &*++[o][F]{-2}  \ar@{-}[r] \ar@{}[d]|(0.8) *{p-2}
 &*++[o][F]{-2}  \ar@{-}[r] \ar@{}[d]|(0.8) *{p-1}
 &*++[o][F]{-2}  \ar@{-}[r]  \ar@{}[d]|(0.8) *{p}
 &*++[o][F]{-2}  \ar@{-}[r] \ar@{}[d]|(0.8) *{\frac{p+1}{2}}
 &*++[o][F]{-\frac{p+1}{2}} \ar@{}[d]|(0.8) *{1}\\
&&&&&&&
}
}
\end{eqnarray*}
Note that $\min \{ n \in\mathbb N \;|\; \mathrm{deg} [ n D] \ge 0 \} =p$.
Therefore we can compute the fundamental cycle by Lemma \ref{formula} and Corollary \ref{n_0=min}.
We have $e(R)=\frac{p+1}{2}$ by Proposition \ref{Z^2}.
\end{ex}

\begin{rem}
Example \ref{ex1} implies that  $p(m)> 2m-1$, where $p(m)$ is the positive integer in Theorem \ref{$F$-rationality p(m)}.
\end{rem}

\begin{rem}
Theorem \ref{$F$-rationality p(m)} does not hold for higher dimensional graded rings.
In fact, consider the ring $R=k[x,y,z,w]/(x^3+y^3+z^3+w^p)$.
Then $R$ has rational singularities but $R$ is not $F$-rational if $3$ does not divide $p-1$ by Theorem \ref{HH F-rational}, \cite[Remark 3.8]{EH} and \cite[Proposition 2.1]{Fe}.
\end{rem}

\section{Classification of $R(\mathbb P_k^1,D)$ which is a rational triple point and rational  fourth point}

In this section we classify normal graded rings $R(\mathbb P_k^1,D)$ with a rational singularity and $e(R(\mathbb P_k^1,D))=3$ and $4$.

\subsection{Preliminaries of classification of $R(\mathbb P_k^1,D)$}
In this subsection, we give results for the classification of $R(\mathbb P_k^1,D)$ with a rational singularity and $e(R(\mathbb P_k^1,D))=3$ and $4$.

\begin{lem}\label{bigger coefficient rationality}
Let $D_1=\sum_{i=1}^ra_iP_i$ and $D_2=\sum_{i=1}^rb_iP_i$ be ample $\mathbb Q$-divisors on $\mathbb P_k^1$, where $P_i$ are distinct points of $\mathbb P_k^1$.
Assume  $a_i \ge b_i$ for any $i$.
If $R(\mathbb P_k^1,D_2)$ has a rational singularity, then $R(\mathbb P_k^1,D_1)$ has a rational singularity.
\end{lem}

\begin{proof}
Since $\mathrm{deg}[nD_1]\ge \mathrm{deg}[nD_2]$ for any $n\in\mathbb N$,
$R(\mathbb P_k^1,D_1)$ has a rational singularity by Theorem \ref{graded rational singularity}.
\end{proof}

\begin{lem}\label{D form}
Let $D=sP_0 - \sum_{i=1}^r \frac{c_i}{d_i} P_i$ be an ample $\mathbb Q$-divisor on $\mathbb P_k^1$, where $s,c_i,d_i\in \mathbb N$ with  $0<c_i<d_i$, and $P_i$ are distinct points of $\mathbb P_k^1$.
Let $R=R(\mathbb P_k^1,D)$ and $f:X\to \mathrm{Spec} (R)$  the minimal good resolution.
Assume that $R$ has a rational singularity.  
\begin{enumerate}
\item If $e(R)= 3$, then the dual graph of $f$ has the following property;\\
There is unique $(-3)$-curve  and others are $(-2)$-curves.
In this case, $D$ is one of the following: for some $n_i,a,b\in \mathbb Z_{\ge 0}$,
\[3P_0 - \sum_{i=1}^3 \frac{n_i}{n_i+1} P_i \ \ \mbox{ or}\]
 \[2P_0 - \sum_{i=1}^2 \frac{n_i}{n_i+1} P_i-\frac{1}{[[(2)^a,3,(2)^b]]}P_3.\]

\item If $e(R)= 4$,  then the dual graph of $f$ has one of  the following properties;
\begin{enumerate}
\item There is unique $(-4)$-curve  and others are $(-2)$-curves.
In this case, $D$ is one of the following: for some $n_i,a,b\in \mathbb Z_{\ge 0}$,
\[4P_0 - \sum_{i=1}^4 \frac{n_i}{n_i+1} P_i \ \mbox{ or}\]
 \[2P_0 - \sum_{i=1}^2 \frac{n_i}{n_i+1} P_i-\frac{1}{[[(2)^a,4,(2)^b]]}P_3.\]
\item There is unique $(-3)$-curve and others are $(-2)$-curves.
In this case, $D$ is one of the following: for some $n_i,a,b\in \mathbb Z_{\ge 0}$,
\[3P_0 - \sum_{i=1}^4 \frac{n_i}{n_i+1} P_i\ \ \mbox{ or}\]
 \[2P_0 - \sum_{i=1}^2 \frac{n_i}{n_i+1} P_i-\frac{1}{[[(2)^a,3,(2)^b]]}P_3.\]
\item There are two $(-3)$-curves  and others are $(-2)$-curves.
In this case, $D$ is one of the following: for some $n_i,n,a,b,c,d\in \mathbb Z_{\ge 0}$,
\[3P_0 - \sum_{i=1}^2 \frac{n_i}{n_i+1} P_i-\frac{1}{[[(2)^a,3,(2)^b]]}P_3,\]
\[3P_0 - \sum_{i=1}^2 \frac{n_i}{n_i+1} P_i-\frac{1}{[[(2)^a,3,(2)^b,3,(2)^c]]}P_3\ \ \mbox{ or }\]
 \[2P_0 - \frac{n}{n+1} P_1-\frac{1}{[[(2)^a,3,(2)^b]]} P_2 -\frac{1}{[[(2)^c,3,(2)^d]]} P_3.\]
\end{enumerate}
\end{enumerate}
\end{lem}

\begin{proof}
We prove only  $(2)$, as  $(1)$ is proved similarly.
By  Proposition \ref{Z^2}, the dual graph of $f$ has one of  the following properties;
\begin{enumerate}
\item[(a)] There is unique $(-4)$-curve  and others are $(-2)$-curves.
\item[(b)] There is unique $(-3)$-curve and others are $(-2)$-curves.
\item[(c)] There are two $(-3)$-curves  and others are $(-2)$-curves.
\end{enumerate}

By  Lemma \ref{non-rational s<r-1}, we have $s+1\ge r$. 
Let $Z$ be the fundamental cycle of $f$, and let $E_0$ be the central curve of $f$. 
If   $s+1= r$, then $\mathrm{Coeff}_{E_0}(Z)\ge 2$ by Corollary \ref{n_0=min}.
Therefore by Proposition \ref{Z^2}, if $E_0$ is a $(-4)$-curve and $r=5$, then $e(R)\ge 6$, and if  there are two $(-3)$-curves in the dual graph, $E_0$ is a $(-3)$-curve and $r=4$, then $e(R)\ge 5$.

Note that $ [[(2)^m]]=\frac{m+1}{m}$ for $m\in\mathbb N$ by Example \ref{[[2]]}.
By  Theorem \ref{dual graph of graded ring}, 
we can determine the coefficients of $D$.
\end{proof}

\begin{lem}\label{[[2,n,2]]}  Let $n,a,b\in \mathbb Z_{\ge 0}$ with $n\ge 2$. Then we have
\[  [[(2)^a, n, (2)^b]] = \dfrac{\big((a+1)n-(2a+1)\big)b+(a+1)n-a}{\big(an-(2a-1)\big)b+an-(a-1)}.\]
\end{lem}

\begin{proof}
Note that $ [[(2)^m]]=\frac{m+1}{m}$ for $m\in\mathbb N$ by Example \ref{[[2]]}.
We prove this by induction on $a$.
If $a=0$, then 
\[  [[(2)^a, n, (2)^b]] = [[n, \frac{b+1}{b}]]= n-\frac{b}{b+1}=\frac{(n-1)b+n}{b+1}.\]
If $a>0$, then 
\begin{eqnarray*}
  [[(2)^{a+1}, n, (2)^b]]& =& [[2,(2)^a,n, (2)^b]]\\
&=& [[2,\dfrac{\big((a+1)n-(2a+1)\big)b+(a+1)n-a}{\big(an-(2a-1)\big)b+an-(a-1)}]]\\
&=&2-\dfrac{\big(an-(2a-1)\big)b+an-(a-1)}{\big((a+1)n-(2a+1)\big)b+(a+1)n-a}\\
&=&\dfrac{\big((a+2)n-(2a+3)\big)b+(a+2)n-(a+1)}{\big((a+1)n-(2a+1)\big)b+(a+1)n-a}.
\end{eqnarray*}
\end{proof}

\begin{lem}\label{[[2,3,2,3,2]]}  Let $a,b,c\in \mathbb Z_{\ge 0}$. Then we have
\[  [[(2)^a, 3, (2)^b,3, (2)^c]] = \dfrac{\big((a+2)b+ 3a+5\big)c+(2a+4)b+5a+8}{\big((a+1)b + 3a+2\big)c+(2a+2)b+5a+3}.\]

\end{lem}

\begin{proof}
We prove this by induction on $a$.
If $a=0$, then by Lemma \ref{[[2,n,2]]} 
\[   [[(2)^a, 3, (2)^b,3, (2)^c]] = [[3,\dfrac{(b+2)c+2b+3}{(b+1)c + 2b+1}]]=\dfrac{(2b +5)c+4b+8}{(b +2)c+2b+3}.\]
If $a>0$, then 
\begin{eqnarray*}
 [[(2)^{a+1}, 3, (2)^b,3, (2)^c]]& =& [[2,(2)^a, 3, (2)^b,3, (2)^c]]\\
&=& [[2,\dfrac{\big((a+2)b+ 3a+5\big)c+(2a+4)b+5a+8}{\big((a+1)b + 3a+2\big)c+(2a+2)b+5a+3}]]\\
&=&2-\dfrac{\big((a+1)b + 3a+2\big)c+(2a+2)b+5a+3}{\big((a+2)b+ 3a+5\big)c+(2a+4)b+5a+8}\\
&=&\dfrac{\big((a+3)b+ 3a+8\big)c+(2a+6)b+5a+13}{\big((a+2)b + 3a+5\big)c+(2a+4)b+5a+8}.
\end{eqnarray*}
\end{proof}

In next subsections,
we will use Lemma \ref{bigger coefficient rationality} and the following result to check whether $R(\mathbb P_k^1,D)$ has a rational singularity for $D$ in the list of Lemma \ref{D form}.
\begin{lem}\label{(2,2,n)type}
Let $D=2P_0 -  a_1 P_1 -  a_2P_2-  a_3 P_3$ be a $\mathbb Q$-divisor on $\mathbb P_k^1$, where $a_i\in \mathbb Q_{\ge 0}$  and $P_i$ are distinct points of $\mathbb P_k^1$.
Then  $R(\mathbb P_k^1,D)$ has a rational singularity,
if $(a_1,a_2,a_3)$ is equal to $(\frac{1}{2},\frac{1}{2},\frac{n}{n+1})$ for some $n\in \mathbb Z_{\ge 0}$ or $(\frac{1}{2},\frac{2}{3},\frac{4}{5})$.  
\end{lem}

\begin{proof}
If $(a_1,a_2,a_3)=(\frac{1}{2},\frac{1}{2},\frac{n}{n+1})$ for some $n\in \mathbb Z_{\ge 0}$, then for any $l\in\mathbb N$,
\[
\mathrm{deg}[lD]=2l-\left\lceil  \frac{l}{2}\right\rceil-\left\lceil  \frac{l}{2}\right\rceil-\left\lceil  \frac{ln}{n+1}\right\rceil\ge \left[ \frac{l}{2}\right]-\left\lceil  \frac{l}{2}\right\rceil\ge -1,
\]
which implies that $R(\mathbb P_k^1,D)$ has a rational singularity by Theorem \ref{graded rational singularity}. 

If $(a_1,a_2,a_3)=(\frac{1}{2},\frac{2}{3},\frac{4}{5})$, then $\mathrm{deg}[lD]\ge -1$ for  any $l\in\mathbb N$ with $1\le l \le 29$ and  $\mathrm{deg}[30D]=1$.
Since  $\mathrm{deg}[lD]=\mathrm{deg}[(l-30)D]+\mathrm{deg}[30D]$ any $l\in\mathbb N$ with $l\ge 30$, $\mathrm{deg}[lD]\ge -1$ for  any $l\in\mathbb N$.
Hence $R(\mathbb P_k^1,D)$ has a rational singularity by Theorem \ref{graded rational singularity}. 
\end{proof}

In next subsections, we determine $D$ in the list of Lemma \ref{D form} such that $R(\mathbb P_k^1,D)$ has a rational singularity with  $e(R(\mathbb P_k^1,D))=3$ and $4$ using the following steps:
\begin{enumerate}
\item We will check whether $R(\mathbb P_k^1,D)$ has a rational singularity by Theorem \ref{graded rational singularity} or Lemma \ref{bigger coefficient rationality}.

\item We will determine the fundamental cycle of the minimal good resolution of $\mathrm{Spec} (R(\mathbb P_k^1,D))$ by Theorem \ref{dual graph of graded ring},  Lemma \ref{formula} and  Corollary \ref{n_0=min}.

\item We will determine  $e(R(\mathbb P_k^1,D))$ by Proposition \ref{Z^2}.

\item We will compute the  Hirzebruch-Jung continued fractions 
\[ [[(2)^a, 3, (2)^b]], [[(2)^a, 4, (2)^b]], [[(2)^a, 3, (2)^b,3, (2)^c]] \] by Lemma \ref{[[2,n,2]]} and Lemma \ref{[[2,3,2,3,2]]}.
\end{enumerate}

\subsection{The case there is unique $(-3)$-curve}
In this subsection we classify the  $R(\mathbb P_k^1,D)$ with a rational singularity such that 
there is unique $(-3)$-curve   in the dual graph of the minimal good resolution of $\mathrm{Spec}(R(\mathbb P_k^1,D))$ and others are $(-2)$-curves.
First, we consider the case the central curve is a $(-3)$-curve.

Recall that the  number next to a vertex of a dual graph denotes  the coefficient of the  relevant exceptional divisor in the fundamental cycle.
\begin{prop}\label{central 3}
Let $D=3P_0-\sum_{i=1}^4a_iP_i$  be an ample $\mathbb Q$-divisor on $\mathbb P_k^1$, where $a_i\in \mathbb Q_{\ge 0}$ and $P_i$ are distinct points of $\mathbb P_k^1$.
Assume that $a_1\le a_2\le a_3\le a_4$, $a_1=\frac{a}{a+1}$,  $a_2=\frac{b}{b+1}$, $a_3=\frac{c}{c+1}$ and $a_4=\frac{d}{d+1}$ for $a,b,c,d\in \mathbb Z_{\ge 0}$ and $R(\mathbb P_k^1,D)$ has a rational singularity. 
Then 
$(a_1,a_2,a_3,a_4)=(0,\frac{b}{b+1},\frac{c}{c+1},\frac{d}{d+1})$ for $0\le b\le c\le d$ or $(\frac{1}{2},\frac{1}{2},\frac{c}{c+1},\frac{d}{d+1})$
for $1\le c\le d$.
Moreover  if \[(a_1,a_2,a_3,a_4)=(0,\frac{b}{b+1},\frac{c}{c+1},\frac{d}{d+1})\] for $0\le b\le c\le d$, then $e(R(\mathbb P_k^1,D))=3$ and if \[(a_1,a_2,a_3,a_4)=(\frac{1}{2},\frac{1}{2},\frac{c}{c+1},\frac{d}{d+1})\] for $1\le c\le d$, then $e(R(\mathbb P_k^1,D))=4$.
\end{prop}

\begin{proof}
Since $R(\mathbb P_k^1,D)$ has a rational singularity, we have  $\mathrm{deg} [2D]\ge -1$.
Therefore $a=0$ or $a=b=1$.

Assume $(a_1,a_2,a_3,a_4)=(0,\frac{b}{b+1},\frac{c}{c+1},\frac{d}{d+1})$ for $0\le  b\le c\le d$.
$R(\mathbb P_k^1,D)$ has a rational singularity since $\mathrm{deg} [lD]\ge 0$ for any $l\in\mathbb N$.
The dual graph of the minimal good resolution of $\mathrm{Spec} (R(\mathbb P_k^1,D))$ is the following:
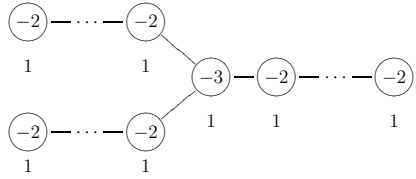
\begin{equation*}
\scalebox{0.6}{
\xymatrix@C=12pt@R=12pt{
 *++[o][F]{-2} \ar@{-}[r] \ar@{}[d]|(0.8) *{1}
 &  \cdots \ar@{-}[r]
 & *++[o][F]{-2} \ar@{}[d]|(0.8) *{1} \ar@{-}[dr]\\
&&&*++[o][F]{-3} \ar@{}[d]|(0.8) *{1}\ar@{-}[r] &*++[o][F]{-2} \ar@{}[d]|(0.8) *{1} \ar@{-}[r]  & \cdots\ar@{-}[r] &*++[o][F]{-2} \ar@{}[d]|(0.8) *{1} \\
*++[o][F]{-2} \ar@{-}[r] \ar@{}[d]|(0.8) *{1}
 &  \cdots \ar@{-}[r]
 & *++[o][F]{-2} \ar@{-}[ur] \ar@{}[d]|(0.8) *{1} &&&&\\
 &&&&&
}
}
\end{equation*}
Therefore $e(R(\mathbb P_k^1,D))=3$.

 Assume $(a_1,a_2,a_3,a_4)=(\frac{1}{2},\frac{1}{2},\frac{c}{c+1},\frac{d}{d+1})$
for $1\le c\le d$.
Then 
$$\mathrm{deg} \left[lD\right]= 3l+\left[-\frac{l}{2}\right]+\left[-\frac{l}{2}\right]+\left[-\frac{lc}{c+1}\right]+\left[-\frac{ld}{d+1}\right]\ge \left[\frac{l}{2}\right]+\left[-\frac{l}{2}\right]\ge -1$$
for any $l\in \mathbb N$.
Therefore $R(\mathbb P_k^1,D)$ has a rational singularity.
The dual graph of the minimal good resolution of $\mathrm{Spec} (R(\mathbb P_k^1,D))$ is the following:
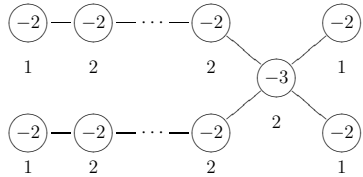
\begin{equation*}
\scalebox{0.6}{
\xymatrix@C=12pt@R=12pt{
 *++[o][F]{-2} \ar@{-}[r] \ar@{}[d]|(0.8) *{1} 
& *++[o][F]{-2} \ar@{-}[r] \ar@{}[d]|(0.8) *{2} 
 &  \cdots \ar@{-}[r]
 & *++[o][F]{-2} \ar@{-}[dr] \ar@{}[d]|(0.8) *{2} 
 &
 & *++[o][F]{-2} \ar@{}[d]|(0.8) *{1} 
 \\
&&&&*++[o][F]{-3} \ar@{-}[dr] \ar@{}[d]|(0.8) *{2}  \ar@{-}[ur]  &\\
*++[o][F]{-2} \ar@{}[d]|(0.8) *{1} \ar@{-}[r] 
& *++[o][F]{-2} \ar@{}[d]|(0.8) *{2}  \ar@{-}[r]
 &  \cdots \ar@{-}[r]
 & *++[o][F]{-2} \ar@{}[d]|(0.8) *{2} \ar@{-}[ur]
 &
& *++[o][F]{-2} \ar@{}[d]|(0.8) *{1} \\
&&&&&&
}
}
\end{equation*}
Therefore $e(R(\mathbb P_k^1,D))=4$.
\end{proof}

Next, we consider the case the central curve is a $(-2)$-curve.
\begin{prop}\label{central 2-3}
Let $D=2P_0-\sum_{i=1}^3a_iP_i$  be an ample $\mathbb Q$-divisor on $\mathbb P_k^1$, where $a_i\in \mathbb Q_{\ge 0}$ and $P_i$ are distinct points of $\mathbb P_k^1$.
Assume that $a_1\le a_2$, $a_1=\frac{m}{m+1}$,  $a_2=\frac{n}{n+1}$, $\frac{1}{a_3}=[[(2)^a,3, (2)^b]]$ for $m,n,a,b\in\mathbb Z_{\ge 0}$ and $R(\mathbb P_k^1,D)$ has a rational singularity.
\begin{enumerate}
\item[(i)]  If  $e(R(\mathbb P_k^1,D))=3$,
then $(a_1,a_2,a_3)$ is  one of the following:
\begin{align*}
\mathrm{(1)}&\ (0,\frac{n}{n+1},\frac{(a+1)b+2a+1}{(a+2)b+2a+3}) & &\mbox{for}\ n\ge 0, a\ge 0, b\ge 0, \\
\mathrm{(2)}&\ (\frac{1}{2},\frac{n}{n+1},\frac{b+1}{2b+3}) & &\mbox{for}\ n\ge 1,b\ge 0 \\
\mathrm{(3)}&\ (\frac{1}{2},\frac{1}{2},\frac{(a+1)b+2a+1}{(a+2)b+2a+3}) & &\mbox{for}\ a\ge 1, b\ge 0, \\
 \mathrm{(4)}&\ (\frac{1}{2},\frac{2}{3},\frac{2b+3}{3b+5})\ \ \mbox{for}\ b\ge 0,& \mathrm{(5)}&\ (\frac{1}{2},\frac{2}{3},\frac{3b+5}{4b+7})\ \ \mbox{for}\ b\ge 0,\\
 \mathrm{(6)}&\ (\frac{1}{2},\frac{2}{3},\frac{7}{9}),&\mathrm{(7)}&\ (\frac{1}{2},\frac{3}{4},\frac{3}{5}),\\
\mathrm{(8)}&\ (\frac{1}{2},\frac{4}{5},\frac{3}{5}),&\mathrm{(9)}&\ (\frac{2}{3},\frac{n}{n+1},\frac{1}{3})\ \ \mbox{for}\ n\ge 2.
\end{align*}

\item[(ii)] If $e(R(\mathbb P_k^1,D))=4$,
then $(a_1,a_2,a_3)$ is  one of the following:
\begin{align*}
 \mathrm{(1)}&\ (\frac{1}{2},\frac{2}{3},\frac{4b+7}{5b+9})\ \ \mbox{for}\ b\ge 1,&\mathrm{(2)}&\ (\frac{1}{2},\frac{3}{4},\frac{2b+3}{3b+5})\ \ \mbox{for}\ b\ge 1,\\
\mathrm{(3)}&\ (\frac{1}{2},\frac{4}{5},\frac{2b+3}{3b+5})\ \ \mbox{for}\ b\ge 1,&\mathrm{(4)}&\ (\frac{1}{2},\frac{5}{6},\frac{3}{5}),\\
\mathrm{(5)}&\ (\frac{1}{2},\frac{6}{7},\frac{3}{5}),& \mathrm{(6)}&\ (\frac{2}{3},\frac{2}{3},\frac{b+1}{2b+3})\ \  \mbox{for}\ b\ge 1,\\
\mathrm{(7)}&\ (\frac{2}{3},\frac{3}{4},\frac{b+1}{2b+3})\ \  \mbox{for}\  b\ge 1,& \mathrm{(8)}&\ (\frac{2}{3},\frac{4}{5},\frac{2}{5}),\\
\mathrm{(9)}&\ (\frac{3}{4},\frac{3}{4},\frac{1}{3}),& \mathrm{(10)}&\ (\frac{3}{4},\frac{4}{5},\frac{1}{3}),\\
\mathrm{(11)}&\ (\frac{3}{4},\frac{5}{6},\frac{1}{3}).
\end{align*}

\end{enumerate}

\end{prop}

\begin{proof}
\noindent
{\bf Case 1.}
We assume that $m=0$.
Then $R(\mathbb P_k^1,D)$ has a rational singularity since 
 $\mathrm{deg} [lD]\ge 0$ for any $l\in\mathbb N$. 
The dual graph of the minimal good resolution of $\mathrm{Spec} (R(\mathbb P_k^1,D))$ is the following:
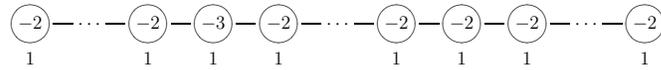
\begin{eqnarray*}
\scalebox{0.6}{
\xymatrix@C=12pt@R=12pt{
 *++[o][F]{-2} \ar@{-}[r] \ar@{}[d]|(0.8) *{1} 
 & \cdots \ar@{-}[r]  
 &*++[o][F]{-2} \ar@{-}[r] \ar@{}[d]|(0.8) *{1} 
 &*++[o][F]{-3} \ar@{-}[r] \ar@{}[d]|(0.8) *{1} 
 &*++[o][F]{-2} \ar@{-}[r]\ar@{}[d]|(0.8) *{1} 
 & \cdots \ar@{-}[r] 
 &*++[o][F]{-2} \ar@{-}[r] \ar@{}[d]|(0.8) *{1} 
 &*++[o][F]{-2} \ar@{-}[r] \ar@{}[d]|(0.8) *{1} 
 &*++[o][F]{-2} \ar@{-}[r] \ar@{}[d]|(0.8) *{1} 
 & \cdots \ar@{-}[r]  
 &*++[o][F]{-2} \ar@{}[d]|(0.8) *{1} \\
&&&&&&&&&&&
}
}
\end{eqnarray*}
Therefore $e(R(\mathbb P_k^1,D))=3$.

\vskip.3truecm
\noindent
{\bf Case 2.}
We assume that $m=1$ and  $a=0$.
Note that   $D\ge  2P_0 -  \frac{1}{2} P_1-\frac{n}{n+1} P_2 -\frac{1}{2}P_3$ by Lemma \ref{[[a]]<[[b]]}. 
Therefore $R(\mathbb P_k^1,D)$ has a rational singularity by Lemma \ref{bigger coefficient rationality} and  Lemma \ref{(2,2,n)type}.
The dual graph of the minimal good resolution of $\mathrm{Spec} (R(\mathbb P_k^1,D))$ is the following:
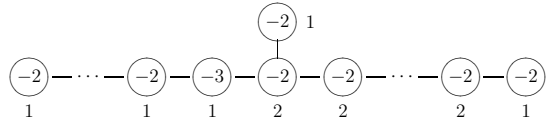
\begin{eqnarray*}
\scalebox{0.6}{
\xymatrix@C=12pt@R=12pt{
 &
 &
 &
 & *++[o][F]{-2}  \ar@{-}[d] \ar@{}[r]|(0.5) *{1}
 &
 &
 & 
\\
  *++[o][F]{-2} \ar@{-}[r] \ar@{}[d]|(0.8) *{1} 
 &  \cdots \ar@{-}[r]
 & *++[o][F]{-2} \ar@{-}[r] \ar@{}[d]|(0.8) *{1} 
 & *++[o][F]{-3} \ar@{-}[r] \ar@{}[d]|(0.8) *{1} 
 & *++[o][F]{-2} \ar@{-}[r] \ar@{}[d]|(0.8) *{2} 
 & *++[o][F]{-2} \ar@{-}[r] \ar@{}[d]|(0.8) *{2} 
 & \cdots \ar@{-}[r] 
 & *++[o][F]{-2} \ar@{-}[r] \ar@{}[d]|(0.8) *{2} 
 & *++[o][F]{-2} \ar@{}[d]|(0.8) *{1} 
\\
&&&&&&&&&
}
}
\end{eqnarray*}
Therefore $e(R(\mathbb P_k^1,D))=3$.

\vskip.3truecm
\noindent
{\bf Case 3.}
We assume that $m=n=1$ and $a\ge 1$.
Note that   $D\ge  2P_0 -  \frac{1}{2} P_1-\frac{1}{2} P_2 -\frac{a+b+1}{a+b+2}P_3$ by Lemma \ref{[[a]]<[[b]]}. 
Therefore $R(\mathbb P_k^1,D)$ has a rational singularity by Lemma \ref{bigger coefficient rationality} and  Lemma \ref{(2,2,n)type}.
The dual graph of the minimal good resolution of $\mathrm{Spec} (R(\mathbb P_k^1,D))$ is the following:
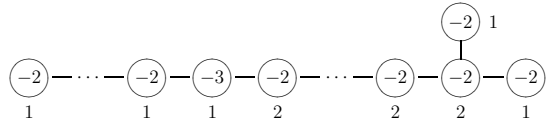
\begin{eqnarray*}
\scalebox{0.6}{
\xymatrix@C=12pt@R=12pt{
 &
 &
 &
 & 
 & 
 &
 &*++[o][F]{-2} \ar@{-}[d] \ar@{}[r]|(0.5) *{1}
 &
\\
 *++[o][F]{-2} \ar@{-}[r] \ar@{}[d]|(0.8) *{1} 
 & \cdots \ar@{-}[r]  
 &*++[o][F]{-2} \ar@{-}[r] \ar@{}[d]|(0.8) *{1} 
 &*++[o][F]{-3} \ar@{-}[r] \ar@{}[d]|(0.8) *{1} 
 &*++[o][F]{-2} \ar@{-}[r] \ar@{}[d]|(0.8) *{2} 
 & \cdots \ar@{-}[r] 
 &*++[o][F]{-2} \ar@{-}[r] \ar@{}[d]|(0.8) *{2} 
 &*++[o][F]{-2} \ar@{-}[r] \ar@{}[d]|(0.8) *{2} 
 &*++[o][F]{-2} \ar@{}[d]|(0.8) *{1}  \\
&&&&&&&&&
}
}
\end{eqnarray*}
Therefore $e(R(\mathbb P_k^1,D))=3$.

\vskip.3truecm
\noindent
{\bf Case 4.}
We assume that $m=1$, $n=2$ and $1 \le a\le 3$.
Note that  $D\ge  2P_0 -  \frac{1}{2} P_1-\frac{2}{3} P_2 -\frac{4}{5}P_3$ by Lemma \ref{[[a]]<[[b]]}. 
Therefore $R(\mathbb P_k^1,D)$ has a rational singularity by Lemma \ref{bigger coefficient rationality} and  Lemma \ref{(2,2,n)type}.
The dual graphs of the minimal good resolution of $\mathrm{Spec} (R(\mathbb P_k^1,D))$ are the following:
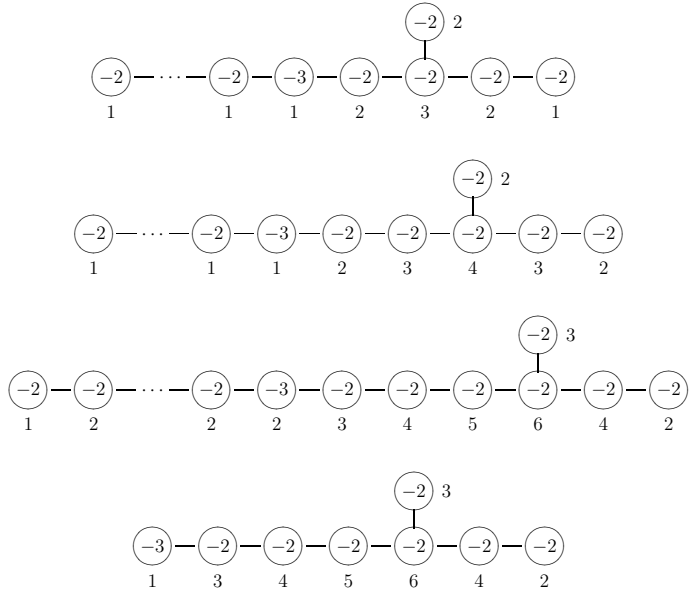
\begin{eqnarray*}
\scalebox{0.6}{
\xymatrix@C=12pt@R=12pt{
 &
 &
 & 
 &
 &*++[o][F]{-2}  \ar@{-}[d] \ar@{}[r]|(0.5) *{2}
 &
 &
\\
 *++[o][F]{-2}  \ar@{-}[r]  \ar@{}[d]|(0.8) *{1} 
 & \cdots \ar@{-}[r]  
 &*++[o][F]{-2}  \ar@{-}[r]  \ar@{}[d]|(0.8) *{1} 
 &*++[o][F]{-3}  \ar@{-}[r] \ar@{}[d]|(0.8) *{1} 
 &*++[o][F]{-2}  \ar@{-}[r] \ar@{}[d]|(0.8) *{2} 
 &*++[o][F]{-2}  \ar@{-}[r] \ar@{}[d]|(0.8) *{3} 
 &*++[o][F]{-2}  \ar@{-}[r] \ar@{}[d]|(0.8) *{2} 
 &*++[o][F]{-2} \ar@{}[d]|(0.8) *{1} \\
&&&&&&&&&
}
}
\end{eqnarray*}
\begin{eqnarray*}
\scalebox{0.6}{
\xymatrix@C=12pt@R=12pt{
 &
 &
 &
 & 
 &
 &*++[o][F]{-2}  \ar@{-}[d] \ar@{}[r]|(0.5) *{2}
 &
 &
\\
  *++[o][F]{-2}  \ar@{}[d]|(0.8) *{1}  \ar@{-}[r] 
 & \cdots \ar@{-}[r]  
 & *++[o][F]{-2}  \ar@{-}[r] \ar@{}[d]|(0.8) *{1}
 & *++[o][F]{-3}  \ar@{-}[r] \ar@{}[d]|(0.8) *{1}
 & *++[o][F]{-2}  \ar@{-}[r] \ar@{}[d]|(0.8) *{2}
 & *++[o][F]{-2}  \ar@{-}[r] \ar@{}[d]|(0.8) *{3}
 & *++[o][F]{-2}  \ar@{-}[r] \ar@{}[d]|(0.8) *{4}
 & *++[o][F]{-2}  \ar@{-}[r] \ar@{}[d]|(0.8) *{3}
 & *++[o][F]{-2}  \ar@{}[d]|(0.8) *{2} \\
&&&&&&&&&
}
}
\end{eqnarray*}
\begin{eqnarray*}
\scalebox{0.6}{
\xymatrix@C=12pt@R=12pt{
 &
 &
 &
 &
 &
 & 
 &
 &*++[o][F]{-2}  \ar@{-}[d] \ar@{}[r]|(0.5) *{3}
 &
 &
\\
  *++[o][F]{-2}  \ar@{}[d]|(0.8) *{1}  \ar@{-}[r] 
& *++[o][F]{-2}  \ar@{-}[r] \ar@{}[d]|(0.8) *{2}
 & \cdots \ar@{-}[r]  
 & *++[o][F]{-2}  \ar@{-}[r] \ar@{}[d]|(0.8) *{2}
 & *++[o][F]{-3}  \ar@{-}[r] \ar@{}[d]|(0.8) *{2}
 & *++[o][F]{-2}  \ar@{-}[r] \ar@{}[d]|(0.8) *{3}
 & *++[o][F]{-2}  \ar@{-}[r] \ar@{}[d]|(0.8) *{4}
 & *++[o][F]{-2}  \ar@{-}[r] \ar@{}[d]|(0.8) *{5}
 & *++[o][F]{-2}  \ar@{-}[r] \ar@{}[d]|(0.8) *{6}
 & *++[o][F]{-2}  \ar@{-}[r] \ar@{}[d]|(0.8) *{4}
 & *++[o][F]{-2}  \ar@{}[d]|(0.8) *{2} \\
&&&&&&&&&&&
}
}
\end{eqnarray*}
\begin{eqnarray*}
\scalebox{0.6}{
\xymatrix@C=12pt@R=12pt{
 &
 & 
 &
 &*++[o][F]{-2}  \ar@{-}[d] \ar@{}[r]|(0.5) *{3}
 &
 &
\\
 *++[o][F]{-3}  \ar@{}[d]|(0.8) *{1}  \ar@{-}[r] 
 & *++[o][F]{-2}  \ar@{-}[r] \ar@{}[d]|(0.8) *{3}
 & *++[o][F]{-2}  \ar@{-}[r] \ar@{}[d]|(0.8) *{4}
 & *++[o][F]{-2}  \ar@{-}[r] \ar@{}[d]|(0.8) *{5}
 & *++[o][F]{-2}  \ar@{-}[r] \ar@{}[d]|(0.8) *{6}
 & *++[o][F]{-2}  \ar@{-}[r] \ar@{}[d]|(0.8) *{4}
 & *++[o][F]{-2}  \ar@{}[d]|(0.8) *{2} \\
&&&&&&&
}
}
\end{eqnarray*}
Therefore $e(R(\mathbb P_k^1,D))=3$ when $1\le a\le 2$ or $a=3$ and $b=0$ and $e(R(\mathbb P_k^1,D))=4$ when $a=3$ and  $b\ge 1$.

\vskip.3truecm
\noindent
{\bf Case 5.}
Assume one of the following holds:
\begin{enumerate}
\item[(i)]
$m=1$, $n=2$ and $a\ge 4$.
\item[(ii)]
$m=1$, $n\ge 5$,   $a=1$ and $b\ge 1$.
\item[(iii)]
$m=1$, $n\ge 7$ and  $a\ge 1$.
\item[(iv)]
$m\ge 2$ and  $a\ge 1$.
\item[(v)]
$m\ge 3$, $n\ge 7$ and $a=0$.
\item[(vi)]
$m\ge 3$ and $b\ge 1$.
\end{enumerate}
Then $R(\mathbb{P}^1_k, D)$ does not have a rational singularity because $\deg[5D] \le -2$ in cases (i) and (ii), $\deg[7D] \le -2$ in cases (iii) and (v), $\deg[2D] \le -2$ in case (iv), and $\deg[3D] \le -2$ in case (vi).

\vskip.3truecm
\noindent
{\bf Case 6.}
We assume that $m=1$, $3\le n\le 4$ and $a=1$.
Note that  $D\ge  2P_0 -  \frac{1}{2} P_1-\frac{4}{5} P_2 -\frac{2}{3}P_3$ by Lemma \ref{[[a]]<[[b]]}.
Therefore $R(\mathbb P_k^1,D)$ has a rational singularity by Lemma \ref{bigger coefficient rationality} and  Lemma \ref{(2,2,n)type}.
The dual graphs of the minimal good resolution of $\mathrm{Spec} (R(\mathbb P_k^1,D))$ are the following: 
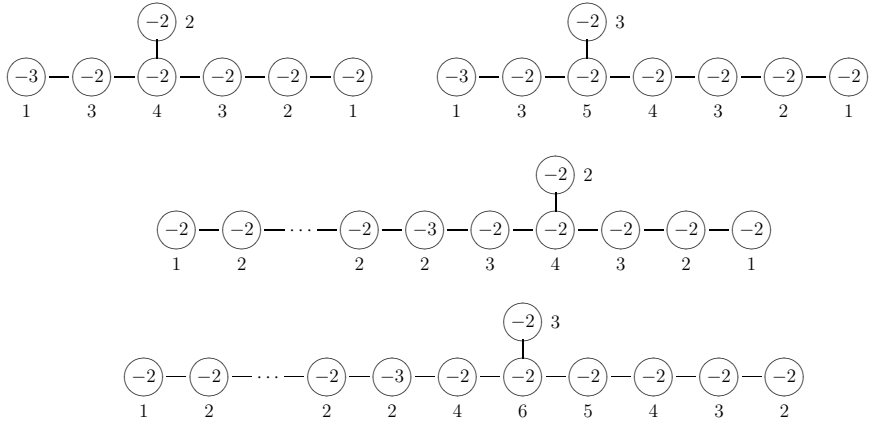
\begin{eqnarray*}
\scalebox{0.6}{
\xymatrix@C=12pt@R=12pt{
 &
 &*++[o][F]{-2}  \ar@{-}[d] \ar@{}[r]|(0.5) *{2}
 &
 & 
 &
\\
  *++[o][F]{-3}  \ar@{}[d]|(0.8) *{1}  \ar@{-}[r] 
 & *++[o][F]{-2}  \ar@{-}[r] \ar@{}[d]|(0.8) *{3}
 & *++[o][F]{-2}  \ar@{-}[r] \ar@{}[d]|(0.8) *{4}
 & *++[o][F]{-2}  \ar@{-}[r] \ar@{}[d]|(0.8) *{3}
 & *++[o][F]{-2}  \ar@{-}[r] \ar@{}[d]|(0.8) *{2}
 & *++[o][F]{-2}  \ar@{}[d]|(0.8) *{1} \\
&&&&&&
}
}\ \ 
\scalebox{0.6}{
\xymatrix@C=12pt@R=12pt{
 &
 &*++[o][F]{-2}  \ar@{-}[d] \ar@{}[r]|(0.5) *{3}
 &
 & 
 &
 &
\\
  *++[o][F]{-3}  \ar@{}[d]|(0.8) *{1}  \ar@{-}[r] 
 & *++[o][F]{-2}  \ar@{-}[r] \ar@{}[d]|(0.8) *{3}
 & *++[o][F]{-2}  \ar@{-}[r] \ar@{}[d]|(0.8) *{5}
 & *++[o][F]{-2}  \ar@{-}[r] \ar@{}[d]|(0.8) *{4}
 & *++[o][F]{-2}  \ar@{-}[r] \ar@{}[d]|(0.8) *{3}
 & *++[o][F]{-2}  \ar@{-}[r] \ar@{}[d]|(0.8) *{2}
 & *++[o][F]{-2}  \ar@{}[d]|(0.8) *{1} \\
&&&&&&&
}
}
\end{eqnarray*}
\[
\scalebox{0.6}{
\xymatrix@C=12pt@R=12pt{
 &
 &
 &
 &
 &
 &*++[o][F]{-2}  \ar@{-}[d] \ar@{}[r]|(0.5) *{2}
 &
 & 
 &
\\
  *++[o][F]{-2}  \ar@{}[d]|(0.8) *{1}  \ar@{-}[r] 
 & *++[o][F]{-2}  \ar@{-}[r] \ar@{}[d]|(0.8) *{2}
 & \cdots \ar@{-}[r]  
 & *++[o][F]{-2}  \ar@{-}[r] \ar@{}[d]|(0.8) *{2}
 & *++[o][F]{-3}  \ar@{-}[r] \ar@{}[d]|(0.8) *{2}
 & *++[o][F]{-2}  \ar@{-}[r] \ar@{}[d]|(0.8) *{3}
 & *++[o][F]{-2}  \ar@{-}[r] \ar@{}[d]|(0.8) *{4}
 & *++[o][F]{-2}  \ar@{-}[r] \ar@{}[d]|(0.8) *{3}
 & *++[o][F]{-2}  \ar@{-}[r] \ar@{}[d]|(0.8) *{2}
 & *++[o][F]{-2}  \ar@{}[d]|(0.8) *{1} \\
&&&&&&&&&&
}
}
\]
\[
\scalebox{0.6}{
\xymatrix@C=12pt@R=12pt{
 &
 &
 &
 &
 &
 &*++[o][F]{-2}  \ar@{-}[d] \ar@{}[r]|(0.5) *{3}
 &
 & 
 &
 &
\\
  *++[o][F]{-2}  \ar@{}[d]|(0.8) *{1}  \ar@{-}[r] 
 & *++[o][F]{-2}  \ar@{-}[r] \ar@{}[d]|(0.8) *{2}
 & \cdots \ar@{-}[r]  
 & *++[o][F]{-2}  \ar@{-}[r] \ar@{}[d]|(0.8) *{2}
 & *++[o][F]{-3}  \ar@{-}[r] \ar@{}[d]|(0.8) *{2}
 & *++[o][F]{-2}  \ar@{-}[r] \ar@{}[d]|(0.8) *{4}
 & *++[o][F]{-2}  \ar@{-}[r] \ar@{}[d]|(0.8) *{6}
 & *++[o][F]{-2}  \ar@{-}[r] \ar@{}[d]|(0.8) *{5}
 & *++[o][F]{-2}  \ar@{-}[r] \ar@{}[d]|(0.8) *{4}
 & *++[o][F]{-2}  \ar@{-}[r] \ar@{}[d]|(0.8) *{3}
 & *++[o][F]{-2}  \ar@{}[d]|(0.8) *{2} \\
&&&&&&&&&&&
}
}
\]
Therefore $e(R(\mathbb P_k^1,D))=3$ when $n=3,4$ and $b=0$ and $e(R(\mathbb P_k^1,D))=4$ when  $n=3,4$ and $b\ge 1$.

\vskip.3truecm
\noindent
{\bf Case 7.}
We assume that $m=1$, $5\le n\le 6$,  $a=1$ and $b=0$.
Let $D'=2P_0 -  \frac{1}{2} P_1-\frac{6}{7} P_2 -\frac{3}{5}P_3$.
Then $\mathrm{deg} [lD']\ge -1$ for any $l\in \mathbb N$
since $\mathrm{deg} [lD']\ge -1$ for $1\le l \le 69$ and $\mathrm{deg} [70D']=3$.
Note that $D\ge D'$ by Lemma \ref{[[a]]<[[b]]}.
Therefore $R(\mathbb P_k^1,D)$ has a rational singularity by Lemma \ref{bigger coefficient rationality}.
The dual graphs of the minimal good resolution of $\mathrm{Spec} (R(\mathbb P_k^1,D))$ are the following: 
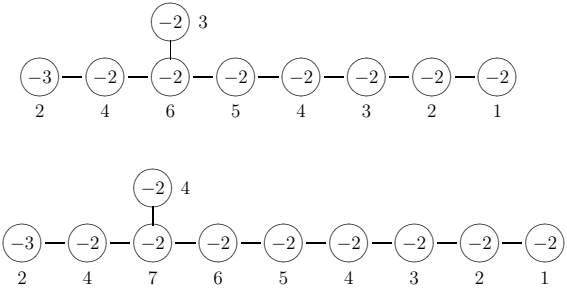
\begin{eqnarray*}
\scalebox{0.6}{
\xymatrix@C=12pt@R=12pt{
 &
 &*++[o][F]{-2}  \ar@{-}[d] \ar@{}[r]|(0.5) *{3}
 &
 & 
 &
 &
 &
\\
   *++[o][F]{-3}  \ar@{}[d]|(0.8) *{2}  \ar@{-}[r] 
 & *++[o][F]{-2}  \ar@{-}[r] \ar@{}[d]|(0.8) *{4}
 & *++[o][F]{-2}  \ar@{-}[r] \ar@{}[d]|(0.8) *{6}
 & *++[o][F]{-2}  \ar@{-}[r] \ar@{}[d]|(0.8) *{5}
 & *++[o][F]{-2}  \ar@{-}[r] \ar@{}[d]|(0.8) *{4}
 & *++[o][F]{-2}  \ar@{-}[r] \ar@{}[d]|(0.8) *{3}
 & *++[o][F]{-2}  \ar@{-}[r] \ar@{}[d]|(0.8) *{2}
 & *++[o][F]{-2}  \ar@{}[d]|(0.8) *{1} \\
&&&&&&&&&
}
}
\end{eqnarray*}
\begin{eqnarray*}
\scalebox{0.6}{
\xymatrix@C=12pt@R=12pt{
 &
 &*++[o][F]{-2}  \ar@{-}[d] \ar@{}[r]|(0.5) *{4}
 &
 & 
 &
 &
 &
 &
\\
  *++[o][F]{-3}  \ar@{}[d]|(0.8) *{2}  \ar@{-}[r] 
 & *++[o][F]{-2}  \ar@{-}[r] \ar@{}[d]|(0.8) *{5}
 & *++[o][F]{-2}  \ar@{-}[r] \ar@{}[d]|(0.8) *{8}
 & *++[o][F]{-2}  \ar@{-}[r] \ar@{}[d]|(0.8) *{7}
 & *++[o][F]{-2}  \ar@{-}[r] \ar@{}[d]|(0.8) *{6}
 & *++[o][F]{-2}  \ar@{-}[r] \ar@{}[d]|(0.8) *{5}
 & *++[o][F]{-2}  \ar@{-}[r] \ar@{}[d]|(0.8) *{4}
 & *++[o][F]{-2}  \ar@{-}[r] \ar@{}[d]|(0.8) *{3}
 & *++[o][F]{-2}  \ar@{}[d]|(0.8) *{2} \\
&&&&&&&&&
}
}
\end{eqnarray*}
Therefore $e(R(\mathbb P_k^1,D))=4$.

\vskip.3truecm
\noindent
{\bf Case 8.}
We assume that $m=2$,   $a=0$ and $b=0$.
Then $R(\mathbb P_k^1,D)$ has a rational singularity since 
$\mathrm{deg} [lD]\ge2l+[-\frac{2l}{3}]-l+[-\frac{l}{3}]=[\frac{l}{3}]+[-\frac{l}{3}]\ge -1$ for any $l\in\mathbb N$.
The dual graph of the minimal good resolution of $\mathrm{Spec} (R(\mathbb P_k^1,D))$ is the following:
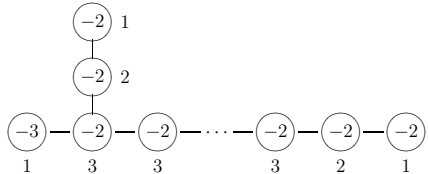
\begin{eqnarray*}
\scalebox{0.6}{
\xymatrix@C=12pt@R=12pt{
 & *++[o][F]{-2}  \ar@{-}[d] \ar@{}[r]|(0.5) *{1}
 &
 &
 & 
\\
 & *++[o][F]{-2}  \ar@{-}[d] \ar@{}[r]|(0.5) *{2}
 &
 &
 &
 &
 & 
\\
  *++[o][F]{-3}  \ar@{}[d]|(0.8) *{1} \ar@{-}[r] 
 & *++[o][F]{-2}  \ar@{-}[r] \ar@{}[d]|(0.8) *{3}
 & *++[o][F]{-2}  \ar@{-}[r] \ar@{}[d]|(0.8) *{3}
 & \cdots \ar@{-}[r] 
 & *++[o][F]{-2}  \ar@{-}[r] \ar@{}[d]|(0.8) *{3}
 & *++[o][F]{-2}  \ar@{-}[r] \ar@{}[d]|(0.8) *{2}
 &  *++[o][F]{-2}  \ar@{}[d]|(0.8) *{1} \\
&&&&&&&
}
}
\end{eqnarray*}
Therefore $e(R(\mathbb P_k^1,D))=3$.

\vskip.3truecm
\noindent
{\bf Case 9.}
We assume that $m=2$, $2\le n\le 4$,   $a=0$ and $b\ge 1$.
Note that  $D\ge  2P_0 -  \frac{2}{3} P_1-\frac{4}{5} P_2 -\frac{1}{2}P_3$ by Lemma \ref{[[a]]<[[b]]}. 
Therefore $R(\mathbb P_k^1,D)$ has a rational singularity by Lemma \ref{bigger coefficient rationality} and  Lemma \ref{(2,2,n)type}.
The dual graphs of the minimal good resolution of $\mathrm{Spec} (R(\mathbb P_k^1,D))$ are the following:
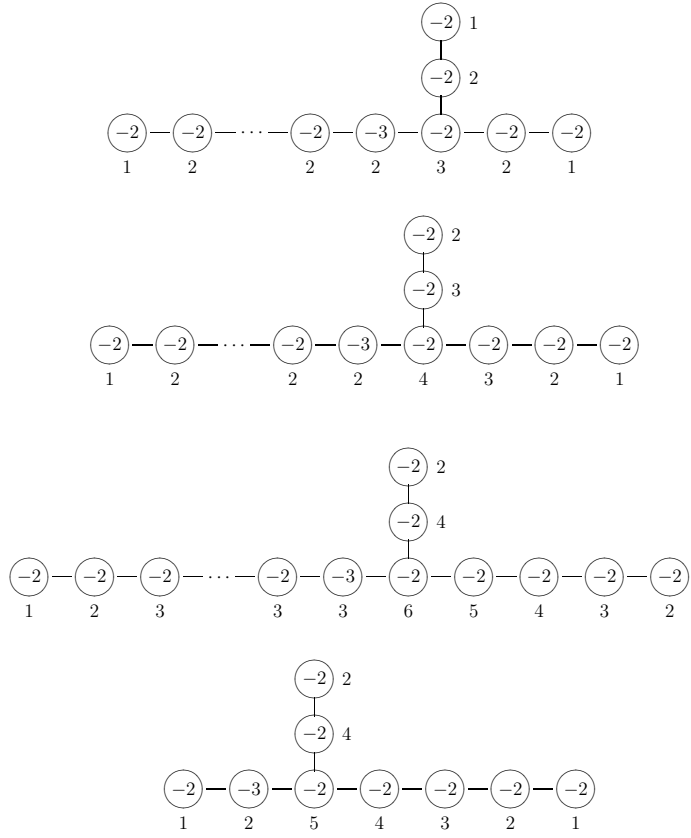
\begin{eqnarray*}
\scalebox{0.6}{
\xymatrix@C=12pt@R=12pt{
 &
 &
 &
 &
 & *++[o][F]{-2}  \ar@{-}[d] \ar@{}[r]|(0.5) *{1}
 &
 &
 & 
\\
 &
 &
 &
 &
 & *++[o][F]{-2}  \ar@{-}[d] \ar@{}[r]|(0.5) *{2}
 &
 &
 &
 &
\\
   *++[o][F]{-2}  \ar@{}[d]|(0.8) *{1} \ar@{-}[r] 
 & *++[o][F]{-2}  \ar@{-}[r] \ar@{}[d]|(0.8) *{2}
 &  \cdots \ar@{-}[r]
 & *++[o][F]{-2}  \ar@{-}[r] \ar@{}[d]|(0.8) *{2}
 & *++[o][F]{-3}  \ar@{-}[r] \ar@{}[d]|(0.8) *{2}
 & *++[o][F]{-2}  \ar@{-}[r] \ar@{}[d]|(0.8) *{3}
 & *++[o][F]{-2}  \ar@{-}[r] \ar@{}[d]|(0.8) *{2}
 &  *++[o][F]{-2}  \ar@{}[d]|(0.8) *{1} \\
&&&&&&&&
}
}
\end{eqnarray*}
\begin{eqnarray*}
\scalebox{0.6}{
\xymatrix@C=12pt@R=12pt{
 &
 &
 &
 &
 & *++[o][F]{-2}  \ar@{-}[d] \ar@{}[r]|(0.5) *{2}
 &
 &
 & 
\\
 &
 &
 &
 &
 & *++[o][F]{-2}  \ar@{-}[d] \ar@{}[r]|(0.5) *{3}
 &
 &
 &
 &
\\
   *++[o][F]{-2}  \ar@{}[d]|(0.8) *{1} \ar@{-}[r] 
 & *++[o][F]{-2}  \ar@{-}[r] \ar@{}[d]|(0.8) *{2}
 &  \cdots \ar@{-}[r]
 & *++[o][F]{-2}  \ar@{-}[r] \ar@{}[d]|(0.8) *{2}
 & *++[o][F]{-3}  \ar@{-}[r] \ar@{}[d]|(0.8) *{2}
 & *++[o][F]{-2}  \ar@{-}[r] \ar@{}[d]|(0.8) *{4}
 & *++[o][F]{-2}  \ar@{-}[r] \ar@{}[d]|(0.8) *{3}
 & *++[o][F]{-2}  \ar@{-}[r] \ar@{}[d]|(0.8) *{2}
 &  *++[o][F]{-2}  \ar@{}[d]|(0.8) *{1} \\
&&&&&&&&&
}
}
\end{eqnarray*}

\begin{eqnarray*}
\scalebox{0.6}{
\xymatrix@C=12pt@R=12pt{
 &
 &
 &
 &
 &
 & *++[o][F]{-2}  \ar@{-}[d] \ar@{}[r]|(0.5) *{2}
 &
 &
 & 
\\
 &
 &
 &
 &
 &
 & *++[o][F]{-2}  \ar@{-}[d] \ar@{}[r]|(0.5) *{4}
 &
 &
 &
 &
\\
   *++[o][F]{-2}  \ar@{}[d]|(0.8) *{1} \ar@{-}[r] 
 & *++[o][F]{-2}  \ar@{-}[r] \ar@{}[d]|(0.8) *{2}
 & *++[o][F]{-2}  \ar@{-}[r] \ar@{}[d]|(0.8) *{3}
 &  \cdots \ar@{-}[r]
 & *++[o][F]{-2}  \ar@{-}[r] \ar@{}[d]|(0.8) *{3}
 & *++[o][F]{-3}  \ar@{-}[r] \ar@{}[d]|(0.8) *{3}
 & *++[o][F]{-2}  \ar@{-}[r] \ar@{}[d]|(0.8) *{6}
 & *++[o][F]{-2}  \ar@{-}[r] \ar@{}[d]|(0.8) *{5}
 & *++[o][F]{-2}  \ar@{-}[r] \ar@{}[d]|(0.8) *{4}
 & *++[o][F]{-2}  \ar@{-}[r] \ar@{}[d]|(0.8) *{3}
 &  *++[o][F]{-2}  \ar@{}[d]|(0.8) *{2} \\
&&&&&&&&&&&&
}
}
\end{eqnarray*}
\begin{eqnarray*}
\scalebox{0.6}{
\xymatrix@C=12pt@R=12pt{
 &
 & *++[o][F]{-2}  \ar@{-}[d] \ar@{}[r]|(0.5) *{2}
 &
 &
 & 
\\
 &
 & *++[o][F]{-2}  \ar@{-}[d] \ar@{}[r]|(0.5) *{4}
 &
 &
 &
 &
\\
   *++[o][F]{-2}  \ar@{}[d]|(0.8) *{1} \ar@{-}[r]
 & *++[o][F]{-3}  \ar@{-}[r] \ar@{}[d]|(0.8) *{2}
 & *++[o][F]{-2}  \ar@{-}[r] \ar@{}[d]|(0.8) *{5}
 & *++[o][F]{-2}  \ar@{-}[r] \ar@{}[d]|(0.8) *{4}
 & *++[o][F]{-2}  \ar@{-}[r] \ar@{}[d]|(0.8) *{3}
 & *++[o][F]{-2}  \ar@{-}[r] \ar@{}[d]|(0.8) *{2}
 &  *++[o][F]{-2}  \ar@{}[d]|(0.8) *{1} \\
&&&&&&
}
}
\end{eqnarray*}
Hence $e(R(\mathbb P_k^1,D))=4$ when $n=2,3$ or $n=4$ and $b=1$ and $e(R(\mathbb P_k^1,D))=5$ when  $n=4$ and $b\ge2$.

\vskip.3truecm
\noindent
{\bf Case 10.}
We assume that $m=2$, $n\ge 5$, $a=0$ and  $b\ge 1$.
Let $E_0$ be the central curve of the minimal good resolution of $\mathrm{Spec} (R(\mathbb P_k^1,D))$ and $E_1$ be the $(-3)$-curve in its dual graph.
Let $Z$ be the fundamental cycle of the minimal good resolution of $\mathrm{Spec} (R(\mathbb P_k^1,D))$.
Then $\mathrm{Coeff}_{E_0}(Z)\ge 6$ by Corollary \ref{n_0=min}.
Therefore  $\mathrm{Coeff}_{E_1}(Z)\ge 3$.
Hence $e(R(\mathbb P_k^1,D))\ge 5$.

\vskip.3truecm
\noindent
{\bf Case 11.}
We assume that $m=3$, $3\le n\le 6$ and  $a=b=0$.
Let $D'=2P_0 -  \frac{3}{4} P_1-\frac{6}{7} P_2 -\frac{1}{3}P_3$.
Then $\mathrm{deg} [lD']\ge -1$ for any $l\in \mathbb N$
since $\mathrm{deg} [lD']\ge -1$ for $1\le l \le 83$ and $\mathrm{deg} [84D']= 5$.
Note that $D\ge D'$ by Lemma \ref{[[a]]<[[b]]}.
Therefore $R(\mathbb P_k^1,D)$ has a rational singularity by Lemma \ref{bigger coefficient rationality}.
The dual graphs of the minimal good resolution of $\mathrm{Spec} (R(\mathbb P_k^1,D))$ are the following:
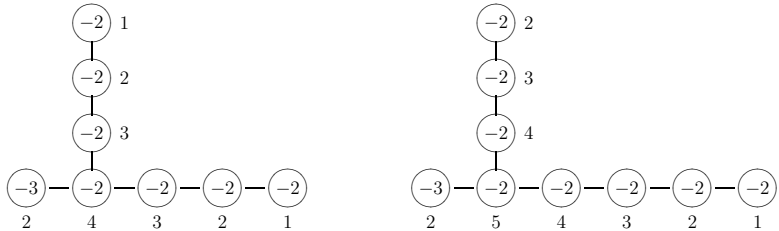
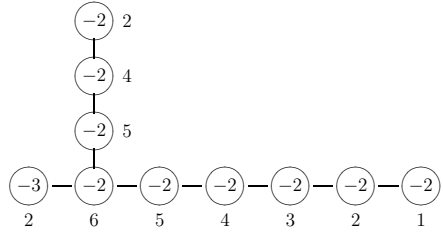
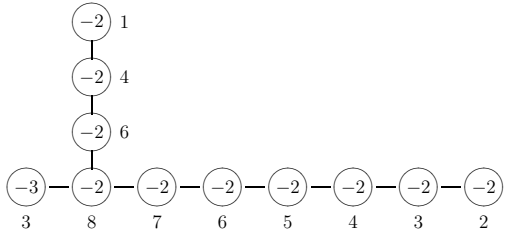
\[
\scalebox{0.6}{
\xymatrix@C=12pt@R=12pt{
& *++[o][F]{-2}  \ar@{-}[d] \ar@{}[r]|(0.5) *{1}
 &
 &
 &
 &
 &
 & 
\\
 & *++[o][F]{-2}  \ar@{-}[d] \ar@{}[r]|(0.5) *{2}
 &
 &
 &
 &
 &
 & 
\\
 & *++[o][F]{-2}  \ar@{-}[d] \ar@{}[r]|(0.5) *{3}
 &
 &
 &
 &
 &
 & 
\\
  *++[o][F]{-3}  \ar@{}[d]|(0.8) *{2} \ar@{-}[r] 
 & *++[o][F]{-2}  \ar@{-}[r] \ar@{}[d]|(0.8) *{4}
 & *++[o][F]{-2}  \ar@{-}[r] \ar@{}[d]|(0.8) *{3}
 & *++[o][F]{-2}  \ar@{-}[r] \ar@{}[d]|(0.8) *{2}
 &  *++[o][F]{-2}  \ar@{}[d]|(0.8) *{1}  \\
&&&&&
}
}
\scalebox{0.6}{
\xymatrix@C=12pt@R=12pt{
 & *++[o][F]{-2}  \ar@{-}[d] \ar@{}[r]|(0.5) *{2}
 &
 &
 &
 &
 &
 & 
\\
 & *++[o][F]{-2}  \ar@{-}[d] \ar@{}[r]|(0.5) *{3}
 &
 &
 &
 &
 &
 & 
\\
 & *++[o][F]{-2}  \ar@{-}[d] \ar@{}[r]|(0.5) *{4}
 &
 &
 &
 &
 &
 & 
\\
  *++[o][F]{-3}  \ar@{}[d]|(0.8) *{2} \ar@{-}[r] 
 & *++[o][F]{-2}  \ar@{-}[r] \ar@{}[d]|(0.8) *{5}
 & *++[o][F]{-2}  \ar@{-}[r] \ar@{}[d]|(0.8) *{4}
 & *++[o][F]{-2}  \ar@{-}[r] \ar@{}[d]|(0.8) *{3}
 & *++[o][F]{-2}  \ar@{-}[r] \ar@{}[d]|(0.8) *{2}
 &  *++[o][F]{-2}  \ar@{}[d]|(0.8) *{1} \\
&&&&&
}
}
\]
\[
\scalebox{0.6}{
\xymatrix@C=12pt@R=12pt{
 & *++[o][F]{-2}  \ar@{-}[d] \ar@{}[r]|(0.5) *{2}
 &
 &
 &
 &
 &
 & 
\\
 & *++[o][F]{-2}  \ar@{-}[d] \ar@{}[r]|(0.5) *{4}
 &
 &
 &
 &
 &
 & 
\\
 & *++[o][F]{-2}  \ar@{-}[d] \ar@{}[r]|(0.5) *{5}
 &
 &
 &
 &
 &
 & 
\\
  *++[o][F]{-3}  \ar@{}[d]|(0.8) *{2} \ar@{-}[r] 
 & *++[o][F]{-2}  \ar@{-}[r] \ar@{}[d]|(0.8) *{6}
 & *++[o][F]{-2}  \ar@{-}[r] \ar@{}[d]|(0.8) *{5}
 & *++[o][F]{-2}  \ar@{-}[r] \ar@{}[d]|(0.8) *{4}
 & *++[o][F]{-2}  \ar@{-}[r] \ar@{}[d]|(0.8) *{3}
 & *++[o][F]{-2}  \ar@{-}[r] \ar@{}[d]|(0.8) *{2}
 &  *++[o][F]{-2}  \ar@{}[d]|(0.8) *{1} \\
&&&&&&&&
}
}
\]
\[
\scalebox{0.6}{
\xymatrix@C=12pt@R=12pt{
& *++[o][F]{-2}  \ar@{-}[d] \ar@{}[r]|(0.5) *{1}
 &
 &
 &
 &
 &
 & 
\\
 & *++[o][F]{-2}  \ar@{-}[d] \ar@{}[r]|(0.5) *{4}
 &
 &
 &
 &
 &
 & 
\\
 & *++[o][F]{-2}  \ar@{-}[d] \ar@{}[r]|(0.5) *{6}
 &
 &
 &
 &
 &
 & 
\\
  *++[o][F]{-3}  \ar@{}[d]|(0.8) *{3} \ar@{-}[r] 
 & *++[o][F]{-2}  \ar@{-}[r] \ar@{}[d]|(0.8) *{8}
 & *++[o][F]{-2}  \ar@{-}[r] \ar@{}[d]|(0.8) *{7}
 & *++[o][F]{-2}  \ar@{-}[r] \ar@{}[d]|(0.8) *{6}
 & *++[o][F]{-2}  \ar@{-}[r] \ar@{}[d]|(0.8) *{5}
 & *++[o][F]{-2}  \ar@{-}[r] \ar@{}[d]|(0.8) *{4}
 & *++[o][F]{-2}  \ar@{-}[r] \ar@{}[d]|(0.8) *{3}
 &  *++[o][F]{-2}  \ar@{}[d]|(0.8) *{2} \\
&&&&&&&
}
}
\]
Therefore $e(R(\mathbb P_k^1,D))=4$ when $3\le n\le 5$ and $e(R(\mathbb P_k^1,D))=5$ when $n=6$.

\end{proof}

\subsection{ The case there is unique $(-4)$-curve }
In this subsection we classify the  $R(\mathbb P_k^1,D)$ with a rational singularity such that 
there is unique $(-4)$-curve    in the dual graph of the minimal good resolution of $\mathrm{Spec}(R(\mathbb P_k^1,D))$ and others are $(-2)$-curves.
First, we consider the case the central curve is a $(-4)$-curve.

\begin{prop}\label{central 4}
Let $D=4P_0-\sum_{i=1}^4a_iP_i$  be an ample $\mathbb Q$-divisor on $\mathbb P_k^1$, where $a_i\in \mathbb Q_{\ge 0}$ and $P_i$ are distinct points of $\mathbb P_k^1$.
Assume that $a_1=\frac{a}{a+1}$,  $a_2=\frac{b}{b+1}$, $a_3=\frac{c}{c+1}$ and $a_4=\frac{d}{d+1}$ for $a,b,c,d\in \mathbb Z_{\ge 0}$.
Then $R(\mathbb P_k^1,D)$ has a rational singularity with $e(R(\mathbb P_k^1,D))=4$. 
\end{prop}

\begin{proof}
$R(\mathbb P_k^1,D)$ has a rational singularity since $\mathrm{deg} [lD]\ge 0$ for any $l\in\mathbb N$. 
The dual graph of the minimal good resolution of $\mathrm{Spec} (R(\mathbb P_k^1,D))$ is the following:
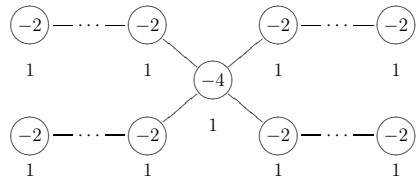
\begin{equation*}
\scalebox{0.6}{
\xymatrix@C=12pt@R=12pt{
  *++[o][F]{-2}  \ar@{}[d]|(0.8) *{1} \ar@{-}[r] 
 &  \cdots \ar@{-}[r]
 &  *++[o][F]{-2}  \ar@{}[d]|(0.8) *{1} \ar@{-}[dr]
 &
 & *++[o][F]{-2}  \ar@{-}[r] \ar@{}[d]|(0.8) *{1} 
 &  \cdots \ar@{-}[r]
 &  *++[o][F]{-2}  \ar@{}[d]|(0.8) *{1}
 \\
&&& *++[o][F]{-4}  \ar@{}[d]|(0.8) *{1} \ar@{-}[dr]  \ar@{-}[ur] &&&\\
 *++[o][F]{-2}  \ar@{}[d]|(0.8) *{1} \ar@{-}[r] 
 &  \cdots \ar@{-}[r]
 &  *++[o][F]{-2}  \ar@{}[d]|(0.8) *{1} \ar@{-}[ur]
 &
& *++[o][F]{-2}  \ar@{-}[r] \ar@{}[d]|(0.8) *{1} 
 &  \cdots \ar@{-}[r]
 &  *++[o][F]{-2}  \ar@{}[d]|(0.8) *{1} \\
&&&&&&&
}
}
\end{equation*}
Therefore $e(R(\mathbb P_k^1,D))=4$.
\end{proof}

Next, we consider the case the central curve is a $(-2)$-curve.
\begin{prop}\label{central 2-4}
Let $D=2P_0-\sum_{i=1}^3a_iP_i$  be an ample $\mathbb Q$-divisor on $\mathbb P_k^1$, where $a_i\in \mathbb Q_{\ge 0}$ and $P_i$ are distinct points of $\mathbb P_k^1$.
Assume that $a_1\le a_2$, $a_1=\frac{m}{m+1}$,  $a_2=\frac{n}{n+1}$,   $\frac{1}{a_3}=[[(2)^a,4, (2)^b]]$ for $m,n,a,b\in \mathbb Z_{\ge 0}$ and $R(\mathbb P_k^1,D)$ has a rational singularity with
 $e(R(\mathbb P_k^1,D))=4$. 
Then $(a_1,a_2,a_3)$ is  one of the following:
\begin{align*}
\mathrm{(1)}&\ (0,\frac{n}{n+1},\frac{(2a+1)b+3a+1}{(2a+3)b+3a+4}) & &\mbox{for}\ n\ge 0, a\ge 0, b\ge 0, \\
\mathrm{(2)}&\ (\frac{1}{2},\frac{n}{n+1},\frac{b+1}{3b+4}) & &\mbox{for}\ n\ge 1,b\ge 0 \\
\mathrm{(3)}&\ (\frac{1}{2},\frac{1}{2},\frac{(2a+1)b+3a+1}{(2a+3)b+3a+4}) & &\mbox{for}\ a\ge 1, b\ge 0, \\
 \mathrm{(4)}&\ (\frac{1}{2},\frac{2}{3},\frac{3b+4}{5b+7})\ \ \mbox{for}\ b\ge 0,& \mathrm{(5)}&\ (\frac{1}{2},\frac{2}{3},\frac{5b+7}{7b+10})\ \ \mbox{for}\ b\ge 0,\\
 \mathrm{(6)}&\ (\frac{1}{2},\frac{2}{3},\frac{7b+10}{9b+13})\ \ \mbox{for}\ b\ge 0,&\mathrm{(7)}&\ (\frac{1}{2},\frac{3}{4},\frac{3b+4}{5b+7})\ \ \mbox{for}\ b\ge 0,\\
\mathrm{(8)}&\ (\frac{1}{2},\frac{4}{5},\frac{3b+4}{5b+7})\ \ \mbox{for}\ b\ge 0,&\mathrm{(9)}&\ (\frac{1}{2},\frac{5}{6},\frac{4}{7}),\\
\mathrm{(10)}&\ (\frac{1}{2},\frac{6}{7},\frac{4}{7}),& \mathrm{(11)}&\ (\frac{2}{3},\frac{n}{n+1},\frac{b+1}{3b+4})\   \mbox{for}\ n\ge 2, b\ge 0,\\
\mathrm{(12)}&\ (\frac{3}{4},\frac{n}{n+1},\frac{1}{4})\ \ \mbox{for}\ n\ge 3.
\end{align*}
\end{prop}

\begin{proof}
\noindent
{\bf Case 1.}
We assume that $m=0$.
Then $R(\mathbb P_k^1,D)$ has a rational singularity since 
 $\mathrm{deg} [lD]\ge 0$ for any $l\in\mathbb N$. 
The dual graph of the minimal good resolution of $\mathrm{Spec} (R(\mathbb P_k^1,D))$ is the following:
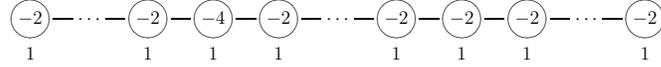
\begin{eqnarray*}
\scalebox{0.6}{
\xymatrix@C=12pt@R=12pt{
  *++[o][F]{-2}  \ar@{}[d]|(0.8) *{1} \ar@{-}[r] 
 & \cdots \ar@{-}[r]  
 & *++[o][F]{-2}  \ar@{-}[r] \ar@{}[d]|(0.8) *{1} 
 & *++[o][F]{-4}  \ar@{-}[r] \ar@{}[d]|(0.8) *{1} 
 & *++[o][F]{-2}  \ar@{-}[r] \ar@{}[d]|(0.8) *{1}
 & \cdots \ar@{-}[r] 
 & *++[o][F]{-2}  \ar@{-}[r] \ar@{}[d]|(0.8) *{1} 
 & *++[o][F]{-2}  \ar@{-}[r] \ar@{}[d]|(0.8) *{1} 
 & *++[o][F]{-2}  \ar@{-}[r] \ar@{}[d]|(0.8) *{1} 
 & \cdots \ar@{-}[r]  
 & *++[o][F]{-2}  \ar@{}[d]|(0.8) *{1} \\
&&&&&&&&&&
}
}
\end{eqnarray*}
Therefore $e(R(\mathbb P_k^1,D))=4$.

\vskip.3truecm
\noindent
{\bf Case 2.}
We assume that $m=1$ and  $a=0$.
Note that   $D\ge  2P_0 -  \frac{1}{2} P_1-\frac{n}{n+1} P_2 -\frac{1}{2}P_3$ by Lemma \ref{[[a]]<[[b]]}. 
Therefore $R(\mathbb P_k^1,D)$ has a rational singularity by Lemma \ref{bigger coefficient rationality} and  Lemma \ref{(2,2,n)type}.
The dual graph of the minimal good resolution of $\mathrm{Spec} (R(\mathbb P_k^1,D))$ is the following:
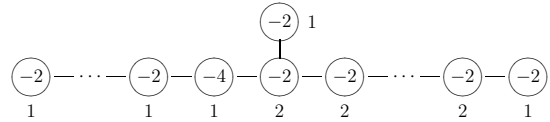
\begin{eqnarray*}
\scalebox{0.6}{
\xymatrix@C=12pt@R=12pt{
 &
 &
 &
 & *++[o][F]{-2}  \ar@{-}[d] \ar@{}[r]|(0.5) *{1}
 &
 &
 & 
\\
   *++[o][F]{-2}  \ar@{}[d]|(0.8) *{1} \ar@{-}[r] 
 &  \cdots \ar@{-}[r]
 & *++[o][F]{-2}  \ar@{-}[r] \ar@{}[d]|(0.8) *{1}
 & *++[o][F]{-4}  \ar@{-}[r] \ar@{}[d]|(0.8) *{1}
 & *++[o][F]{-2}  \ar@{-}[r] \ar@{}[d]|(0.8) *{2}
 & *++[o][F]{-2}  \ar@{-}[r] \ar@{}[d]|(0.8) *{2}
 & \cdots \ar@{-}[r] 
 & *++[o][F]{-2}  \ar@{-}[r] \ar@{}[d]|(0.8) *{2} 
 & *++[o][F]{-2}  \ar@{}[d]|(0.8) *{1}\\
&&&&&&&&
}
}
\end{eqnarray*}
Therefore $e(R(\mathbb P_k^1,D))=4$.

\vskip.3truecm
\noindent
{\bf Case 3.}
We assume that $m=n=1$ and $a\ge 1$.
Note that   $D\ge  2P_0 -  \frac{1}{2} P_1-\frac{1}{2} P_2 -\frac{a+b+1}{a+b+2}P_3$ by Lemma \ref{[[a]]<[[b]]}. 
Therefore $R(\mathbb P_k^1,D)$ has a rational singularity by Lemma \ref{bigger coefficient rationality} and  Lemma \ref{(2,2,n)type}.
The dual graph of the minimal good resolution of $\mathrm{Spec} (R(\mathbb P_k^1,D))$ is the following:
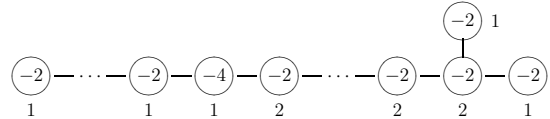
\begin{eqnarray*}
\scalebox{0.6}{
\xymatrix@C=12pt@R=12pt{
 &
 &
 &
 & 
 & 
 &
 & *++[o][F]{-2}  \ar@{-}[d] \ar@{}[r]|(0.5) *{1}
 &
\\
  *++[o][F]{-2}  \ar@{}[d]|(0.8) *{1} \ar@{-}[r] 
 & \cdots \ar@{-}[r]  
 & *++[o][F]{-2}  \ar@{-}[r] \ar@{}[d]|(0.8) *{1} 
 & *++[o][F]{-4}  \ar@{-}[r] \ar@{}[d]|(0.8) *{1} 
 & *++[o][F]{-2}  \ar@{-}[r] \ar@{}[d]|(0.8) *{2}
 & \cdots \ar@{-}[r] 
 & *++[o][F]{-2}  \ar@{-}[r] \ar@{}[d]|(0.8) *{2} 
 & *++[o][F]{-2}  \ar@{-}[r] \ar@{}[d]|(0.8) *{2} 
 & *++[o][F]{-2}  \ar@{}[d]|(0.8) *{1} \\
&&&&&&&&
}
}
\end{eqnarray*}
Therefore $e(R(\mathbb P_k^1,D))=4$.

\vskip.3truecm
\noindent
{\bf Case 4.}
We assume that $m=1$, $n=2$ and $1 \le a\le 3$.
Note that $D\ge  2P_0 -  \frac{1}{2} P_1-\frac{2}{3} P_2 -\frac{4}{5}P_3$ by Lemma \ref{[[a]]<[[b]]}. 
Therefore $R(\mathbb P_k^1,D)$ has a rational singularity by Lemma \ref{bigger coefficient rationality} and  Lemma \ref{(2,2,n)type}.
The dual graphs of the minimal good resolution of $\mathrm{Spec} (R(\mathbb P_k^1,D))$ are the following:
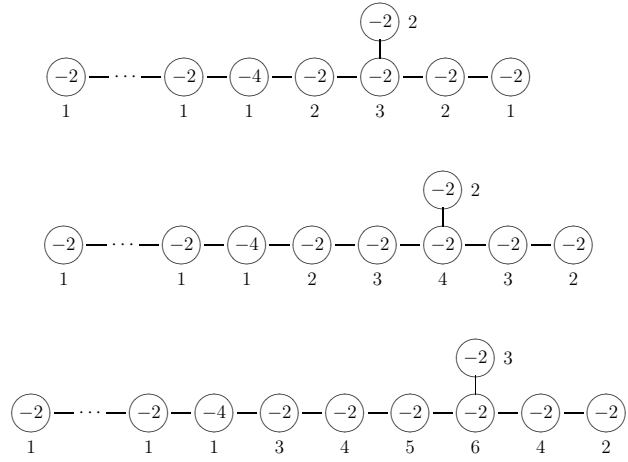
\begin{eqnarray*}
\scalebox{0.6}{
\xymatrix@C=12pt@R=12pt{
 &
 &
 & 
 &
 & *++[o][F]{-2}  \ar@{-}[d] \ar@{}[r]|(0.5) *{2}
 &
 &
\\
  *++[o][F]{-2}  \ar@{}[d]|(0.8) *{1}  \ar@{-}[r] 
 & \cdots \ar@{-}[r]  
 & *++[o][F]{-2}  \ar@{-}[r] \ar@{}[d]|(0.8) *{1}
 & *++[o][F]{-4}  \ar@{-}[r] \ar@{}[d]|(0.8) *{1}
 & *++[o][F]{-2}  \ar@{-}[r] \ar@{}[d]|(0.8) *{2}
 & *++[o][F]{-2}  \ar@{-}[r] \ar@{}[d]|(0.8) *{3}
 & *++[o][F]{-2}  \ar@{-}[r] \ar@{}[d]|(0.8) *{2}
 & *++[o][F]{-2}  \ar@{}[d]|(0.8) *{1} \\
&&&&&&&&&
}
}
\end{eqnarray*}
\begin{eqnarray*}
\scalebox{0.6}{
\xymatrix@C=12pt@R=12pt{
 &
 &
 &
 & 
 &
 & *++[o][F]{-2}  \ar@{-}[d] \ar@{}[r]|(0.5) *{2}
 &
 &
\\
  *++[o][F]{-2}  \ar@{}[d]|(0.8) *{1}  \ar@{-}[r] 
 & \cdots \ar@{-}[r]  
 & *++[o][F]{-2}  \ar@{-}[r] \ar@{}[d]|(0.8) *{1}
 & *++[o][F]{-4}  \ar@{-}[r] \ar@{}[d]|(0.8) *{1}
 & *++[o][F]{-2}  \ar@{-}[r] \ar@{}[d]|(0.8) *{2}
 & *++[o][F]{-2}  \ar@{-}[r] \ar@{}[d]|(0.8) *{3}
 & *++[o][F]{-2}  \ar@{-}[r] \ar@{}[d]|(0.8) *{4}
 & *++[o][F]{-2}  \ar@{-}[r] \ar@{}[d]|(0.8) *{3}
 & *++[o][F]{-2}  \ar@{}[d]|(0.8) *{2} \\
&&&&&&&&
}
}
\end{eqnarray*}
\begin{eqnarray*}
\scalebox{0.6}{
\xymatrix@C=12pt@R=12pt{
 &
 &
 &
 &
 & 
 &
 & *++[o][F]{-2}  \ar@{-}[d] \ar@{}[r]|(0.5) *{3}
 &
 &
\\
  *++[o][F]{-2}  \ar@{}[d]|(0.8) *{1}  \ar@{-}[r] 
 & \cdots \ar@{-}[r]  
 & *++[o][F]{-2}  \ar@{-}[r] \ar@{}[d]|(0.8) *{1}
 & *++[o][F]{-4}  \ar@{-}[r] \ar@{}[d]|(0.8) *{1}
 & *++[o][F]{-2}  \ar@{-}[r] \ar@{}[d]|(0.8) *{3}
 & *++[o][F]{-2}  \ar@{-}[r] \ar@{}[d]|(0.8) *{4}
 & *++[o][F]{-2}  \ar@{-}[r] \ar@{}[d]|(0.8) *{5}
 & *++[o][F]{-2}  \ar@{-}[r] \ar@{}[d]|(0.8) *{6}
 & *++[o][F]{-2}  \ar@{-}[r] \ar@{}[d]|(0.8) *{4}
 & *++[o][F]{-2}  \ar@{}[d]|(0.8) *{2} \\
&&&&&&&&&
}
}
\end{eqnarray*}
Therefore $e(R(\mathbb P_k^1,D))=4$.

\vskip.3truecm
\noindent
{\bf Case 5.}
Assume one of the following holds:
\begin{enumerate}
\item[(i)]
$m=1$, $n=2$ and $a\ge 4$.
\item[(ii)]
$m=1$, $n\ge 3$ and  $a\ge 2$.
\item[(iii)]
$m=1$, $n\ge 7$,  $a\ge 1$ and $b\ge 1$.
\item[(iv)]
$m=1$, $n\ge 9$  and $a\ge 1$.
\item[(v)]
$m\ge 2$ and $a\ge 1$.
\end{enumerate}
Then $R(\mathbb{P}^1_k, D)$ does not have a rational singularity because $\deg[5D] \le -2$ in case (i), $\deg[3D] \le -2$ in case (ii), $\deg[7D] \le -2$ in case (iii), $\deg[9D] \le -2$ in case (iv) and $\deg[2D] \le -2$ in case (v).

\vskip.3truecm
\noindent
{\bf Case 6.}
We assume that $m=1$, $3\le n\le 6$ and $a=1$.
Note that $\frac{5}{3}=[[2,3]]$ and $D\ge  2P_0 -  \frac{1}{2} P_1-\frac{6}{7} P_2 -\frac{3}{5}P_3$ by Lemma \ref{[[a]]<[[b]]}.
Therefore $R(\mathbb P_k^1,D)$ has a rational singularity by Lemma \ref{bigger coefficient rationality} and  Proposition \ref{central 2-3}(ii)(5).
The dual graphs of the minimal good resolution of $\mathrm{Spec} (R(\mathbb P_k^1,D))$ are the following: 
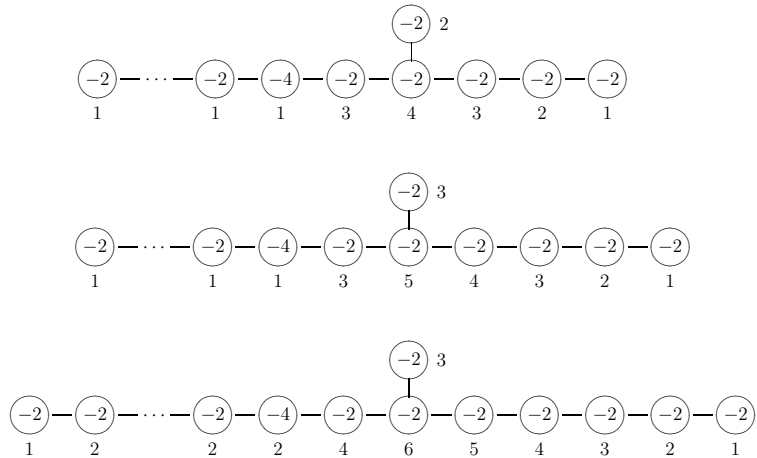
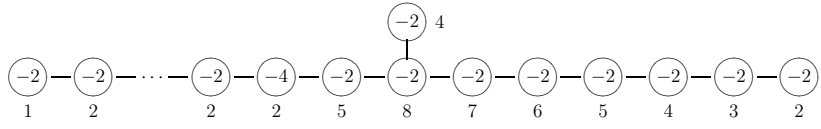
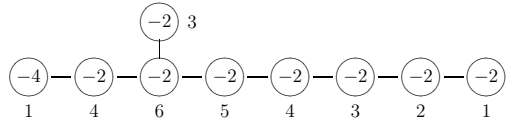
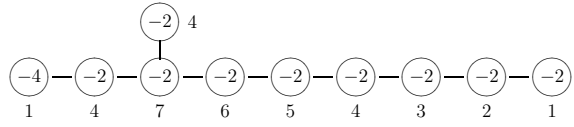
\begin{eqnarray*}
\scalebox{0.6}{
\xymatrix@C=12pt@R=12pt{
 &
 &
 &
 &
 & *++[o][F]{-2}  \ar@{-}[d] \ar@{}[r]|(0.5) *{2}
 &
 & 
 &
\\
  *++[o][F]{-2}  \ar@{}[d]|(0.8) *{1}  \ar@{-}[r] 
 & \cdots \ar@{-}[r]  
 & *++[o][F]{-2}  \ar@{-}[r] \ar@{}[d]|(0.8) *{1}
 & *++[o][F]{-4}  \ar@{-}[r] \ar@{}[d]|(0.8) *{1}
 & *++[o][F]{-2}  \ar@{-}[r] \ar@{}[d]|(0.8) *{3}
 & *++[o][F]{-2}  \ar@{-}[r] \ar@{}[d]|(0.8) *{4}
 & *++[o][F]{-2}  \ar@{-}[r] \ar@{}[d]|(0.8) *{3}
 & *++[o][F]{-2}  \ar@{-}[r] \ar@{}[d]|(0.8) *{2}
 & *++[o][F]{-2}  \ar@{}[d]|(0.8) *{1} \\
&&&&&&&&&&
}
}
\end{eqnarray*}
\begin{eqnarray*}
\scalebox{0.6}{
\xymatrix@C=12pt@R=12pt{
 &
 &
 &
 &
 & *++[o][F]{-2}  \ar@{-}[d] \ar@{}[r]|(0.5) *{3}
 &
 & 
 &
 &
\\
  *++[o][F]{-2}  \ar@{}[d]|(0.8) *{1}  \ar@{-}[r] 
 & \cdots \ar@{-}[r]  
 & *++[o][F]{-2}  \ar@{-}[r] \ar@{}[d]|(0.8) *{1}
 & *++[o][F]{-4}  \ar@{-}[r] \ar@{}[d]|(0.8) *{1}
 & *++[o][F]{-2}  \ar@{-}[r] \ar@{}[d]|(0.8) *{3}
 & *++[o][F]{-2}  \ar@{-}[r] \ar@{}[d]|(0.8) *{5}
 & *++[o][F]{-2}  \ar@{-}[r] \ar@{}[d]|(0.8) *{4}
 & *++[o][F]{-2}  \ar@{-}[r] \ar@{}[d]|(0.8) *{3}
 & *++[o][F]{-2}  \ar@{-}[r] \ar@{}[d]|(0.8) *{2}
 & *++[o][F]{-2}  \ar@{}[d]|(0.8) *{1} \\
&&&&&&&&&
}
}
\end{eqnarray*}
\begin{eqnarray*}
\scalebox{0.6}{
\xymatrix@C=12pt@R=12pt{
 &
 &
 &
 &
 &
 & *++[o][F]{-2}  \ar@{-}[d] \ar@{}[r]|(0.5) *{3}
 &
 & 
 &
 &
 &
\\
  *++[o][F]{-2}  \ar@{}[d]|(0.8) *{1}  \ar@{-}[r] 
 & *++[o][F]{-2}  \ar@{-}[r] \ar@{}[d]|(0.8) *{2}
 & \cdots \ar@{-}[r]  
 & *++[o][F]{-2}  \ar@{-}[r] \ar@{}[d]|(0.8) *{2}
 & *++[o][F]{-4}  \ar@{-}[r] \ar@{}[d]|(0.8) *{2} 
 & *++[o][F]{-2}  \ar@{-}[r] \ar@{}[d]|(0.8) *{4}
 & *++[o][F]{-2}  \ar@{-}[r] \ar@{}[d]|(0.8) *{6}
 & *++[o][F]{-2}  \ar@{-}[r] \ar@{}[d]|(0.8) *{5}
 & *++[o][F]{-2}  \ar@{-}[r] \ar@{}[d]|(0.8) *{4}
 & *++[o][F]{-2}  \ar@{-}[r] \ar@{}[d]|(0.8) *{3}
 & *++[o][F]{-2}  \ar@{-}[r] \ar@{}[d]|(0.8) *{2}
 & *++[o][F]{-2}  \ar@{}[d]|(0.8) *{1} \\
&&&&&&&&&&&
}
}
\end{eqnarray*}
\begin{eqnarray*}
\scalebox{0.6}{
\xymatrix@C=12pt@R=12pt{
 &
 &
 &
 &
 &
 & *++[o][F]{-2}  \ar@{-}[d] \ar@{}[r]|(0.5) *{4}
 &
 & 
 &
 &
 &
 &
\\
  *++[o][F]{-2}  \ar@{}[d]|(0.8) *{1}  \ar@{-}[r] 
 & *++[o][F]{-2}  \ar@{-}[r] \ar@{}[d]|(0.8) *{2}
 & \cdots \ar@{-}[r]  
 & *++[o][F]{-2}  \ar@{-}[r] \ar@{}[d]|(0.8) *{2}
 & *++[o][F]{-4}  \ar@{-}[r] \ar@{}[d]|(0.8) *{2} 
 & *++[o][F]{-2}  \ar@{-}[r] \ar@{}[d]|(0.8) *{5}
 & *++[o][F]{-2}  \ar@{-}[r] \ar@{}[d]|(0.8) *{8}
 & *++[o][F]{-2}  \ar@{-}[r] \ar@{}[d]|(0.8) *{7}
 & *++[o][F]{-2}  \ar@{-}[r] \ar@{}[d]|(0.8) *{6}
 & *++[o][F]{-2}  \ar@{-}[r] \ar@{}[d]|(0.8) *{5}
 & *++[o][F]{-2}  \ar@{-}[r] \ar@{}[d]|(0.8) *{4}
 & *++[o][F]{-2}  \ar@{-}[r] \ar@{}[d]|(0.8) *{3}
 & *++[o][F]{-2}  \ar@{}[d]|(0.8) *{2} \\
&&&&&&&&&&&&
}
}
\end{eqnarray*}
\begin{eqnarray*}
\scalebox{0.6}{
\xymatrix@C=12pt@R=12pt{
 &
 & *++[o][F]{-2}  \ar@{-}[d] \ar@{}[r]|(0.5) *{3}
 &
 & 
 &
 &
 &
\\
   *++[o][F]{-4}  \ar@{}[d]|(0.8) *{1}  \ar@{-}[r] 
 & *++[o][F]{-2}  \ar@{-}[r] \ar@{}[d]|(0.8) *{4}
 & *++[o][F]{-2}  \ar@{-}[r] \ar@{}[d]|(0.8) *{6}
 & *++[o][F]{-2}  \ar@{-}[r] \ar@{}[d]|(0.8) *{5}
 & *++[o][F]{-2}  \ar@{-}[r] \ar@{}[d]|(0.8) *{4}
 & *++[o][F]{-2}  \ar@{-}[r] \ar@{}[d]|(0.8) *{3}
 & *++[o][F]{-2}  \ar@{-}[r] \ar@{}[d]|(0.8) *{2}
 & *++[o][F]{-2}  \ar@{}[d]|(0.8) *{1} \\
&&&&&&&&&
}
}
\end{eqnarray*}
\begin{eqnarray*}
\scalebox{0.6}{
\xymatrix@C=12pt@R=12pt{
 &
 & *++[o][F]{-2}  \ar@{-}[d] \ar@{}[r]|(0.5) *{4}
 &
 & 
 &
 &
 &
 &
\\
  *++[o][F]{-4}  \ar@{}[d]|(0.8) *{1}  \ar@{-}[r] 
 & *++[o][F]{-2}  \ar@{-}[r] \ar@{}[d]|(0.8) *{4}
 & *++[o][F]{-2}  \ar@{-}[r] \ar@{}[d]|(0.8) *{7}
 & *++[o][F]{-2}  \ar@{-}[r] \ar@{}[d]|(0.8) *{6}
 & *++[o][F]{-2}  \ar@{-}[r] \ar@{}[d]|(0.8) *{5}
 & *++[o][F]{-2}  \ar@{-}[r] \ar@{}[d]|(0.8) *{4}
 & *++[o][F]{-2}  \ar@{-}[r] \ar@{}[d]|(0.8) *{3}
 & *++[o][F]{-2}  \ar@{-}[r] \ar@{}[d]|(0.8) *{2}
 & *++[o][F]{-2}  \ar@{}[d]|(0.8) *{1} \\
&&&&&&&&
}
}
\end{eqnarray*}
Therefore $e(R(\mathbb P_k^1,D))=4$ when $n=3,4$ or $n=5,6$ and $b=0$ and $e(R(\mathbb P_k^1,D))=6$ when  $n=5,6$ and $b\ge 1$.

\vskip.3truecm
\noindent
{\bf Case 7.}
Assume one of the following holds:
\begin{enumerate}
\item[(i)]
$m=1$, $7\le n\le 8$,  $a=1$ and $b=0$.
\item[(ii)]
$m=3$, $a=0$ and  $b\ge 1$.
\item[(iii)]
$m\ge 4$ and $a=0$. 
\end{enumerate}
In this case, we have  $e(R(\mathbb P_k^1,D))\ge 6$.
We consider only case (i)  since we can apply the same argument to cases (ii) and (iii).
Let $E_0$ be the central curve of the minimal good resolution of $\mathrm{Spec} (R(\mathbb P_k^1,D))$ and $E_1$ be the $(-4)$-curve in its dual graph.
Let $Z$ be the fundamental cycle of the minimal good resolution of $\mathrm{Spec} (R(\mathbb P_k^1,D))$.
We have $\mathrm{Coeff}_{E_0}(Z)\ge 8$  by Corollary \ref{n_0=min}.
Therefore  $\mathrm{Coeff}_{E_1}(Z)\ge 2$ by Lemma \ref{formula}.
Hence $e(R(\mathbb P_k^1,D))\ge 6$.

\vskip.3truecm
\noindent
{\bf Case 8.}
We assume that $m=2$ and  $a=0$.
Note that $D\ge  2P_0 -  \frac{2}{3} P_1-\frac{n}{n+1} P_2 -\frac{1}{3}P_3$ by Lemma \ref{[[a]]<[[b]]}. 
Therefore $R(\mathbb P_k^1,D)$ has a rational singularity by Lemma \ref{bigger coefficient rationality} and  Proposition \ref{central 2-3}(i)(9).
The dual graph of the minimal good resolution of $\mathrm{Spec} (R(\mathbb P_k^1,D))$ is the following:
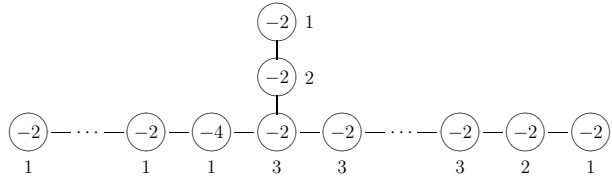
\begin{eqnarray*}
\scalebox{0.6}{
\xymatrix@C=12pt@R=12pt{
 &
 &
 &
 & *++[o][F]{-2}  \ar@{-}[d] \ar@{}[r]|(0.5) *{1}
 &
 &
 & 
\\
 &
 &
 &
 & *++[o][F]{-2}  \ar@{-}[d] \ar@{}[r]|(0.5) *{2}
 &
 &
 &
 &
 & 
\\
   *++[o][F]{-2}  \ar@{}[d]|(0.8) *{1} \ar@{-}[r] 
 &  \cdots \ar@{-}[r]
 & *++[o][F]{-2}  \ar@{-}[r] \ar@{}[d]|(0.8) *{1}
 & *++[o][F]{-4}  \ar@{-}[r] \ar@{}[d]|(0.8) *{1}
 & *++[o][F]{-2}  \ar@{-}[r] \ar@{}[d]|(0.8) *{3}
 & *++[o][F]{-2}  \ar@{-}[r] \ar@{}[d]|(0.8) *{3}
 & \cdots \ar@{-}[r] 
 & *++[o][F]{-2}  \ar@{-}[r] \ar@{}[d]|(0.8) *{3}
 & *++[o][F]{-2}  \ar@{-}[r] \ar@{}[d]|(0.8) *{2}
 &  *++[o][F]{-2}  \ar@{}[d]|(0.8) *{1} \\
&&&&&&&&&
}
}
\end{eqnarray*}
Therefore $e(R(\mathbb P_k^1,D))=4$.

\vskip.3truecm
\noindent
{\bf Case 9.}
We assume that $m=3$ and  $a=b=0$.
Then $R(\mathbb P_k^1,D)$ has a rational singularity since 
$\mathrm{deg} [lD]\ge2l+[-\frac{3l}{4}]-l+[-\frac{l}{4}]=[\frac{l}{4}]+[-\frac{l}{4}]\ge -1$ for any $l\in\mathbb N$.
The dual graph of the minimal good resolution of $\mathrm{Spec} (R(\mathbb P_k^1,D))$ is the following:
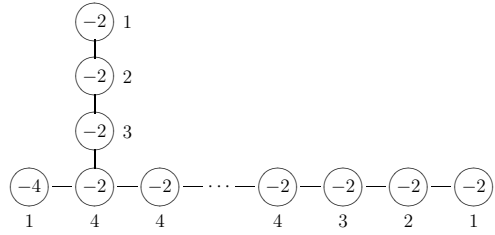
\begin{eqnarray*}
\scalebox{0.6}{
\xymatrix@C=12pt@R=12pt{
 & *++[o][F]{-2}  \ar@{-}[d] \ar@{}[r]|(0.5) *{1}
 &
 &
 &
 &
 &
 & 
\\
 & *++[o][F]{-2}  \ar@{-}[d] \ar@{}[r]|(0.5) *{2}
 &
 &
 &
 &
 &
 & 
\\
 & *++[o][F]{-2}  \ar@{-}[d] \ar@{}[r]|(0.5) *{3}
 &
 &
 &
 &
 &
 & 
\\
  *++[o][F]{-4}  \ar@{}[d]|(0.8) *{1} \ar@{-}[r] 
 & *++[o][F]{-2}  \ar@{-}[r] \ar@{}[d]|(0.8) *{4}
 & *++[o][F]{-2}  \ar@{-}[r] \ar@{}[d]|(0.8) *{4}
 & \cdots \ar@{-}[r] 
 & *++[o][F]{-2}  \ar@{-}[r] \ar@{}[d]|(0.8) *{4}
 & *++[o][F]{-2}  \ar@{-}[r] \ar@{}[d]|(0.8) *{3}
 & *++[o][F]{-2}  \ar@{-}[r] \ar@{}[d]|(0.8) *{2}
 &  *++[o][F]{-2}  \ar@{}[d]|(0.8) *{1} \\
&&&&&&&
}
}
\end{eqnarray*}
Therefore $e(R(\mathbb P_k^1,D))=4$.

\end{proof}

\subsection{The case there are two $(-3)$-curves}
In this subsection we classify the  $R(\mathbb P_k^1,D)$ with a rational singularity such that 
there are two $(-3)$-curves   in the dual graph of the minimal good resolution of $\mathrm{Spec}(R(\mathbb P_k^1,D))$ and others are $(-2)$-curves.
First, we consider the case the central curve is a $(-3)$-curve.

\begin{prop}\label{central 3-3}
Let $D=3P_0-\sum_{i=1}^3a_iP_i$  be an ample $\mathbb Q$-divisor on $\mathbb P_k^1$, where $a_i\in \mathbb Q_{\ge 0}$ and $P_i$ are distinct points of $\mathbb P_k^1$.
Assume that $a_1=\frac{m}{m+1}$,  $a_2=\frac{n}{n+1}$, $\frac{1}{a_3}=[[(2)^a,3, (2)^b]]$ for $m,n,a,b\in \mathbb Z_{\ge 0}$.
Then $R(\mathbb P_k^1,D)$ has a rational singularity with $e(R(\mathbb P_k^1,D))=4$. 
\end{prop}

\begin{proof}
$R(\mathbb P_k^1,D)$ has a rational singularity since $\mathrm{deg} [lD]\ge 0$ for any $l\in\mathbb N$. 
The dual graph of the minimal good resolution of $\mathrm{Spec} (R(\mathbb P_k^1,D))$ is the following:
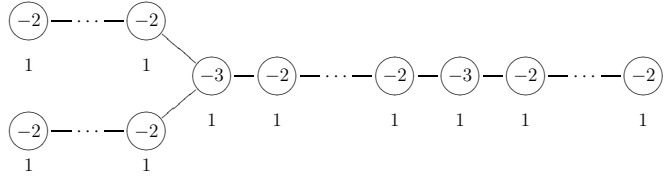
\begin{equation*}
\scalebox{0.6}{
\xymatrix@C=12pt@R=12pt{
  *++[o][F]{-2}  \ar@{}[d]|(0.8) *{1} \ar@{-}[r] 
 &  \cdots \ar@{-}[r]
 &  *++[o][F]{-2}  \ar@{}[d]|(0.8) *{1} \ar@{-}[dr]\\
&&& *++[o][F]{-3}  \ar@{-}[r] \ar@{}[d]|(0.8) *{1} & *++[o][F]{-2}  \ar@{-}[r] \ar@{}[d]|(0.8) *{1} & \cdots\ar@{-}[r] & *++[o][F]{-2}  \ar@{-}[r] \ar@{}[d]|(0.8) *{1}& *++[o][F]{-3}  \ar@{-}[r] \ar@{}[d]|(0.8) *{1} & *++[o][F]{-2}  \ar@{-}[r] \ar@{}[d]|(0.8) *{1} & \cdots\ar@{-}[r] & *++[o][F]{-2}  \ar@{}[d]|(0.8) *{1}\\
 *++[o][F]{-2}  \ar@{}[d]|(0.8) *{1} \ar@{-}[r] 
 &  \cdots \ar@{-}[r]
 &  *++[o][F]{-2}  \ar@{}[d]|(0.8) *{1} \ar@{-}[ur]  &&&&&&&&\\
&&&&&&&&&&&
}
}
\end{equation*}
Therefore $e(R(\mathbb P_k^1,D))=4$.
\end{proof}

Next, we consider the case the central curve is a $(-2)$-curve and there are two $(-3)$-curves in one branch.
\begin{prop}\label{central 2-3-3 1}
Let $D=2P_0-\sum_{i=1}^3a_iP_i$  be an ample $\mathbb Q$-divisor on $\mathbb P_k^1$, where $a_i\in \mathbb Q_{\ge 0}$ and $P_i$ are distinct points of $\mathbb P_k^1$.
Assume that $a_1\le a_2$, $a_1=\frac{m}{m+1}$,  $a_2=\frac{n}{n+1}$, $\frac{1}{a_3}=[[(2)^a,3, (2)^b,3,(2)^c]]$ for $m,n,a,b,c\in \mathbb Z_{\ge 0}$ and $R(\mathbb P_k^1,D)$ has a rational singularity with
 $e(R(\mathbb P_k^1,D))=4$. 
Then $(a_1,a_2,a_3)$ is  one of the following:
\begin{align*}
\mathrm{(1)}&\ (0,\frac{n}{n+1},\frac{\big((a+1)b + 3a+2\big)c+(2a+2)b+5a+3}{\big((a+2)b+ 3a+5\big)c+(2a+4)b+5a+8})\  \mbox{for}\ n,  a , b,c\ge 0, \\
\mathrm{(2)}&\ (\frac{1}{2},\frac{n}{n+1},\frac{(b +2)c+2b+3}{(2b+ 5)c+4b+8})\ \mbox{for}\ n\ge 1,b\ge 0,c\ge 0, \\
\mathrm{(3)}&\ (\frac{1}{2},\frac{1}{2},\frac{\big((a+1)b + 3a+2\big)c+(2a+2)b+5a+3}{\big((a+2)b+ 3a+5\big)c+(2a+4)b+5a+8}) \ \mbox{for}\ a\ge 1, b\ge 0,c\ge 0, \\
 \mathrm{(4)}&\ (\frac{1}{2},\frac{2}{3},\frac{(2b + 5)c+4b+8}{(3b+ 8)c+6b+13})\ \mbox{for}\ b\ge 0,c\ge 0,\\
 \mathrm{(5)}&\ (\frac{1}{2},\frac{2}{3},\frac{(3b + 8)c+6b+13}{(4b+ 11)c+8b+18})\ \mbox{for}\ b\ge 0, c\ge 0.
\end{align*}
\end{prop}

\begin{proof}
\noindent
{\bf Case 1.}
We assume that $m=0$.
Then $R(\mathbb P_k^1,D)$ has a rational singularity since 
 $\mathrm{deg} [lD]\ge 0$ for any $l\in\mathbb N$.  
The dual graph of the minimal good resolution of $\mathrm{Spec} (R(\mathbb P_k^1,D))$ is the following:
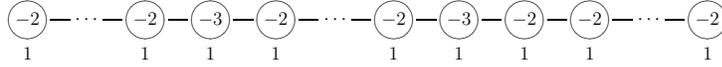
\begin{eqnarray*}
\scalebox{0.6}{
\xymatrix@C=12pt@R=12pt{
   *++[o][F]{-2}  \ar@{}[d]|(0.8) *{1} \ar@{-}[r] 
 & \cdots \ar@{-}[r]  
 & *++[o][F]{-2}  \ar@{-}[r] \ar@{}[d]|(0.8) *{1} 
 & *++[o][F]{-3}  \ar@{-}[r] \ar@{}[d]|(0.8) *{1} 
 & *++[o][F]{-2}  \ar@{-}[r] \ar@{}[d]|(0.8) *{1} 
 &  \cdots \ar@{-}[r]
 & *++[o][F]{-2}  \ar@{-}[r] \ar@{}[d]|(0.8) *{1}
 & *++[o][F]{-3}  \ar@{-}[r] \ar@{}[d]|(0.8) *{1}
 & *++[o][F]{-2}  \ar@{-}[r] \ar@{}[d]|(0.8) *{1}
 & *++[o][F]{-2}  \ar@{-}[r] \ar@{}[d]|(0.8) *{1}
 & \cdots \ar@{-}[r] 
 &  *++[o][F]{-2}  \ar@{}[d]|(0.8) *{1} \\
&&&&&&&&&&&&
}
}
\end{eqnarray*}
Therefore $e(R(\mathbb P_k^1,D))=4$.

\vskip.3truecm
\noindent
{\bf Case 2.}
We assume that $m=1$ and  $a=0$.
Note that $D\ge  2P_0 -  \frac{1}{2} P_1-\frac{n}{n+1} P_2 -\frac{1}{2}P_3$ by Lemma \ref{[[a]]<[[b]]}. 
Therefore $R(\mathbb P_k^1,D)$ has a rational singularity by Lemma \ref{bigger coefficient rationality} and  Lemma \ref{(2,2,n)type}.
The dual graph of the minimal good resolution of $\mathrm{Spec} (R(\mathbb P_k^1,D))$ is the following:
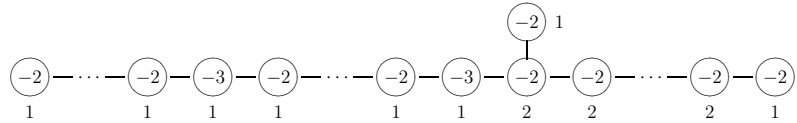
\begin{eqnarray*}
\scalebox{0.6}{
\xymatrix@C=12pt@R=12pt{
 &
 &
 &
 &
 &
 & 
 &
 & *++[o][F]{-2}  \ar@{-}[d] \ar@{}[r]|(0.5) *{1}
 &
 &
 & 
\\
   *++[o][F]{-2}  \ar@{}[d]|(0.8) *{1} \ar@{-}[r] 
 & \cdots \ar@{-}[r]  
 & *++[o][F]{-2}  \ar@{-}[r] \ar@{}[d]|(0.8) *{1} 
 & *++[o][F]{-3}  \ar@{-}[r] \ar@{}[d]|(0.8) *{1} 
 & *++[o][F]{-2}  \ar@{-}[r] \ar@{}[d]|(0.8) *{1} 
 &  \cdots \ar@{-}[r]
 & *++[o][F]{-2}  \ar@{-}[r] \ar@{}[d]|(0.8) *{1}
 & *++[o][F]{-3}  \ar@{-}[r] \ar@{}[d]|(0.8) *{1}
 & *++[o][F]{-2}  \ar@{-}[r] \ar@{}[d]|(0.8) *{2}
 & *++[o][F]{-2}  \ar@{-}[r] \ar@{}[d]|(0.8) *{2}
 & \cdots \ar@{-}[r] 
 & *++[o][F]{-2}  \ar@{-}[r] \ar@{}[d]|(0.8) *{2}
 &  *++[o][F]{-2}  \ar@{}[d]|(0.8) *{1} \\
&&&&&&&&&&&&&
}
}
\end{eqnarray*}
Therefore $e(R(\mathbb P_k^1,D))=4$.

\vskip.3truecm
\noindent
{\bf Case 3.}
We assume that $m=n=1$ and $a\ge 1$.
Note that $D\ge  2P_0 -  \frac{1}{2} P_1-\frac{1}{2} P_2 -\frac{a+b+c+2}{a+b+c+3}P_3$ by Lemma \ref{[[a]]<[[b]]}. 
Therefore $R(\mathbb P_k^1,D)$ has a rational singularity by Lemma \ref{bigger coefficient rationality} and  Lemma \ref{(2,2,n)type}.
The dual graph of the minimal good resolution of $\mathrm{Spec} (R(\mathbb P_k^1,D))$ is the following:
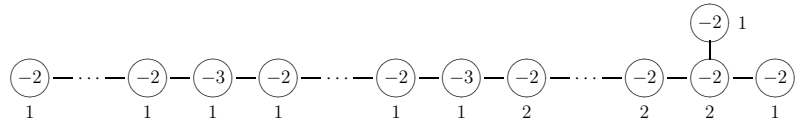
\begin{eqnarray*}
\scalebox{0.6}{
\xymatrix@C=12pt@R=12pt{
 &
 &
 &
 &
 &
 &
 &
 & 
 & 
 &
 & *++[o][F]{-2}  \ar@{-}[d] \ar@{}[r]|(0.5) *{1}
 &
\\
  *++[o][F]{-2}  \ar@{}[d]|(0.8) *{1} \ar@{-}[r] 
 & \cdots \ar@{-}[r]  
 & *++[o][F]{-2}  \ar@{-}[r] \ar@{}[d]|(0.8) *{1} 
 & *++[o][F]{-3}  \ar@{-}[r] \ar@{}[d]|(0.8) *{1} 
 & *++[o][F]{-2}  \ar@{-}[r] \ar@{}[d]|(0.8) *{1} 
 & \cdots \ar@{-}[r]  
 & *++[o][F]{-2}  \ar@{-}[r] \ar@{}[d]|(0.8) *{1} 
 & *++[o][F]{-3}  \ar@{-}[r] \ar@{}[d]|(0.8) *{1} 
 & *++[o][F]{-2}  \ar@{-}[r] \ar@{}[d]|(0.8) *{2}
 & \cdots \ar@{-}[r] 
 & *++[o][F]{-2}  \ar@{-}[r] \ar@{}[d]|(0.8) *{2} 
 & *++[o][F]{-2}  \ar@{-}[r] \ar@{}[d]|(0.8) *{2} 
 & *++[o][F]{-2}  \ar@{}[d]|(0.8) *{1} \\
&&&&&&&&&&&&&
}
}
\end{eqnarray*}
Therefore $e(R(\mathbb P_k^1,D))=4$.

\vskip.3truecm
\noindent
{\bf Case 4.}
We assume that $m=1$, $n=2$ and $1\le a\le 3$.
Note that $D\ge  2P_0 -  \frac{1}{2} P_1-\frac{2}{3} P_2 -\frac{4}{5}P_3$ by Lemma \ref{[[a]]<[[b]]}. 
Therefore $R(\mathbb P_k^1,D)$ has a rational singularity by Lemma \ref{bigger coefficient rationality} and  Lemma \ref{(2,2,n)type}.
The dual graph of the minimal good resolution of $\mathrm{Spec} (R(\mathbb P_k^1,D))$ is the following:
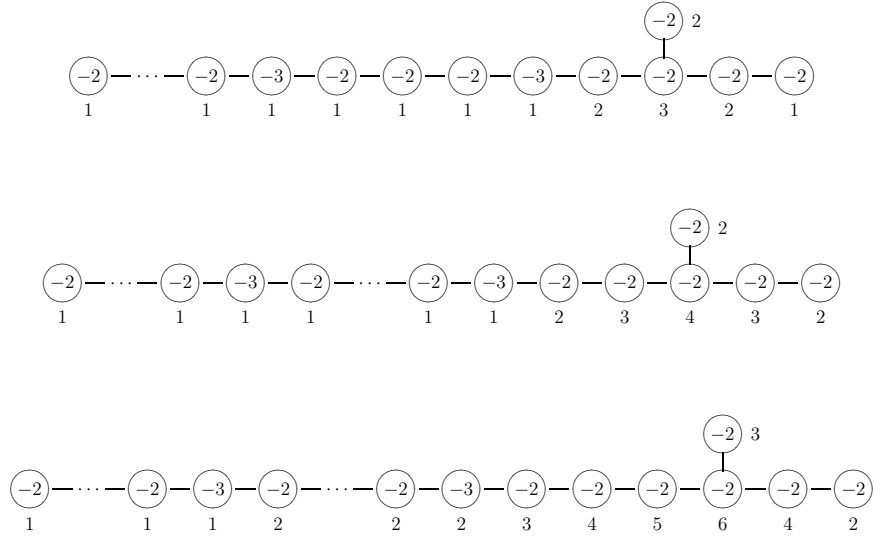
\begin{eqnarray*}
\scalebox{0.6}{
\xymatrix@C=12pt@R=12pt{
 &
 &
 &
 &
 &
 &
 &
 &
 & *++[o][F]{-2}  \ar@{-}[d] \ar@{}[r]|(0.5) *{2}
 &
\\
  *++[o][F]{-2}  \ar@{}[d]|(0.8) *{1} \ar@{-}[r] 
 & \cdots \ar@{-}[r]  
 & *++[o][F]{-2}  \ar@{-}[r] \ar@{}[d]|(0.8) *{1} 
 & *++[o][F]{-3}  \ar@{-}[r] \ar@{}[d]|(0.8) *{1} 
 & *++[o][F]{-2}  \ar@{-}[r] \ar@{}[d]|(0.8) *{1} 
 & *++[o][F]{-2}  \ar@{-}[r] \ar@{}[d]|(0.8) *{1} 
 & *++[o][F]{-2}  \ar@{-}[r] \ar@{}[d]|(0.8) *{1} 
 & *++[o][F]{-3}  \ar@{-}[r] \ar@{}[d]|(0.8) *{1} 
 & *++[o][F]{-2}  \ar@{-}[r] \ar@{}[d]|(0.8) *{2}
 & *++[o][F]{-2}  \ar@{-}[r] \ar@{}[d]|(0.8) *{3}
 & *++[o][F]{-2}  \ar@{-}[r] \ar@{}[d]|(0.8) *{2} 
 & *++[o][F]{-2}  \ar@{}[d]|(0.8) *{1} \\
&&&&&&&&&&&&
}
}
\end{eqnarray*}
\begin{eqnarray*}
\scalebox{0.6}{
\xymatrix@C=12pt@R=12pt{
 &
 &
 &
 &
 &
 &
 &
 & 
 &
 & *++[o][F]{-2}  \ar@{-}[d] \ar@{}[r]|(0.5) *{2}
 &
\\
  *++[o][F]{-2}  \ar@{}[d]|(0.8) *{1} \ar@{-}[r] 
 & \cdots \ar@{-}[r]  
 & *++[o][F]{-2}  \ar@{-}[r] \ar@{}[d]|(0.8) *{1} 
 & *++[o][F]{-3}  \ar@{-}[r] \ar@{}[d]|(0.8) *{1} 
 & *++[o][F]{-2}  \ar@{-}[r] \ar@{}[d]|(0.8) *{1} 
 & \cdots \ar@{-}[r]  
 & *++[o][F]{-2}  \ar@{-}[r] \ar@{}[d]|(0.8) *{1} 
 & *++[o][F]{-3}  \ar@{-}[r] \ar@{}[d]|(0.8) *{1} 
 & *++[o][F]{-2}  \ar@{-}[r] \ar@{}[d]|(0.8) *{2}
 & *++[o][F]{-2}  \ar@{-}[r] \ar@{}[d]|(0.8) *{3} 
 & *++[o][F]{-2}  \ar@{-}[r] \ar@{}[d]|(0.8) *{4} 
 & *++[o][F]{-2}  \ar@{-}[r] \ar@{}[d]|(0.8) *{3} 
 & *++[o][F]{-2}  \ar@{}[d]|(0.8) *{2} \\
&&&&&&&&&&&&&
}
}
\end{eqnarray*}
\begin{eqnarray*}
\scalebox{0.6}{
\xymatrix@C=12pt@R=12pt{
 &
 &
 &
 &
 &
 &
 &
 & 
 & 
 &
 & *++[o][F]{-2}  \ar@{-}[d] \ar@{}[r]|(0.5) *{3}
 &
\\
  *++[o][F]{-2}  \ar@{}[d]|(0.8) *{1} \ar@{-}[r] 
 & \cdots \ar@{-}[r]  
 & *++[o][F]{-2}  \ar@{-}[r] \ar@{}[d]|(0.8) *{1} 
 & *++[o][F]{-3}  \ar@{-}[r] \ar@{}[d]|(0.8) *{1} 
 & *++[o][F]{-2}  \ar@{-}[r] \ar@{}[d]|(0.8) *{2} 
 & \cdots \ar@{-}[r]  
 & *++[o][F]{-2}  \ar@{-}[r] \ar@{}[d]|(0.8) *{2} 
 & *++[o][F]{-3}  \ar@{-}[r] \ar@{}[d]|(0.8) *{2}
 & *++[o][F]{-2}  \ar@{-}[r] \ar@{}[d]|(0.8) *{3}
 & *++[o][F]{-2}  \ar@{-}[r] \ar@{}[d]|(0.8) *{4} 
 & *++[o][F]{-2}  \ar@{-}[r] \ar@{}[d]|(0.8) *{5} 
 & *++[o][F]{-2}  \ar@{-}[r] \ar@{}[d]|(0.8) *{6} 
 & *++[o][F]{-2}  \ar@{-}[r] \ar@{}[d]|(0.8) *{4} 
 & *++[o][F]{-2}  \ar@{}[d]|(0.8) *{2} \\
&&&&&&&&&&&&&&
}
}
\end{eqnarray*}
Therefore $e(R(\mathbb P_k^1,D))=4$ when $a=1,2$ and $e(R(\mathbb P_k^1,D))=5$ when $a=3$.

\vskip.3truecm
\noindent
{\bf Case 5.}
Assume one of the following holds:
\begin{enumerate}
\item[(i)]
$m=1$, $n=2$ and $a\ge 4$.
\item[(ii)]
$m=1$, $n\ge 3$ and  $a\ge 2$.
\item[(iii)]
$m\ge 2$ and $a\ge 1$.
\end{enumerate}
Then $R(\mathbb{P}^1_k, D)$ does not have a rational singularity because $\deg[5D] \le -2$ in case (i), $\deg[3D] \le -2$ in case (ii), and $\deg[2D] \le -2$ in case (iii).

\vskip.3truecm
\noindent
{\bf Case 6.}
Assume one of the following holds:
\begin{enumerate}
\item[(i)]
$m=1$,  $n\ge 3$ and $a=1$.
\item[(ii)]
$m\ge 2$ and $a=0$.
\end{enumerate}
In this case, we have  $e(R(\mathbb P_k^1,D))\ge 5$.
We consider only case (i)  since we can apply the same argument to case (ii).
Let $E_0$ be the central curve of the minimal good resolution of $\mathrm{Spec} (R(\mathbb P_k^1,D))$ and $E_1$,  $E_2$ be the $(-3)$-curves in its dual graph.
Let $Z$ be the fundamental cycle of the minimal good resolution of $\mathrm{Spec} (R(\mathbb P_k^1,D))$.
We have $\mathrm{Coeff}_{E_0}(Z)\ge 4$ by Corollary \ref{n_0=min}.
Therefore  $\mathrm{Coeff}_{E_1}(Z)+\mathrm{Coeff}_{E_2}(Z)\ge 3$ by Lemma \ref{formula}.
Hence $e(R(\mathbb P_k^1,D))\ge 5$.

\end{proof}

Finally, we consider the case the central curve is a $(-2)$-curve and there is  one $(-3)$-curve in one branch.
\begin{prop}\label{central 2-3-3 2}
Let $D=2P_0-\sum_{i=1}^3a_iP_i$  be an ample $\mathbb Q$-divisor on $\mathbb P_k^1$, where $a_i\in \mathbb Q_{\ge 0}$ and $P_i$ are distinct points of $\mathbb P_k^1$.
Assume that $a_2\le a_3$, $a_1=\frac{m}{m+1}$,  $\frac{1}{a_2}=[[(2)^a,3, (2)^b]]$, $\frac{1}{a_3}=[[(2)^c,3, (2)^d]]$ for $m,a,b,c,d\in \mathbb Z_{\ge 0}$ and $R(\mathbb P_k^1,D)$ has a rational singularity with
 $e(R(\mathbb P_k^1,D))=4$. 
Then $(a_1,a_2,a_3)$ is  one of the following:
\begin{align*}
\mathrm{(1)}&\ (0,\frac{(a+1)b+2a+1}{(a+2)b+2a+3},\frac{(c+1)d+2c+1}{(c+2)d+2c+3}) & &\mbox{for}\ a\ge 0, b\ge 0,c\ge 0,d\ge 0 \\
\mathrm{(2)}&\ (\frac{m}{m+1},\frac{b+1}{2b+3},\frac{d+1}{2d+3}) & &\mbox{for}\ m\ge 1,b\ge 0,d\ge 0 \\
\mathrm{(3)}&\ (\frac{1}{2},\frac{b+1}{2b+3},\frac{(c+1)d+2c+1}{(c+2)d+2c+3}) & &\mbox{for}\  b\ge 0,c\ge 1,d\ge 0 \\
\mathrm{(4)}&\ (\frac{1}{2},\frac{2b+3}{3b+5},\frac{2d+3}{3d+5})& &\mbox{for}\ b\ge 0,d\ge 0,\\
\mathrm{(5)}&\ (\frac{1}{2},\frac{3}{5},\frac{3d+5}{4d+7})& & \mbox{for}\ d\ge 0,\\
\mathrm{(6)}&\ (\frac{1}{2},\frac{3}{5},\frac{4d+7}{5d+9})& &  \mbox{for}\ d\ge 0,\\
\mathrm{(7)}&\ (\frac{2}{3},\frac{1}{3},\frac{(c+1)d+2c+1}{(c+2)d+2c+3})& & \mbox{for}\ c\ge1 , d\ge 0,\\
\mathrm{(8)}&\ (\frac{m}{m+1},\frac{1}{3},\frac{2d+3}{3d+5})& & \mbox{for}\ m\ge 3,d\ge 0.
\end{align*}
\end{prop}

\begin{proof}
Note that $a\le c$ by Lemma \ref{[[a]]<[[b]]}. \\
\noindent
{\bf Case 1.}
We assume that $m=0$.
Then $R(\mathbb P_k^1,D)$ has a rational singularity since 
 $\mathrm{deg} [lD]\ge 0$ for any $l\in\mathbb N$.  
The dual graph of the minimal good resolution of $\mathrm{Spec} (R(\mathbb P_k^1,D))$ is the following:
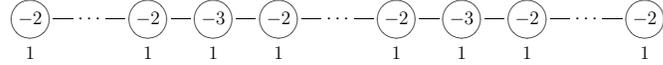
\begin{eqnarray*}
\scalebox{0.6}{
\xymatrix@C=12pt@R=12pt{
   *++[o][F]{-2}  \ar@{}[d]|(0.8) *{1} \ar@{-}[r] 
 & \cdots \ar@{-}[r]  
 & *++[o][F]{-2}  \ar@{-}[r] \ar@{}[d]|(0.8) *{1} 
 & *++[o][F]{-3}  \ar@{-}[r] \ar@{}[d]|(0.8) *{1} 
 & *++[o][F]{-2}  \ar@{-}[r] \ar@{}[d]|(0.8) *{1} 
 &  \cdots \ar@{-}[r]
 & *++[o][F]{-2}  \ar@{-}[r] \ar@{}[d]|(0.8) *{1}
 & *++[o][F]{-3}  \ar@{-}[r] \ar@{}[d]|(0.8) *{1}
 & *++[o][F]{-2}  \ar@{-}[r] \ar@{}[d]|(0.8) *{1}
 & \cdots \ar@{-}[r] 
 &  *++[o][F]{-2}  \ar@{}[d]|(0.8) *{1} \\
&&&&&&&&&&&
}
}
\end{eqnarray*}
Therefore $e(R(\mathbb P_k^1,D))=4$.

\vskip.3truecm
\noindent
{\bf Case 2.}
We assume that  $m\ge 1$ and $a=c=0$. 
Note that $D\ge  2P_0 -  \frac{m}{m+1} P_1-\frac{1}{2} P_2 -\frac{1}{2}P_3$ by Lemma \ref{[[a]]<[[b]]}. 
Therefore $R(\mathbb P_k^1,D)$ has a rational singularity by Lemma \ref{bigger coefficient rationality} and  Lemma \ref{(2,2,n)type}.
The dual graph of the minimal good resolution of $\mathrm{Spec} (R(\mathbb P_k^1,D))$ is the following:
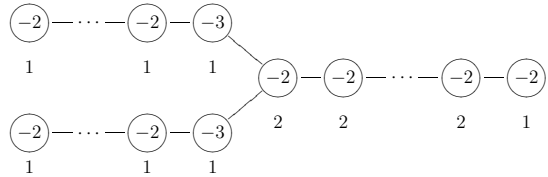
\begin{equation*}
\scalebox{0.6}{
\xymatrix@C=12pt@R=12pt{
  *++[o][F]{-2}  \ar@{}[d]|(0.8) *{1} \ar@{-}[r] 
 &  \cdots \ar@{-}[r]
 & *++[o][F]{-2}  \ar@{-}[r] \ar@{}[d]|(0.8) *{1}
 &  *++[o][F]{-3}  \ar@{}[d]|(0.8) *{1} \ar@{-}[dr]\\
&&&& *++[o][F]{-2}  \ar@{-}[r] \ar@{}[d]|(0.8) *{2} & *++[o][F]{-2}  \ar@{-}[r] \ar@{}[d]|(0.8) *{2} & \cdots\ar@{-}[r] & *++[o][F]{-2}  \ar@{-}[r] \ar@{}[d]|(0.8) *{2}& *++[o][F]{-2}  \ar@{}[d]|(0.8) *{1}\\
 *++[o][F]{-2}  \ar@{}[d]|(0.8) *{1} \ar@{-}[r] 
 &  \cdots \ar@{-}[r]
& *++[o][F]{-2}  \ar@{-}[r] \ar@{}[d]|(0.8) *{1}
 &  *++[o][F]{-3}  \ar@{}[d]|(0.8) *{1} \ar@{-}[ur] &&&&&& \\
&&&&&&&&&
}
}
\end{equation*}
Therefore $e(R(\mathbb P_k^1,D))=4$.

\vskip.3truecm
\noindent
{\bf Case 3.}
We assume that $m=1$, $a=0$ and $c\ge 1$. 
Note that $D\ge  2P_0 -  \frac{1}{2} P_1-\frac{1}{2} P_2 -\frac{c+d+1}{c+d+2}P_3$ by Lemma \ref{[[a]]<[[b]]}. 
Therefore $R(\mathbb P_k^1,D)$ has a rational singularity by Lemma \ref{bigger coefficient rationality} and  Lemma \ref{(2,2,n)type}.
The dual graph of the minimal good resolution of $\mathrm{Spec} (R(\mathbb P_k^1,D))$ is the following:
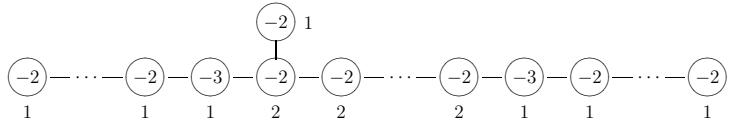
\begin{eqnarray*}
\scalebox{0.6}{
\xymatrix@C=12pt@R=12pt{
 &
 &
 &
 & *++[o][F]{-2}  \ar@{-}[d] \ar@{}[r]|(0.5) *{1}
 &
 &
 & 
\\
   *++[o][F]{-2}  \ar@{}[d]|(0.8) *{1} \ar@{-}[r] 
 &  \cdots \ar@{-}[r]
 & *++[o][F]{-2}  \ar@{-}[r] \ar@{}[d]|(0.8) *{1}
 & *++[o][F]{-3}  \ar@{-}[r] \ar@{}[d]|(0.8) *{1}
 & *++[o][F]{-2}  \ar@{-}[r] \ar@{}[d]|(0.8) *{2}
 & *++[o][F]{-2}  \ar@{-}[r] \ar@{}[d]|(0.8) *{2}
 & \cdots \ar@{-}[r] 
 & *++[o][F]{-2}  \ar@{-}[r] \ar@{}[d]|(0.8) *{2}
 & *++[o][F]{-3}  \ar@{-}[r] \ar@{}[d]|(0.8) *{1}
 & *++[o][F]{-2}  \ar@{-}[r] \ar@{}[d]|(0.8) *{1}
 & \cdots \ar@{-}[r] 
 &  *++[o][F]{-2}  \ar@{}[d]|(0.8) *{1} \\
&&&&&&&&&&&&
}
}
\end{eqnarray*}
Therefore $e(R(\mathbb P_k^1,D))=4$.

\vskip.3truecm
\noindent
{\bf Case 4.}
We assume that  $m=1$, $a=1$ and $1\le c\le 3$. 
Note that $D\ge  2P_0 -  \frac{1}{2} P_1-\frac{2}{3} P_2 -\frac{4}{5}P_3$ by Lemma \ref{[[a]]<[[b]]}. 
Therefore $R(\mathbb P_k^1,D)$ has a rational singularity by Lemma \ref{bigger coefficient rationality} and  Lemma \ref{(2,2,n)type}.
The dual graphs of the minimal good resolution of $\mathrm{Spec} (R(\mathbb P_k^1,D))$ are the following:
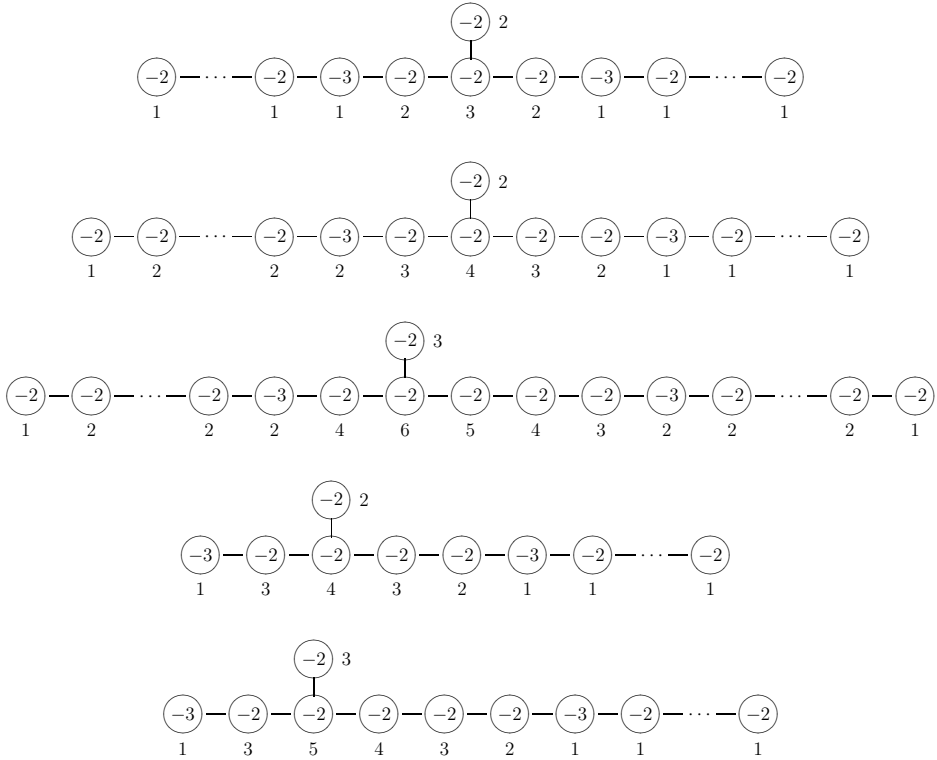
\begin{eqnarray*}
\scalebox{0.6}{
\xymatrix@C=12pt@R=12pt{
 &
 &
 &
 &
 & *++[o][F]{-2}  \ar@{-}[d] \ar@{}[r]|(0.5) *{2}
 &
 &
 & 
\\
   *++[o][F]{-2}  \ar@{}[d]|(0.8) *{1} \ar@{-}[r] 
 &  \cdots \ar@{-}[r]
 & *++[o][F]{-2}  \ar@{-}[r] \ar@{}[d]|(0.8) *{1}
 & *++[o][F]{-3}  \ar@{-}[r] \ar@{}[d]|(0.8) *{1}
 & *++[o][F]{-2}  \ar@{-}[r] \ar@{}[d]|(0.8) *{2}
 & *++[o][F]{-2}  \ar@{-}[r] \ar@{}[d]|(0.8) *{3}
 & *++[o][F]{-2}  \ar@{-}[r] \ar@{}[d]|(0.8) *{2}
 & *++[o][F]{-3}  \ar@{-}[r] \ar@{}[d]|(0.8) *{1}
 & *++[o][F]{-2}  \ar@{-}[r] \ar@{}[d]|(0.8) *{1}
 & \cdots \ar@{-}[r] 
 &  *++[o][F]{-2}  \ar@{}[d]|(0.8) *{1} \\
&&&&&&&&&&&
}
}
\end{eqnarray*}
\begin{eqnarray*}
\scalebox{0.6}{
\xymatrix@C=12pt@R=12pt{
 &
 &
 &
 &
 &
 & *++[o][F]{-2}  \ar@{-}[d] \ar@{}[r]|(0.5) *{2}
 &
 &
 & 
\\
   *++[o][F]{-2}  \ar@{}[d]|(0.8) *{1} \ar@{-}[r] 
 & *++[o][F]{-2}  \ar@{-}[r] \ar@{}[d]|(0.8) *{2}
 &  \cdots \ar@{-}[r]
 & *++[o][F]{-2}  \ar@{-}[r] \ar@{}[d]|(0.8) *{2}
 & *++[o][F]{-3}  \ar@{-}[r] \ar@{}[d]|(0.8) *{2}
 & *++[o][F]{-2}  \ar@{-}[r] \ar@{}[d]|(0.8) *{3}
 & *++[o][F]{-2}  \ar@{-}[r] \ar@{}[d]|(0.8) *{4}
 & *++[o][F]{-2}  \ar@{-}[r] \ar@{}[d]|(0.8) *{3}
 & *++[o][F]{-2}  \ar@{-}[r] \ar@{}[d]|(0.8) *{2}
 & *++[o][F]{-3}  \ar@{-}[r] \ar@{}[d]|(0.8) *{1}
 & *++[o][F]{-2}  \ar@{-}[r] \ar@{}[d]|(0.8) *{1}
 & \cdots \ar@{-}[r] 
 &  *++[o][F]{-2}  \ar@{}[d]|(0.8) *{1} \\
&&&&&&&&&&&&&
}
}
\end{eqnarray*}
\begin{eqnarray*}
\scalebox{0.6}{
\xymatrix@C=12pt@R=12pt{
 &
 &
 &
 &
 &
 & *++[o][F]{-2}  \ar@{-}[d] \ar@{}[r]|(0.5) *{3}
 &
 &
 & 
\\
   *++[o][F]{-2}  \ar@{}[d]|(0.8) *{1} \ar@{-}[r] 
 & *++[o][F]{-2}  \ar@{-}[r] \ar@{}[d]|(0.8) *{2}
 &  \cdots \ar@{-}[r]
 & *++[o][F]{-2}  \ar@{-}[r] \ar@{}[d]|(0.8) *{2}
 & *++[o][F]{-3}  \ar@{-}[r] \ar@{}[d]|(0.8) *{2}
 & *++[o][F]{-2}  \ar@{-}[r] \ar@{}[d]|(0.8) *{4}
 & *++[o][F]{-2}  \ar@{-}[r] \ar@{}[d]|(0.8) *{6}
 & *++[o][F]{-2}  \ar@{-}[r] \ar@{}[d]|(0.8) *{5}
 & *++[o][F]{-2}  \ar@{-}[r] \ar@{}[d]|(0.8) *{4}
 & *++[o][F]{-2}  \ar@{-}[r] \ar@{}[d]|(0.8) *{3}
 & *++[o][F]{-3}  \ar@{-}[r] \ar@{}[d]|(0.8) *{2}
 & *++[o][F]{-2}  \ar@{-}[r] \ar@{}[d]|(0.8) *{2}
 & \cdots \ar@{-}[r] 
 & *++[o][F]{-2}  \ar@{-}[r] \ar@{}[d]|(0.8) *{2}
 &  *++[o][F]{-2}  \ar@{}[d]|(0.8) *{1} \\
&&&&&&&&&&&&&&&
}
}
\end{eqnarray*}
\begin{eqnarray*}
\scalebox{0.6}{
\xymatrix@C=12pt@R=12pt{
 &
 & *++[o][F]{-2}  \ar@{-}[d] \ar@{}[r]|(0.5) *{2}
 &
 &
 & 
\\
     *++[o][F]{-3}  \ar@{}[d]|(0.8) *{1} \ar@{-}[r] 
 & *++[o][F]{-2}  \ar@{-}[r] \ar@{}[d]|(0.8) *{3}
 & *++[o][F]{-2}  \ar@{-}[r] \ar@{}[d]|(0.8) *{4}
 & *++[o][F]{-2}  \ar@{-}[r] \ar@{}[d]|(0.8) *{3}
 & *++[o][F]{-2}  \ar@{-}[r] \ar@{}[d]|(0.8) *{2}
 & *++[o][F]{-3}  \ar@{-}[r] \ar@{}[d]|(0.8) *{1}
 & *++[o][F]{-2}  \ar@{-}[r] \ar@{}[d]|(0.8) *{1}
 & \cdots \ar@{-}[r] 
 &  *++[o][F]{-2}  \ar@{}[d]|(0.8) *{1} \\
&&&&&&&&&&
}
}
\end{eqnarray*}
\begin{eqnarray*}
\scalebox{0.6}{
\xymatrix@C=12pt@R=12pt{
 &
 & *++[o][F]{-2}  \ar@{-}[d] \ar@{}[r]|(0.5) *{3}
 &
 &
 & 
\\
     *++[o][F]{-3}  \ar@{}[d]|(0.8) *{1} \ar@{-}[r] 
 & *++[o][F]{-2}  \ar@{-}[r] \ar@{}[d]|(0.8) *{3}
 & *++[o][F]{-2}  \ar@{-}[r] \ar@{}[d]|(0.8) *{5}
 & *++[o][F]{-2}  \ar@{-}[r] \ar@{}[d]|(0.8) *{4}
 & *++[o][F]{-2}  \ar@{-}[r] \ar@{}[d]|(0.8) *{3}
 & *++[o][F]{-2}  \ar@{-}[r] \ar@{}[d]|(0.8) *{2}
 & *++[o][F]{-3}  \ar@{-}[r] \ar@{}[d]|(0.8) *{1}
 & *++[o][F]{-2}  \ar@{-}[r] \ar@{}[d]|(0.8) *{1}
 & \cdots \ar@{-}[r] 
 &  *++[o][F]{-2}  \ar@{}[d]|(0.8) *{1} \\
&&&&&&&&&&
}
}
\end{eqnarray*}
Therefore $e(R(\mathbb P_k^1,D))=4$ when $c=1$ or $c=2,3$ and $b=0$,
$e(R(\mathbb P_k^1,D))=5$ when $c=2$ and $b\ge 1$ and
$e(R(\mathbb P_k^1,D))=6$ when $c=3$ and $b\ge 1$.

\vskip.3truecm
\noindent
{\bf Case 5.}
Assume one of the following holds:
\begin{enumerate}
\item[(i)]
$m=1$, $a=1$ and $c\ge 4$. 
\item[(ii)]
$m\ge 2$, $a=0$,  $b\ge 1$ and $c\ge 1$. 
\item[(iii)]
$m\ge 3$, $a=b=0$ and $c\ge 2$. 
\end{enumerate}
In this case, we have  $e(R(\mathbb P_k^1,D))\ge 5$.
We consider only case (i)  since we can apply the same argument to cases (ii) and (iii).
Let $E_0$ be the central curve of the minimal good resolution of $\mathrm{Spec} (R(\mathbb P_k^1,D))$ and $E_1$  be the $(-3)$-curve in the branch corresponding to $P_2$ in its dual graph.
Let $Z$ be the fundamental cycle of the minimal good resolution of $\mathrm{Spec} (R(\mathbb P_k^1,D))$.
We have $\mathrm{Coeff}_{E_0}(Z)\ge 6$ by Corollary \ref{n_0=min}.
By Lemma \ref{formula}, we have $\mathrm{Coeff}_{E_1}(Z)\ge 2$.
Hence $e(R(\mathbb P_k^1,D))\ge 5$.

\vskip.3truecm
\noindent
{\bf Case 6.}
Assume one of the following holds:
\begin{enumerate}
\item[(i)]
$m\ge 1$, $a\ge 2$ and $c\ge 2$.
\item[(ii)]
$m\ge 2$, $a\ge 1$ and $c\ge 1$.
\end{enumerate}
Then $R(\mathbb{P}^1_k, D)$ does not have a rational singularity because $\deg[3D] \le -2$ in case (i), and $\deg[2D] \le -2$ in case (ii).

\vskip.3truecm
\noindent
{\bf Case 7.}
We assume that  $m=2$,  $a=b=0$  and $c\ge 1$.
Note that $D\ge  2P_0 -  \frac{2}{3} P_1-\frac{1}{3} P_2 -\frac{c+d+1}{c+d+2}P_3$ by Lemma \ref{[[a]]<[[b]]}. 
Therefore $R(\mathbb P_k^1,D)$ has a rational singularity by Lemma \ref{bigger coefficient rationality} and  Proposition \ref{central 2-3}.(i).(9).
The dual graph of the minimal good resolution of $\mathrm{Spec} (R(\mathbb P_k^1,D))$ is the following:
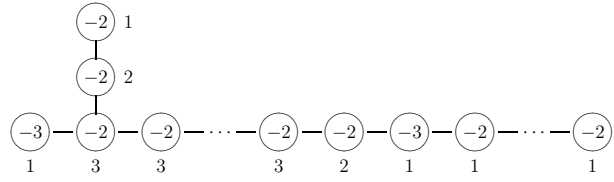
\begin{eqnarray*}
\scalebox{0.6}{
\xymatrix@C=12pt@R=12pt{
& *++[o][F]{-2}  \ar@{-}[d] \ar@{}[r]|(0.5) *{1}
 &
 &
 & 
\\
 & *++[o][F]{-2}  \ar@{-}[d] \ar@{}[r]|(0.5) *{2}
 &
 &
 & 
\\
   *++[o][F]{-3}  \ar@{}[d]|(0.8) *{1} \ar@{-}[r] 
 & *++[o][F]{-2}  \ar@{-}[r] \ar@{}[d]|(0.8) *{3}
 & *++[o][F]{-2}  \ar@{-}[r] \ar@{}[d]|(0.8) *{3}
 & \cdots \ar@{-}[r] 
 & *++[o][F]{-2}  \ar@{-}[r] \ar@{}[d]|(0.8) *{3}
 & *++[o][F]{-2}  \ar@{-}[r] \ar@{}[d]|(0.8) *{2}
 & *++[o][F]{-3}  \ar@{-}[r] \ar@{}[d]|(0.8) *{1}
 & *++[o][F]{-2}  \ar@{-}[r] \ar@{}[d]|(0.8) *{1}
 & \cdots \ar@{-}[r] 
 &  *++[o][F]{-2}  \ar@{}[d]|(0.8) *{1} \\
&&&&&&&&&&
}
}
\end{eqnarray*}
Therefore $e(R(\mathbb P_k^1,D))=4$.

\vskip.3truecm
\noindent
{\bf Case 8.}
We assume that $m\ge 3$, $a=b=0$ and $c=1$. 
Note that $D\ge  2P_0 -  \frac{m}{m+1} P_1-\frac{1}{3} P_2 -\frac{2}{3}P_3$ by Lemma \ref{[[a]]<[[b]]}. 
Therefore $R(\mathbb P_k^1,D)$ has a rational singularity by Lemma \ref{bigger coefficient rationality} and  Proposition \ref{central 2-3}.(i).(9).
The dual graph of the minimal good resolution of $\mathrm{Spec} (R(\mathbb P_k^1,D))$ is the following:
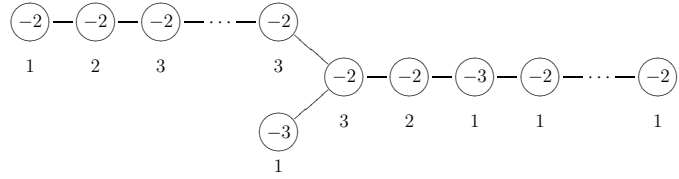
\begin{equation*}
\scalebox{0.6}{
\xymatrix@C=12pt@R=12pt{
  *++[o][F]{-2}  \ar@{}[d]|(0.8) *{1} \ar@{-}[r] 
 & *++[o][F]{-2}  \ar@{-}[r] \ar@{}[d]|(0.8) *{2}
 & *++[o][F]{-2}  \ar@{-}[r] \ar@{}[d]|(0.8) *{3}
 &  \cdots \ar@{-}[r]
 &  *++[o][F]{-2}  \ar@{}[d]|(0.8) *{3} \ar@{-}[dr]\\
&&&&& *++[o][F]{-2}  \ar@{-}[r] \ar@{}[d]|(0.8) *{3} & *++[o][F]{-2}  \ar@{-}[r] \ar@{}[d]|(0.8) *{2} & *++[o][F]{-3}  \ar@{-}[r] \ar@{}[d]|(0.8) *{1}  & *++[o][F]{-2}  \ar@{-}[r] \ar@{}[d]|(0.8) *{1} & \cdots\ar@{-}[r] & *++[o][F]{-2}  \ar@{}[d]|(0.8) *{1}\\
 &
 & 
 & 
 &  *++[o][F]{-3}  \ar@{}[d]|(0.8) *{1} \ar@{-}[ur] &&&&&&\\
&&&&&&&&&&&
}
}
\end{equation*}
Therefore $e(R(\mathbb P_k^1,D))=4$.

\end{proof}

\subsection{Classification of $R(\mathbb P_k^1,D)$ which is a rational triple point and rational  fourth point}
In this subsection, we summarize our results of this section in the following theorem.
\begin{thm}\label{classification summarize}
Let $D=sP_0-\sum_{i=1}^ra_iP_i$  be an ample $\mathbb Q$-divisor on $\mathbb P_k^1$, where $s\in \mathbb N$ and $a_i\in \mathbb Q$ with $0\le a_i<1$ and $P_i$ are distinct points of $\mathbb P_k^1$.
Assume that $R(\mathbb P_k^1,D)$ has a rational singularity.
Suppose if $T(\frac{1}{a_i})=T(\frac{1}{a_j})$ for $i<j$, then $a_i\le a_j$, and if $T(\frac{1}{a_i})=\emptyset$ and $T(\frac{1}{a_j})\not=\emptyset$, then $i<j$, where $T(*)$ is defined in Definition \ref{T frac}.
\vskip.2truecm
\noindent
$(1)$ If  $e(R(\mathbb P_k^1,D))=3$, then $(s,a_1,\dots,a_r)$ is  one of the following:
Here, $n$, $a$, $b$, $c$ are any non-negative integers.
{\rm
\begin{multienumerate}
\mitemxx{$(3,\frac{a}{a+1},\frac{b}{b+1},\frac{c}{c+1})$,}{$(2,0,\frac{n}{n+1},\frac{(a+1)b+2a+1}{(a+2)b+2a+3})$,}
\mitemxx{$(2,\frac{1}{2},\frac{n+1}{n+2},\frac{b+1}{2b+3})$,}{$(2,\frac{1}{2},\frac{1}{2},\frac{(a+2)b+2a+3}{(a+3)b+2a+5})$,}
\mitemxxx{$(2,\frac{1}{2},\frac{2}{3},\frac{2b+3}{3b+5})$,}{$(2,\frac{1}{2},\frac{2}{3},\frac{3b+5}{4b+7})$,}{$(2,\frac{1}{2},\frac{2}{3},\frac{7}{9})$,}
\mitemxxx{$(2,\frac{1}{2},\frac{3}{4},\frac{3}{5})$,}{$(2,\frac{1}{2},\frac{4}{5},\frac{3}{5})$,}{$(2,\frac{2}{3},\frac{n+2}{n+3},\frac{1}{3})$,}
\end{multienumerate}
}
\vskip.2truecm
\noindent
$(2)$ If  $e(R(\mathbb P_k^1,D))=4$, then $(s,a_1,\dots,a_r)$ is  one of the following:
Here, $m$, $n$, $a$, $b$, $c$, $d$ are any non-negative integers.
{\rm
\begin{multienumerate}
 \mitemxxx{$(3,\frac{1}{2},\frac{1}{2},\frac{c}{c+1},\frac{d}{d+1})$,}{$(2,\frac{1}{2},\frac{2}{3},\frac{4b+11}{5b+14})$,}{$(2,\frac{1}{2},\frac{3}{4},\frac{2b+5}{3b+8})$,}
 \mitemxxx{$(2,\frac{1}{2},\frac{4}{5},\frac{2b+5}{3b+8})$,}{$(2,\frac{1}{2},\frac{5}{6},\frac{3}{5})$,}{$(2,\frac{1}{2},\frac{6}{7},\frac{3}{5})$,}
 \mitemxxx{$(2,\frac{2}{3},\frac{2}{3},\frac{b+2}{2b+5})$,}{$(2,\frac{2}{3},\frac{3}{4},\frac{b+2}{2b+5})$,}{$(2,\frac{2}{3},\frac{4}{5},\frac{2}{5})$,}
 \mitemxxx{$(2,\frac{3}{4},\frac{3}{4},\frac{1}{3})$,}{$(2,\frac{3}{4},\frac{4}{5},\frac{1}{3})$,}{$(2,\frac{3}{4},\frac{5}{6},\frac{1}{3})$,}
 \mitemxx{$(4,\frac{a}{a+1},\frac{b}{b+1},\frac{c}{c+1},\frac{d}{d+1})$,}{$(2,0,\frac{n}{n+1},\frac{(2a+1)b+3a+1}{(2a+3)b+3a+4})$,}
 \mitemxx{$(2,\frac{1}{2},\frac{n+1}{n+2},\frac{b+1}{3b+4})$,}{$(2,\frac{1}{2},\frac{1}{2},\frac{(2a+3)b+3a+4}{(2a+5)b+3a+7})$,}
 \mitemxxx{$(2,\frac{1}{2},\frac{2}{3},\frac{3b+4}{5b+7})$,}{$(2,\frac{1}{2},\frac{2}{3},\frac{5b+7}{7b+10})$,}{$(2,\frac{1}{2},\frac{2}{3},\frac{7b+10}{9b+13})$,}
 \mitemxxx{$(2,\frac{1}{2},\frac{3}{4},\frac{3b+4}{5b+7})$,}{$(2,\frac{1}{2},\frac{4}{5},\frac{3b+4}{5b+7})$,}{$(2,\frac{1}{2},\frac{5}{6},\frac{4}{7})$,}
 \mitemxxx{$(2,\frac{1}{2},\frac{6}{7},\frac{4}{7})$,}{$(2,\frac{2}{3},\frac{n+2}{n+3},\frac{b+1}{3b+4})$,}{$(2,\frac{3}{4},\frac{n+3}{n+4},\frac{1}{4})$,}
 \mitemx{$(3,\frac{m}{m+1},\frac{n}{n+1},\frac{(a+1)b+2a+1}{(a+2)b+2a+3})$,}
\mitemx{$(2,0,\frac{n}{n+1},\frac{\big((a+1)b + 3a+2\big)c+(2a+2)b+5a+3}{\big((a+2)b+ 3a+5\big)c+(2a+4)b+5a+8})$,}
 \mitemx{$(2,\frac{1}{2},\frac{n+1}{n+2},\frac{(b +2)c+2b+3}{(2b+ 5)c+4b+8})$,}
\mitemx{$(2,\frac{1}{2},\frac{1}{2},\frac{\big((a+2)b + 3a+5\big)c+(2a+4)b+5a+8}{\big((a+3)b+ 3a+8\big)c+(2a+6)b+5a+13})$,}
 \mitemxx{$(2,\frac{1}{2},\frac{2}{3},\frac{(2b + 5)c+4b+8}{(3b+ 8)c+6b+13})$,}{$(2,\frac{1}{2},\frac{2}{3},\frac{(3b + 8)c+6b+13}{(4b+ 11)c+8b+18})$,}
 \mitemxx{$(2,0,\frac{(a+1)b+2a+1}{(a+2)b+2a+3},\frac{(c+1)d+2c+1}{(c+2)d+2c+3})$,}{$(2,\frac{m+1}{m+2},\frac{b+1}{2b+3},\frac{d+1}{2d+3})$,}
 \mitemxx{$(2,\frac{1}{2},\frac{b+1}{2b+3},\frac{(c+2)d+2c+3}{(c+3)d+2c+5})$,}{$(2,\frac{1}{2},\frac{2b+3}{3b+5},\frac{2d+3}{3d+5})$,}
\mitemxx{$(2,\frac{1}{2},\frac{3}{5},\frac{3d+5}{4d+7})$,}{$(2,\frac{1}{2},\frac{3}{5},\frac{4d+7}{5d+9})$,}
\mitemxx{$(2,\frac{2}{3},\frac{1}{3},\frac{(c+2)d+2c+3}{(c+3)d+2c+5})$,}{$(2,\frac{m+3}{m+4},\frac{1}{3},\frac{2d+3}{3d+5})$.}
\end{multienumerate}
}

\end{thm}

\section{$F$-rationality of two-dimensional  graded rings with rational triple point and rational fourth point}
In this section, we determine $p(3)$ and $p(4)$ in Theorem \ref{Main theorem} using the classification in Section 5.

We can reduce the calculation to check the $F$-rationality of $R(\mathbb P_k^1,D)$ using the following lemma when we prove the theorems in this section.
\begin{lem}\label{deg le -2}
Let $D=2P_0-\sum_{i=1}^3b_iP_i$  be an ample $\mathbb Q$-divisor on $\mathbb P_k^1$, where $b_i\in \mathbb Q_{>0}$ and $P_i$ are distinct points of $\mathbb P_k^1$.
\begin{enumerate}
\item If $(b_1,b_2,b_3)=(\frac{1}{2},\frac{1}{2},\frac{n}{n+1})$ for $n\in\mathbb N$, then $\mathrm{deg} [-lD]\le -2$ 
for $l\in \mathbb N\setminus  2\mathbb N$ and $\mathrm{deg} [-lD]\le -1$ 
for $l\in   2\mathbb N$.


\item If $(b_1,b_2,b_3)=(\frac{1}{2},\frac{2}{3},\frac{3}{4})$, then $\mathrm{deg} [-lD]\le -2$ \\
for $l\in \mathbb N$ with $l\neq 2,3,4,6,8,12$.

\item If $(b_1,b_2,b_3)=(\frac{1}{2},\frac{2}{3},\frac{4}{5})$, then $\mathrm{deg} [-lD]\le -2$ \\
for $l\in \mathbb N$ with $l\neq 2,3,4,5,6,8,9,10,12,14,15,18,20,24,30$.
 
\item If $(b_1,b_2,b_3)=(\frac{1}{3},\frac{2}{3},\frac{n}{n+1})$ for $n\in\mathbb N$, then $\mathrm{deg} [-lD]\le -2$ 
for $l\in \mathbb N\setminus  3\mathbb N$ and $\mathrm{deg} [-lD]\le -1$ 
for $l\in  3\mathbb N$.

\item If $(b_1,b_2,b_3)=(\frac{1}{4},\frac{3}{4},\frac{n}{n+1})$ for $n\in\mathbb N$, then $\mathrm{deg} [-lD]\le -2$ 
for $l\in \mathbb N\setminus  4\mathbb N$ and $\mathrm{deg} [-lD]\le -1$ 
for $l\in 4\mathbb N$.
\end{enumerate}
\end{lem}

\begin{proof}
This lemma follows immediately by direct computation.
\end{proof}

\begin{thm}\label{F-rationality of RTP}
Let $R$ be a two-dimensional graded ring with $e(R)=3$ and a rational singularity.
If $p\ge 7$,
then $R$ is $F$-rational.
Furthermore, 
this inequality is best possible.
\end{thm}

\begin{proof}
Example \ref{ex1} shows that there exists  a two-dimensional non-$F$-rational graded ring $R$ with a rational singularity,  $e(R)=3$ and $p=5$.

From now on, we assume that $p\ge 7$.
By Theorem \ref{DP const},  Theorem \ref{graded rational singularity} and Theorem \ref{classification summarize},
there exists an ample $\mathbb Q$-divisor $D$ on $\mathbb P_k^1$   in the list of Theorem \ref{classification summarize}.(1) with 
$R\cong R(\mathbb P_k^1,D)$.
Let $D=sP_0-\sum_{i=1}^3a_iP_i$, where $s\in\mathbb N$, $0\le a_i<1$ and $P_i$ are distinct points of $\mathbb P_k^1$.
Let  $n,a,b,c$ be non-negative integers.
If necessary, we may reorder $(a_1,a_2,a_3)$.

\vskip.3truecm
\noindent
{\bf Case 1.}
We assume that  $(s,a_1,a_2,a_3)$ is one of the followings:
\[(3,\frac{a}{a+1},\frac{b}{b+1},\frac{c}{c+1}),\ \ \  (2,0,\frac{n}{n+1},\frac{(a+1)b+2a+1}{(a+2)b+2a+3}).\] 
Then $R(\mathbb P_k^1,D)$ is $F$-rational by Proposition \ref{f-rational deg 1}.(1).

\vskip.3truecm
\noindent
{\bf Case 2.} 
We assume that  $s=2$ and $(a_1,a_2,a_3)$ is one of the followings:
\[(\frac{1}{2},\frac{2}{3},\frac{7}{9}),\ (\frac{1}{2},\frac{3}{4},\frac{3}{5}),\ (\frac{1}{2},\frac{4}{5},\frac{3}{5}).\]
Then $R(\mathbb P_k^1,D)$ is $F$-rational by Theorem \ref{$F$-rationality when p not | d_i}.

\vskip.3truecm
\noindent
{\bf Case 3.}
We assume that  $s=2$ and $(a_1,a_2,a_3)$ is one of the followings:
\[
(\frac{1}{2},\frac{b+1}{2b+3},\frac{n+1}{n+2}),\ \ \  (\frac{1}{2},\frac{1}{2},\frac{(a+2)b+2a+3}{(a+3)b+2a+5}).
\]
Then $D\ge 2P_0-\frac{1}{2}P_1-\frac{1}{2}P_2-\frac{l}{l+1}P_3$ for sufficiently large number $l$.
Therefore $R(\mathbb P_k^1,D)$ is $F$-rational by Theorem \ref{criterion} and Lemma \ref{deg le -2}(1).

\vskip.3truecm
\noindent
{\bf Case 4.}
We assume that  $s=2$ and $(a_1,a_2,a_3)$ is one of the followings:
\[
(\frac{1}{2},\frac{2}{3},\frac{2b+3}{3b+5}),\ \ \ (\frac{1}{2},\frac{2}{3},\frac{3b+5}{4b+7}).
\]
Then $D\ge 2P_0-\frac{1}{2}P_1-\frac{2}{3}P_2-\frac{3}{4}P_3$.
Therefore $R(\mathbb P_k^1,D)$ is $F$-rational by Theorem \ref{criterion} and Lemma \ref{deg le -2}(2).

\vskip.3truecm
\noindent
{\bf Case 5.}
We assume that  
$(s,a_1,a_2,a_3)=(2,\frac{1}{3},\frac{2}{3},\frac{n+2}{n+3}).$
Then $D\ge 2P_0-\frac{1}{3}P_1-\frac{2}{3}P_2-\frac{l}{l+1}P_3$ for sufficiently large number $l$.
Therefore $R(\mathbb P_k^1,D)$ is $F$-rational by Theorem \ref{criterion} and Lemma \ref{deg le -2}(4).

By the above discussion, if $p\ge 7$,
then $R$ is $F$-rational.
\end{proof}

\begin{thm}\label{F-rationality of RFP}
Let $R$ be a two-dimensional graded ring with $e(R)=4$ and a rational singularity.
If $p\ge 11$,
then $R$ is $F$-rational.
Furthermore, 
this inequality is best possible.
\end{thm}

\begin{proof}
Example \ref{ex1} shows that there exists  a two-dimensional non-$F$-rational graded ring $R$ with a rational singularity,  $e(R)=4$ and $p=7$.

From now on, we assume that $p\ge 11$.
By Theorem \ref{DP const},  Theorem \ref{graded rational singularity} and Theorem \ref{classification summarize},
there exists an ample $\mathbb Q$-divisor $D$ on $\mathbb P_k^1$   in the list of Theorem \ref{classification summarize}.(2) with 
$R\cong R(\mathbb P_k^1,D)$.
Let $D=sP_0-\sum_{i=1}^ra_iP_i$, where $s\in\mathbb N$, $0\le a_i<1$ and $P_i$ are distinct points of $\mathbb P_k^1$.
Let  $m,n,a,b,c,d$ be non-negative integers.

\vskip.3truecm
\noindent
{\bf Case 1.}
We assume that  
$(s,a_1,\dots,a_r)=(3,\frac{1}{2},\frac{1}{2},\frac{c}{c+1},\frac{d}{d+1}).$
We have 
$\mathrm{deg} [-lD]$ $\le -2$ for $l\in 2\mathbb N$ and $\mathrm{deg} [-lD]\le -3$ for $l\in \mathbb N\setminus 2\mathbb N$.
Then $R(\mathbb P_k^1,D)$ is $F$-rational by Theorem \ref{criterion}.

\vskip.3truecm
\noindent
{\bf Case 2.} 
We assume that  $s=2$ and $(a_1,a_2,a_3)$ is one of the followings:
\[(\frac{1}{2},\frac{5}{6},\frac{3}{5}), (\frac{1}{2},\frac{6}{7},\frac{3}{5}), (\frac{2}{3},\frac{4}{5},\frac{2}{5}),
(\frac{3}{4},\frac{3}{4},\frac{1}{3}),\]
\[ (\frac{3}{4},\frac{4}{5},\frac{1}{3}), (\frac{3}{4},\frac{5}{6},\frac{1}{3}),
(\frac{1}{2},\frac{5}{6},\frac{4}{7}), (\frac{1}{2},\frac{6}{7},\frac{4}{7}).\]
Then $R(\mathbb P_k^1,D)$ is $F$-rational by Theorem \ref{$F$-rationality when p not | d_i}.

\vskip.3truecm
\noindent
{\bf Case 3.}
We assume that  $(s,a_1,\dots,a_r)$ is one of the followings:
\[(4,\frac{a}{a+1},\frac{b}{b+1},\frac{c}{c+1},\frac{d}{d+1}),\ \ \   (2,0,\frac{n}{n+1},\frac{(2a+1)b+3a+1}{(2a+3)b+3a+4}),\]
\[(3,\frac{m}{m+1},\frac{n}{n+1},\frac{(a+1)b+2a+1}{(a+2)b+2a+3}),\]
\[(2,0,\frac{n}{n+1},\frac{\big((a+1)b + 3a+2\big)c+(2a+2)b+5a+3}{\big((a+2)b+ 3a+5\big)c+(2a+4)b+5a+8}),\] 
\[(2,0,\frac{(a+1)b+2a+1}{(a+2)b+2a+3},\frac{(c+1)d+2c+1}{(c+2)d+2c+3}).\]
Then $R(\mathbb P_k^1,D)$ is $F$-rational by Proposition \ref{f-rational deg 1}.(1).

\vskip.3truecm
In the rest of this proof, we  always assume that $s=2$ and $r=3$.
If necessary, we may reorder $(a_1,a_2,a_3)$.

\vskip.3truecm
\noindent
{\bf Case 4.}
We assume that  $(a_1,a_2,a_3)$ is one of the followings:
\[
(\frac{1}{2},\frac{2}{3},\frac{4b+11}{5b+14}),
(\frac{1}{2},\frac{2b+5}{3b+8},\frac{3}{4}),
(\frac{1}{2},\frac{2b+5}{3b+8},\frac{4}{5}),
(\frac{b+2}{2b+5},\frac{2}{3},\frac{2}{3}),\]
\[
(\frac{b+2}{2b+5},\frac{2}{3},\frac{3}{4}),
(\frac{1}{2},\frac{2}{3},\frac{3b+4}{5b+7}),
(\frac{1}{2},\frac{2}{3},\frac{5b+7}{7b+10}),
\]
\[
(\frac{1}{2},\frac{2}{3},\frac{7b+10}{9b+13}),
(\frac{1}{2},\frac{3b+4}{5b+7},\frac{3}{4}),
(\frac{1}{2},\frac{3b+4}{5b+7},\frac{4}{5}),\]
\[
(\frac{1}{2},\frac{2}{3},\frac{(2b + 5)c+4b+8}{(3b+ 8)c+6b+13}),
(\frac{1}{2},\frac{2}{3},\frac{(3b + 8)c+6b+13}{(4b+ 11)c+8b+18}),
\]
\[
(\frac{1}{2},\frac{2b+3}{3b+5},\frac{2d+3}{3d+5}),
(\frac{1}{2},\frac{3}{5},\frac{3d+5}{4d+7}),
(\frac{1}{2},\frac{3}{5},\frac{4d+7}{5d+9}).\]
Then $D\ge 2P_0-\frac{1}{2}P_1-\frac{2}{3}P_2-\frac{4}{5}P_3$.
Therefore $R(\mathbb P_k^1,D)$ is $F$-rational by Theorem \ref{criterion} and Lemma \ref{deg le -2}(3).

\vskip.3truecm
\noindent
{\bf Case 5.}
We assume that  $(a_1,a_2,a_3)$ is one of the followings:
\[
(\frac{1}{2},\frac{b+1}{3b+4},\frac{n+1}{n+2}),
(\frac{1}{2},\frac{1}{2},\frac{(2a+3)b+3a+4}{(2a+5)b+3a+7}),
(\frac{1}{2},\frac{(b +2)c+2b+3}{(2b+ 5)c+4b+8},\frac{n+1}{n+2}),\]
\[(\frac{1}{2},\frac{1}{2},\frac{\big((a+2)b + 3a+5\big)c+(2a+4)b+5a+8}{\big((a+3)b+ 3a+8\big)c+(2a+6)b+5a+13}),
\]
\[
(\frac{b+1}{2b+3},\frac{d+1}{2d+3},\frac{m+1}{m+2}),
(\frac{1}{2},\frac{b+1}{2b+3},\frac{(c+2)d+2c+3}{(c+3)d+2c+5}).
\]
Then $D\ge 2P_0-\frac{1}{2}P_1-\frac{1}{2}P_2-\frac{l}{l+1}P_3$ for sufficiently large number $l$.
Note that if $(a_1,a_2,a_3)=(\frac{b+1}{2b+3},\frac{d+1}{2d+3},\frac{m+1}{m+2})$, then we have 
$\mathrm{deg} [-tD]\le -3$ for $t\in 2\mathbb N$.
Therefore $R(\mathbb P_k^1,D)$ is $F$-rational by Theorem \ref{criterion} and Lemma \ref{deg le -2}(1).

\vskip.3truecm
\noindent
{\bf Case 6.}
We assume that  $(a_1,a_2,a_3)$ is one of the followings:
\[
(\frac{b+1}{3b+4},\frac{2}{3},\frac{n+2}{n+3}),
(\frac{1}{3},\frac{2}{3},\frac{(c+2)d+2c+3}{(c+3)d+2c+5}),
(\frac{1}{3},\frac{2d+3}{3d+5},\frac{m+3}{m+4}).\]
Then $D\ge 2P_0-\frac{1}{3}P_1-\frac{2}{3}P_2-\frac{l}{l+1}P_3$ for sufficiently large number $l$.
Therefore $R(\mathbb P_k^1,D)$ is $F$-rational by Theorem \ref{criterion} and Lemma \ref{deg le -2}(4).

\vskip.3truecm
\noindent
{\bf Case 7.}
We assume that  $(a_1,a_2,a_3)=(\frac{3}{4},\frac{n+3}{n+4},\frac{1}{4})$.
Then $R(\mathbb P_k^1,D)$ is $F$-rational by Theorem \ref{criterion} and Lemma \ref{deg le -2}(5).

By the above discussion, if $p\ge 11$,
then $R$ is $F$-rational.
\end{proof}


\end{document}